\documentclass[final,1p,times,authoryear]{elsarticle}
\usepackage{amsmath,amssymb,amsthm,amsfonts,bbm,wrapfig,stmaryrd,url,textcomp,wasysym}
\usepackage{enumerate}
\usepackage{graphicx}
\usepackage{enumitem}
\usepackage[all]{xy}
\usepackage[colorlinks,final,hyperindex]{hyperref}

\newtheorem{theorem}{Theorem}
\newtheorem{lemma}[theorem]{Lemma}

\newtheorem{proposition}[theorem]{Proposition}

\theoremstyle{definition}
\newtheorem{definition}[theorem]{Definition}

\newtheorem{example}[theorem]{Example}
\newtheorem{remark}[theorem]{Remark}

\renewenvironment{proof}{\noindent{\em Proof.}}{$\Box$~\\}
\newtheorem*{acknowledgment}{Acknowledgment}

\newcommand{\pospart}[1]{#1^+}
\newcommand{\negpart}[1]{#1^-}
\newcommand{\pcwext}[1]{\mathcal{P}#1}
\newcommand{\distmod}[1]{\mathcal{D}#1}
\newcommand{\hdistmod}[1]{\hat{\mathcal{D}}#1}
\newcommand{\slimmod}[1]{\hat{\mathcal{D}}#1}
\newcommand{\slimext}[1]{\hat{\mathcal{P}}#1}
\newcommand{\bivmod}[1]{\mathcal{D}_2#1}

\def\reg{\textsuperscript{\textregistered}}
\def\tm{\leavevmode\hbox{$\rm {}^{TM}$}}
\newcommand{\mma}{Mathematica$\reg$}
\newcommand{\mpl}{Maple\tm}

\newcommand{\NN}{\mathbb{N}}

\newcommand{\QQ}{\mathbb{Q}}
\newcommand{\RR}{\mathbb{R}}
\newcommand{\CC}{\mathbb{C}}

\newcommand{\der}{\partial}

\newcommand{\nglb}{\text{\o}}

\renewcommand{\ker}{\operatorname{Ker}}
\newcommand{\im}{\operatorname{Im}}

\def\<{\langle}
\def\>{\rangle}

\newcommand{\trans}[1]{
  \mathrel{\vbox{\offinterlineskip\ialign{%
    \hfil##\hfil\cr
    $\scriptscriptstyle*$\cr
    \noalign{\kern0.3ex}
    #1\cr
}}}}

\newcommand{\cum}{{\textstyle\varint}}

\newcommand{\ev}{\evl}

\newcommand{\sgn}{\operatorname{sgn}}

\newcommand{\nonzero}[1]{#1^\times}

\newcommand{\galg}{\mathcal{F}}

\def\ddiagup{\rlap{\raisebox{-0.4mm}{$\scriptscriptstyle\diagup$}}%
  \raisebox{0.4mm}{$\scriptscriptstyle\diagup$}}
\def\vcum{{\setbox0=%
    \hbox{$\textstyle{\scriptscriptstyle\diagup}{\varint}$}%
    \textstyle{\vcenter{\hbox{$\scriptscriptstyle\diagup$}}\kern-.5\wd0}%
    \!\varint}}
\def\wcum{{\setbox0=%
    \hbox{$\textstyle{\ddiagup}{\varint}$}%
    \textstyle{\vcenter{\hbox{$\ddiagup$}}\kern-.5\wd0}%
    \!\varint}}
\newcommand{\vder}{\eth}

\newcommand{\vcumd}{\vcum_{\!\mathcal M}}
\newcommand{\vderd}{\vder_{\mathcal M}}
\newcommand{\dirm}{\mathcal{M}}

\newcommand{\bivsh}[2]{S^{#1}_{#2}}
\newcommand{\bivmsh}[2]{\smash{\acute{\text{S}}}^{#1}_{#2}}
\newcommand{\bivmev}[2]{\operatorname{\acute{e}v}^{#1}_{#2}}

\newcommand{\bspc}{\mathcal{B}}

\newcommand{\trp}[1]{#1^\intercal}

\newcommand{\evl}{\text{\scshape\texttt e}}
\newcommand{\mevl}{\acute{\text{\scshape\texttt e}}}
\newcommand{\mev}{\mevl}
\newcommand{\mshift}{\acute{\text{S}}}
\newcommand{\psevl}{\hat{\evl}}

\newcommand{\Aut}{\mathrm{Aut}_K}

\newcommand{\intdiffop}{\galg_\Phi[\der,\cum]}
\newcommand{\diffop}{\galg[\der]}
\newcommand{\rbop}{\galg[\cum]}
\newcommand{\eqintdiffop}{\galg[\der,\cum_{\!\Phi}]}

\newcommand{\bvp}[2]{\boxed{\begin{array}{l}#1\\#2\end{array}}}

\let\phi\varphi
\let\epsilon\varepsilon
\let\mathbb\mathbbm

\begin{document}

\begin{frontmatter}

\title{An Integro-Differential Structure for Dirac Distributions}

\author{Markus Rosenkranz}
\address{Johannes Kepler University Linz, Austria}
\ead{marcus@rosenkranz.or.at}
\ead[url]{http://www.risc.jku.at/home/marcus}

\author{Nitin Serwa}
\address{University of Kent, Canterbury, United Kingdom}
\ead{ns512@kent.ac.uk}
\ead[url]{https://www.kent.ac.uk/smsas/our-people/profiles/serwa\_nitin.html}

\begin{abstract}
  We develop a new algebraic setting for treating piecewise functions and distributions together
  with suitable differential and Rota-Baxter structures. Our treatment aims to provide the algebraic
  underpinning for symbolic computation systems handling such objects. In particular, we show that
  the Green's function of regular boundary problems (for linear ordinary differential equations) can
  be expressed naturally in the new setting and that it is characterized by the corresponding
  distributional differential equation known from analysis.
\end{abstract}

\end{frontmatter}

\section{Introduction}
\label{sec:intro}

It is indisputable that differential algebra~\cite{Ritt1966,Kolchin1973}, differential Galois
theory~\cite{PutSinger2003} as well as various other approaches of Symbolic Analysis have made
outstanding contributions to the theory of differential equations~\cite{Seiler1997}. From their
particular algebraic-algorithmic vantage points, they provide powerful tools for describing and
analyzing the structure of solutions. Interestingly, the \emph{theory of distributions}---in modern
analysis the hard bedrock supporting the theory of linear (ordinary and partial) differential
equations---has received comparably little attention in Symbolic Analysis.

One reason for this is perhaps that the standard approach to distributions seems to be inherently
topological in nature; even the very definition of distributions involves the continuous dual of
certain carefully chosen function spaces~\cite{DuistermaatKolk2010}. Of course, such an objection
begs the question: Namely, how much \emph{algebraic structure} can one extract from the
Algebra-Analysis mixture at first encountered? Even differential algebra was in the same situation
before basic notions such as differential rings were introduced.

Based on the results of the present paper, we think the main obstacle to an algebraic treatment of
distributions is the widespread limitation of differential algebra to structures having only
derivations (differential rings / fields / algebras / modules). In a setting thus limited, one can
say little more than the following: ``A distribution like~$\delta_a$ has arbitrary formal
derivatives~$\delta_a', \delta_a'', \dots$''. In effect, one treats~$\delta_a$ as a differential
indeterminate. But the characteristic feature of the Dirac distribution~$\delta_a$ is of course that
it effects an evaluation at~$a$ when it ``appears under the integral''. For making this central idea
precise (in algebraic terms), one has to take recourse to the theory of \emph{Rota-Baxter
  algebras}~\cite{Guo2012}. By the same token, one must also provide an algebraic treatment of
evaluation; both of these are linked in the structure of integro-differential
algebras~\cite{RosenkranzRegensburger2008a} by a crucial relation~\eqref{eq:eval}.

Once integration enters the stage, it is clear that we should also consider \emph{piecewise
  functions} since the latter may be built up from the Heaviside function~$H_a(x) = H(x-a)$, which
is in turn the integral of the Dirac delta distribution~$\delta_a$. In fact, it is natural to start
with the theory of piecewise (smooth or continuous) functions since these may usefully be endowed
with a Rota-Baxter operator without involving distributions. One might even be tempted to build up
distributions via this route, simply adding a derivation that maps~$H_a$ to the Dirac
distribution~$\delta_a$. As we shall see subsequently (Remark~\ref{rem:no-diffring}), our actual
development must follow a slightly different route, though we shall indeed treat piecewise functions
before introducing Dirac distributions.

Assuming one can extract some ``algebraic substance'' from the theory of distributions, what does it
achieve? In particular, does it allow any symbolic computation for practically important
applications? We think the answer is yes, as we would like to demonstrate here: While the primary
purpose of this paper is to lay out the foundations, we do include an application section to sketch
one domain where our algebraic approach to distributions can be employed---\emph{boundary problems}
for linear ordinary differential equations (LODE). Here distributions come into their own: As we
shall see in detail, one may actually distinguish three different (though related) roles for the
algebraic Heaviside functions and Dirac distributions: (1) The Green's function of a regular
boundary problem is naturally a piecewise function (or even a proper distribution in the case of
ill-posed boundary problems). (2) It can be shown to satisfy a differential equation with the Dirac
distribution on its right-hand side. (3) The Green's operator may act on functions which are only
piecewise smooth.

Of course this does not exhaust the possible scope of applications. Eventually, \emph{computer
  algebra systems} like~\mpl\ and~\mma\ should be able to treat distributions and piecewise
functions much like any other ``functional'' terms. They should provide support for all crucial
operations on these objects, including many that we cannot address here (e.g.\@ convolution,
Fourier/Laplace transforms, composition). For practical applications, one often needs to be able to
use piecewise functions and distributions at suitable places in algebraic and differential equations
to be solved or simplified. We hope our approach will provide a convenient starting point for
further development in this direction.

\smallskip

\textbf{Structure of the paper.} In detail, we will develop the subject matter as follows. After
completing this Introduction by explaining some crucial notation, we briefly review the theory of
\emph{differential Rota-Baxter algebras and modules}, which form the basic algebraic framework for
the rest of the paper (Section~\ref{sec:intdiffalg}). In the next section we build up the algebra of
\emph{piecewise functions} and show that is a Rota-Baxter extension of the ground algebra
(Proposition~\ref{prop:pcw}), generalizing the familiar setting of piecewise smooth or piecewise
analytic functions (Examples~\ref{ex:pcw-smooth} and~\ref{ex:pcw-anl}). Then \emph{distributions}
are introduced as a differential Rota-Baxter module (Section~\ref{sec:distmod}), following an
independent route but such that the piecewise functions reappear as a Rota-Baxter subalgebra
(Theorem~\ref{thm:dist-drbm}). In fact, we shall see that the distributions even form an
integro-differential module (Proposition~\ref{prop:intdiff-mod}), that they can be characterized by
a natural universal property (Proposition~\ref{prop:dirac-module}), and that they inherit the shift
structure from the ground algebra (Theorem~\ref{thm:intdiff-mod-shifted}). We end the section by
exhibiting the filtration structure of the integro-differential module of distributions. The topic
of the next section is a---rather modest---species of \emph{bivariate distributions}
(Section~\ref{sec:biv-distributions}), though large enough to cater for the applications of the next
section. Just as the univariate distribution module contains the Rota-Baxter subalgebra of piecewise
functions, the bivariate version contains the bivariate piecewise functions as a subalgebra relative
to both Rota-Baxter structures (Proposition~\ref{prop:pcw-biv}). However, the main result of this
section is that the bivariate distribution module is a differential Rota-Baxter module with respect
to both differential Rota-Baxter structures, containing isomorphic copies of both univariate
distribution modules plus ``diagonal'' distributions (Theorem~\ref{thm:biv-distmod}). Equipped with
these tools, we turn to the aforementioned applications in the theory of LODE \emph{boundary
  problems} (Section~\ref{sec:appl}). Our first goal is to generalize the algorithm extracting
Green's functions from Green's operators given in~\cite{RosenkranzSerwa2015} to bivariate
distributions over ordinary shifted integro-differential algebras
(Theorem~\ref{thm:extr-corr}). Next we show that such a Green's function also satisfies an algebraic
version of the well-known distributional differential equation with~$\delta(x-\xi)$ on the
right-hand side (Theorem~\ref{thm:dist-bvp}). Finally, we confirm that the corresponding Green's
operator of an arbitrary well-posed boundary problem may actually be applied to piecewise functions
(Proposition~\ref{prop:pcw-forcing}). We conclude with some thoughts about future developments.

\smallskip

\textbf{Notation.} With the exception of the ring of integro-differential operators (to be
introduced in Section~\ref{sec:intdiffalg}), all rings and algebras in this paper are assumed to be
\emph{commutative} and---unless stated otherwise---also \emph{unitary}. Algebras are over a ground
ring~$K$ that will usually be a field (in fact an ordered field for most of the time). The set of
nonzero elements of~$K$ is denoted by~$\nonzero{K}$. We write $\Aut(\galg)$ for the group of
$K$-algebra automorphisms of an algebra~$\galg$. By a \emph{character} of~$\galg$ we mean an algebra
homomorphism~$\galg \to K$. If~$P$ and~$Q$ are any linear operators on~$\galg$, their
\emph{commutator} is denoted by~$[P, Q] := PQ-QP$. If~$S$ is a semigroup, we write~$\galg[S]$ for
the \emph{semigroup algebra} of~$S$ over~$\galg$, by which we mean the monoid
algebra~\cite[p.~104]{Lang2002} of the unitarization~$S \uplus \{ 1\}$.

If~$(\galg, \der)$ is a differential algebra, we write as usual~$f' := \der f$
and~$f^{(k)} = \der^k f$ for the derivatives of an element~$f \in \galg$. For a set of differential
indeterminates~$X$, the algebra of differential polynomials~$\galg\{X\}$ is the free object in the
category of differential $\galg$-algebras. Similarly, the $\galg$-submodule~$\galg\{X\}_1$
consisting of \emph{affine differential polynomials}, i.e.\@ those having total degree at most~$1$,
is the free object in the category of differential $\galg$-modules.

Since in Section~\ref{sec:pcwext} we will be dealing with $K$-algebras where $(K, <)$ is an
\emph{ordered field} (hence of characteristic zero), it is useful to introduce some notation for
ordered fields. We denote the \emph{minimum} and \emph{maximum} of two elements~$a, b \in K$
by~$a \sqcap b$ and~$a \sqcup b$, respectively. We agree that~$\sqcap, \sqcup$ have precedence
over~$+,-$. Furthermore, we shall write~$\pospart{a} := a \sqcup 0$ and~$\negpart{a} := a \sqcap 0$
for the \emph{positive and negative part} of~$a \in K$; then we have~$a = \pospart{a} + \negpart{a}$
and $|a| = \pospart{a} - \negpart{a}$. We observe that both~$(K, \sqcap)$ and~$(K, \sqcup)$ are
semigroups, which we denote by~$K_\sqcap$ and~$K_\sqcup$, respectively. We define the
\emph{Heaviside operator} $H\colon K \to K$ by
\begin{equation*}
  H(a) = \begin{cases}
    0 & \text{if $a<0$,}\\
    \eta & \text{if $a=0$,}\\
    1 & \text{if $a>0$}
    \end{cases}
\end{equation*}
for~$a \in K$. The appropriate choice of~$\eta \in K$ is somewhat subtle; we will repeatedly come
back to this point. From the analytic point of view, we might think of the
mapping $\RR \to \RR, a \mapsto H(a)$ as a (representative of an) $L^2$
function\footnote{\label{fn:heav-val}These remarks are purely motivational. It is important to
  distinguish $H(a)$ from $H_a$, which we shall introduce below, in Definition~\ref{def:pcwext}, as
  the actual algebraic model of the \emph{Heaviside function}~$x \mapsto H(x-a)$.}, and then the
choice of~$\eta \in \RR$ is of course immaterial. For the algebraic treatment, however, we will
distinguish three more or less natural possibilities (the terminology is again motivated by the case
$K = \RR$):
\begin{itemize}
\item The \emph{left continuous} convention uses~$\eta=0$.
\item In contrast, the \emph{right continuous} choice is to put~$\eta=1$.
\item Finally, the \emph{symmetric} setting~$\eta=1/2$ is essentially the sign function in the sense
 that one has~$\sgn(a) = 2 H(a) - 1$. It is neither left nor right continuous.
\end{itemize}
In this paper, we use the left-continuous convention~$\eta=0$, but we will discuss the other
possibilities as we develop the corresponding algebraic structures. For convenience we
shall use also~$\bar{H}(x) := 1-H(x)$ for the \emph{dual Heaviside operator}.

\section{Differential Rota-Baxter Algebras and Modules}
\label{sec:intdiffalg}

Just as a differential algebra~$(\galg, \der)$ encodes the essence of derivatives, the basic
algebraic structure for encoding integration is a \emph{Rota-Baxter algebra} $(\galg, \cum)$,
meaning a $K$-algebra with a $K$-linear operator~$\cum\colon \galg \to \galg$ satisfying the
Rota-Baxter axiom
\begin{equation}
  \label{eq:rb-axiom}
  \cum f \cdot \cum g = \cum f \cum g + \cum g \cum f ,
\end{equation}
which we also call the \emph{weak} Rota-Baxter axiom in view of an important generalization that we
shall explain soon. At this juncture we should point out our \emph{parenthesis convention} for
nested Rota-Baxter operators:\footnote{Note also that here and henceforth we use \emph{operator
    notation} for~$\der$ and~$\cum$, as it is common in analysis. So the Leibniz rule
  is~$\der \, fg = f \der g + g \der f$ rather than~$d(fg) = f \, d(g) + g \, d(f)$ when using
  functional notation with~$d$. In the same vein, \eqref{eq:rb-axiom} would
  be~$P(f) \, P(g) = P(f \, P(g)) + P(g \, P(f))$ in functional notation with~$P$.} The scope
of~$\cum$ extends across all implicit products (denoted by juxtaposition), terminated by $\cdot$ as
on the left-hand side of~\eqref{eq:rb-axiom}. While this saves a host of parentheses, one must be
careful to distinguish~$\cum f \, \cum g$ and~$\cum f \cdot \cum g$.

It is often necessary to combine differential and Rota-Baxter structures, especially for application
areas like boundary problems, but also for the algebraic theory of distributions that we are about
to build up. There are two important ways of coupling the two structures. The weaker one is called a
\emph{differential Rota-Baxter algebra} $(\galg, \der, \cum)$ in~\cite{GuoKeigher2008}; by
definition this is a differential algebra~$(\galg, \der)$ and a Rota-Baxter algebra~$(\galg, \cum)$
such that the Rota-Baxter operator~$\cum$ is a section of the derivation~$\der$. Thus the
differential and Rota-Baxter structures are only coupled by the so-called \emph{section axiom}
$\der \circ \cum = 1_\galg$.

In many cases the coupling is stronger: We call~$(\galg, \cum, \der)$ an \emph{integro-differential
  algebra} if~$(\galg, \der)$ is a differential algebra and~$\cum\colon \galg \to \galg$ is a
$K$-linear operator satisfying the \emph{strong}\footnote{In~\cite{RosenkranzRegensburger2008a} we
  have called it the ``differential Rota-Baxter axiom''. In the present context we prefer to avoid
  this terminology as it might be misconstrued as characterizing differential Rota-Baxter algebras.}
\emph{Rota-Baxter axiom}~\cite[Eqn.~(6)]{RosenkranzRegensburger2008a}, namely
\begin{equation}
  \label{eq:str-rb-axiom}
  f \cum g = \cum fg + \cum f' \cum g .
\end{equation}
This terminology stems from the fact that the weak axiom~\eqref{eq:rb-axiom} is a consequence (just
replace~$f$ by~$\cum f$ in the strong axiom and use the section axiom) while there are differential
Rota-Baxter algebras that are not integro-differential algebras. The first such example was found by
G.~Regensburger, using a quotient of a polynomial ring~\cite[Ex.~3]{RosenkranzRegensburger2008a}. In
fact, we shall soon encounter a natural example from analysis, namely piecewise smooth functions
``interpreted in the $L^2$ style'' (Proposition~\ref{prop:pcw-drb}).

There are various equivalent characterizations of the difference between differential Rota-Baxter
and integro-differential algebras~\cite[Thm.~2.5]{GuoRegensburgerRosenkranz2012}, for example
that~$\im{\cum} \subset \galg$ is an ideal rather than a subalgebra, or that~$\cum$ is linear not just
over~$K$ but over~$\ker{\der}$. One reformulation that is important here involves the so-called
\emph{induced evaluation}
\begin{equation}
  \label{eq:eval}
  \evl := 1_\galg - \cum \der ,
\end{equation}
which is a just projector onto~$K$ along~$\im{\cum}$ for a general differential Rota-Baxter algebra
but moreover multiplicative for an integro-differential algebra.

We call an (integro-)differential algebra \emph{ordinary} if~$\ker{\der} = K$. In that case, $\evl$
is a linear functional for a general differential Rota-Baxter algebra and a character for an
integro-differential algebra. We will usually start from an ordinary integro-differential
algebra~$(\galg, \der, \cum)$. In fact, ordinary differential Rota-Baxter algebras are automatically
integro-differential (since then linearity over~$K$ is actually over~$\ker{\der}$).

The notion of evaluation is crucial for the algebraic theory of integration. For certain purposes
(cf.\@ Definition~\ref{def:shifted-evaluations}), it will thus be useful to extend it to the more
general setting of a plain Rota-Baxter algebra~$(\galg, \cum)$, where an \emph{evaluation} is any
character~$\evl\colon \galg \to K$ with~$\evl \cum = 0$. This generalizes also the case of so-called
\emph{ordinary Rota-Baxter algebras}~\cite{RosenkranzGaoGuo2015}, defined as Rota-Baxter
algebras~$(\galg, \cum)$ where~$\cum\colon \galg \to \galg$ is injective
and~$\im(\cum) \dotplus K = \galg$; the projector~$\evl$ onto~$K$ along~$\im(\cum)$ is then a
distinguished evaluation. As noted in~\cite{RosenkranzGaoGuo2015}, each ordinary Rota-Baxter algebra
corresponds to a unique integro-differential algebra~$(\galg, \der, \cum)$ such that~\eqref{eq:eval}
holds; $(\galg, \der, \cum)$ is thus ordinary in the usual sense of~$\ker{\der} = K$.

We think of the evaluation as evaluating at a certain point~$o$, namely the (implicit)
initialization point of the Rota-Baxter operator~$\cum = \cum_{\!o}^x$. While this is only
suggestive notation, we can consider an arbitrary character~$\phi\colon \galg \to K$ and turn the
given Rota-Baxter operator~$\cum$ into a new one~$\cum_{\!\phi} := (1-\phi) \cum$ that we call
\emph{initialized at~$\phi$}. Its \emph{initialized function space}~$\im{\cum_{\!\phi}}$ is given
by~$\ker{\phi}$, the functions ``vanishing at~$\phi$''. This may be viewed as an algebraic
description of integrals~$\cum_{\!\phi}^x$ from various fixed initialization points~$\phi$ to the
variable upper bound~$x$. We will have a more rigid connection when we construct piecewise functions
via shift maps (Definition~\ref{def:shifted-evaluations}), labeling the characters~$\phi = \ev_c$
by points~$c \in K$. In this context we will often use~$f(c)$ as a suggestive shorthand
for~$\ev_c(f)$, likewise~$\cum_{\!c}$ for~$\cum_{\!\ev_c}$.

\begin{example}
  \label{ex:std-analysis}
  The \emph{standard example} from analysis is~$\galg = C^\infty(\RR)$ with derivation
  $\der \, f(x) = df/dx$ and the Rota-Baxter operator~$\cum f(x) = \cum_{\!0}^x f(x) \, dx$.  Here
  the initialized functions~$f(x)$ are those with~$f(0) = 0$, corresponding to the
  evaluation~$\evl(f) = f(0)$. Any other evaluation~$\ev_c(f) := f(c)$ may be used for generating
  additional Rota-Baxter operators~$\cum_{\!c} \, f = \cum_{\!c}^x \, f(x) \, dx$.
\end{example}

Within this paper we cannot review the theory and algorithms for \emph{linear boundary problems}
over an ordinary integro-differential algebra~$(\galg, \der, \cum)$; let us refer the reader
to~\cite{RosenkranzRegensburger2008a}. Here we just recall that, given a collection~$\Phi$ of
characters~$\galg \to K$, one constructs the \emph{ring of integro-differential
  operators}~$\intdiffop$ with canonical direct decomposition
\begin{equation*}
  \intdiffop = \diffop \dotplus \rbop \dotplus (\Phi)
\end{equation*}
as $K$-vector spaces; here~$\diffop$ is the usual ring of differential operators and~$\rbop$ the
corresponding (nonunitary) ring of integral operators (generated over~$\galg$ by $\cum$)
while~$(\Phi)$ is the two-sided ideal generated by the character set~$\Phi$. The latter may be
characterized as left $\galg$-module generated by \emph{Stieltjes conditions}, defined as the
\emph{right} ideal~$\Phi \cdot \intdiffop$. In the standard example above, these are arbitrary
linear combinations of local conditions (derivative evaluations of any order) and global conditions
(definite integrals with premultiplied weighting functions).

A \emph{boundary problem} is a pair~$(T, \bspc)$ consisting of a monic differential
operator~$T \in \diffop$ of order~$n$ and a boundary space~$\bspc \subset \galg^*$ spanned by~$n$
linearly independent Stieltjes conditions~$\beta_1, \dots, \beta_n$. We call~$(T, \bspc)$ regular
iff~$\ker{T} \dotplus \bspc^\perp = \galg$, where the \emph{orthogonal} is defined as the admissible
function space~$\bspc^\perp := \{ f \in \galg \mid \forall_{\beta \in \bspc} \: \beta(f) = 0 \}$.
Regularity of~$(T, \bspc)$ is equivalent to the classical stipulation: There is exactly one
solution~$u \in \galg$ of
\begin{equation}
  \label{eq:bvp}
  \bvp{Tu = f, }{\beta(u) = 0 \quad (\beta \in \bspc)}
\end{equation}
for every forcing function~$f \in \galg$. Having a fundamental system~$u_1, \dots, u_n$ of the
homogeneous system, meaning a $K$-basis of~$\ker{T}$, this may be checked algorithmically: The
regularity of~$(T, \bspc)$ is equivalent to the regularity of the \emph{evaluation matrix} $\beta(u)
\in K^{n \times n}$ formed by evaluating each~$\beta_i \in \bspc$ on each~$u_j \in \ker{T}$.

The \emph{Green's operator}~$G$ of a regular boundary problem~$(T, \bspc)$ is characterized by the
relations~$TG = 1_\galg$ and~$\im{G} = \bspc^\perp$; it is the map~$f \mapsto u$ for~\eqref{eq:bvp}
and may be computed as an element of the operator ring~$G \in \intdiffop$. Using the natural action
of~$\intdiffop$ on~$\galg$, one may check that~$u := Gf$ actually satisfies~\eqref{eq:bvp}.

The above ring-theoretic notions (differential algebra, Rota-Baxter algebra, differential
Rota-Baxter algebra, integro-differential algebra) all have natural module-theoretic analogs. For
example, a \emph{differential Rota-Baxter module}~$(M, \vder, \vcum)$ over a differential
Rota-Baxter algebra~$(\galg, \der, \cum)$ consists of a derivation~$\vder\colon M \to M$ in the
sense that~$\vder \, f \phi = (\der f) \, \phi + f \, \vder \phi$ for~$f \in \galg$
and~$\phi \in M$, and a Rota-Baxter operator~$\vcum\colon M \to M$ characterized by the (weak)
Rota-Baxter axiom
\begin{equation*}
  \cum f \cdot \vcum \phi = \cum f \, \vcum \phi + \vcum (\cum f) \phi
\end{equation*}
for~$f \in \galg$ and~$\phi \in M$; confer also~\cite[Ex.~3.7(b)]{GaoGuoRosenkranz2015a}. It is now
also clear what one means by a Rota-Baxter module. The notion of \emph{integro-differential module},
however, is slightly more subtle since we must now distinguish the strong Rota-Baxter
axiom~\eqref{eq:str-rb-axiom} for coefficients and the one for module elements; we shall postpone
this discussion to later when it is needed (Lemma~\ref{lem:char-intdiff-mod}). For now let us just
agree to call~$(M, \vder, \vcum)$ \emph{ordinary} iff~$\ker{\vder} = K$.

When dealing with bivariate distributions, we shall come across algebras with two distinct
differential and/or Rota-Baxter structures. In such a case we shall speak of \emph{duplex}
structures. For example, a duplex Rota-Baxter
algebra~$(\galg_2, \der_x, \der_\xi, \cum^x, \cum^\xi)$ is characterized by requiring
both~$(\galg_2, \der_x, \cum^\xi)$ and~$(\galg_2, \der_\xi, \cum^x)$ to be Rota-Baxter algebras. (It
should be noted that from our algebraic viewpoint $\der_x, \der_\xi$ and~$\cum^x, \cum^\xi$ are
just two pairs of derivations and Rota-Baxter operators that we might as well call~$d, e$ and~$P, Q$
should we care to.)

\section{The Piecewise Extension}
\label{sec:pcwext}

The passage from smooth functions $C^\infty(\RR)$ to piecewise smooth functions $PC^\infty(\RR)$ can
be achieved by adding characteristic functions for all intervals $[a,b] \subset \RR$, and these can
in turn be generated by the well-known \emph{Heaviside function} $H(x) \in PC^\infty(\RR)$ as
$1_{[a,b]}(x) = H(x-a) \, H(b-x)$. We shall come back to this motivating instance
(Example~\ref{ex:pcw-smooth}).

Our present goal is to describe the passage from a suitable integro-differential algebra $\galg$ to
its \emph{piecewise extension} $\pcwext{\galg}$ in an abstract algebraic manner. As we have just
seen, it is sufficient to adjoin algebraic Heaviside functions to $\galg$. These can be defined in a
natural way if the ground field~$K$ of the given integro-differential algebra $\galg$ is an ordered
field\footnote{Distributions are usually defined as generalizations of functions of a \emph{real}
  variable, meaning either $\RR^n \to \RR$ or $\RR^n \to \CC$. The case of a complex variable
  $\CC^n \to \CC$ is effectively treated as $\RR^{2n} \to \CC$, ignoring the field structure of
  $\CC \cong \RR \times \RR$. Starting from an ordered field thus seems plausible.}; in classical
analysis this is of course~$K = \RR$.  For the algebraic construction it is sufficient to employ the
semigroup algebra of~$K_\sqcup$ over~$\galg$.

\begin{definition}
\label{def:pcwext}
Let $\galg$ be an algebra over an ordered ring~$(K, <)$. Then we define
its \emph{piecewise extension} as $\pcwext{\galg}:=\galg [K_{\sqcup}]$.
\end{definition}

We denote the identity element of $\pcwext{\galg}$ by~$1$ and the other generators
by~$H_a \: (a \in K)$.  Then $\pcwext{\galg}$ can be viewed as the quotient of the polynomial
ring~$\galg[H_a \mid a \in K]$ modulo the ideal generated by the relations
$H_a H_b - H_{a \sqcup b} \: (a, b \in K)$. Moreover, linearity of the order on~$K$ gives rise to
the \emph{exchange law}~$H_{a \sqcup b} + H_{a \sqcap b} = H_a + H_b \: (a,b \in K)$, which implies
in turn that the piecewise extension $\pcwext{\galg} = \galg [K_{\sqcup}]$ is isomorphic to its
\emph{dual}~$\galg [K_{\sqcap}]$ via $H_a \mapsto 1-\bar{H}_a$, where the $\bar{H}_a$ denote the
generators of the dual. We will restrict ourselves to~$\pcwext{\galg} = \galg [K_{\sqcup}]$,
using~$\bar{H}_a := 1-H_a \in \galg [K_{\sqcup}]$ as shorthand notation. Introducing the alternative
notation~$H(x-a) :=H_a$ and~$H(a-x) :=\bar{H}_a$, the above relations entail
\begin{align*}
  H(x-a) \, H(x-b) &= H(x-a \sqcup b),\\
  H(a-x) \, H(b-x) &= H(x-a \sqcap b),\\
  H(a-x) \, H(x-b) &= 0 \quad \text{if~$a < b$}.
\end{align*}
In the classical setting (Example~\ref{ex:std-analysis}), this provides a faithful model of the
(rising and falling) Heaviside functions based at various points~$a,b \in \RR$. We will elaborate on
the analysis setting in due course (Example~\ref{ex:pcw-smooth}).

Nevertheless, one may wonder if there is any intrinsically \emph{algebraic characterization}. One
possibility is this: Call an algebra~$\galg$ over an ordered ring~$(K, <)$ \emph{order-related} if
it encodes the order of the ground ring within its multiplicative structure, i.e.\@ if there exists
a monoid embedding $H \colon (K, {\sqcup}) \hookrightarrow (\galg, \cdot)$ so that
$H_a H_b=H_{a{\sqcup}b} \: (a,b \in K)$. An order-related morphism between order-related rings is an
algebra homomorphism~$\zeta\colon \galg \to \tilde{\galg}$ such
that~$\zeta(H_a) = \tilde{H}_a \: (a \in K)$. Then the piecewise extension $\pcwext{\galg}$ can be
characterized as \emph{universal order-related extension algebra} of $\galg$, meaning every
embedding~$\galg \hookrightarrow A$ into an order-related algebra~$A$ factors through the algebra
embedding $\galg \hookrightarrow \pcwext{\galg}$ via a unique order-related
morphism~$\pcwext{\galg} \to A$. The verification is straightforward.

In order to introduce a Rota-Baxter operator $\cum\colon \pcwext{\galg} \to \pcwext{\galg}$
encapsulating the integration of piecewise continuous---in particular: piecewise smooth---functions,
we need a notion of algebraic domain with \emph{multiple evaluation points} (intuitively this is
because integrating against a step function based at~$a \in K=\RR$ amounts to starting off the
integral at~$a$, with integration constant induced by evaluation at~$a$). One way to make this
precise is in terms of a \emph{shifted Rota-Baxter algebra}: using an action of the additive group
of the ground field for shifting evaluation to arbitrary field points.

\begin{definition}
  \label{def:shifted-evaluations}
  By a \emph{shift map} on an algebra~$\galg$ we mean a group homomorphism
  $$S\colon (K, +) \to \big( \Aut(\galg), \circ \big), \qquad\text{written } S_a \, f = f(x+a) \quad\text{for $a \in K$, $f \in \galg$}.$$
  If Rota-Baxter/derivation operators are present, we require compatibility conditions:
  \begin{enumerate}
  \item We call~$(\galg, \cum, S)$ a \emph{shifted Rota-Baxter algebra} if~$S$ is a shift map on a Rota-Baxter algebra~$(\galg, \cum)$ with evaluation~$\evl$ such that~$[S_c, \cum] = \ev_c \cum$ for all~$c \in K$, where~$\ev_c := \evl \circ S_c$ is called the evaluation at~$c$.
  \item We call~$(\galg, \der, S)$ a \emph{shifted differential algebra} if~$S$ is a shift map on a differential algebra~$(\galg, \der)$ such that~$[S_c, \der] = 0$ for all~$c \in K$.
  \item We call~$(\galg, \der, \cum, S)$ a \emph{shifted differential Rota-Baxter algebra} if~$(\galg, \der, \cum)$ is a differential Rota-Baxter algebra such that both~$(\galg, \cum, S)$ and~$(\galg, \der, S)$ are shifted.
  \end{enumerate}
  In the sequel, we suppress the shift map~$S$ when referring to structures such as~$(\galg, \der, \cum, S)$.
\end{definition}

The most important \emph{examples} are of course the Rota-Baxter
algebra~$\big( C(\RR), \cum_{\!0}^x \big)$ and the integro-differential
algebra~$\big( C^\infty(\RR), \cum_{\!0}^x, \tfrac{d}{dx} \big)$, both with the shift
map~$f(x) \mapsto f(x+a)$ and the corresponding evaluations~$\ev_c f(x) = f(c)$.

In a shifted Rota-Baxter algebra~$(\galg, \cum)$, all evaluations~$\ev_c\colon \galg \to K$ are
characters but~$\ev_0 = \evl$ is distinguished\footnote{The formulation in terms of a distinguished
  character~$\evl$ is practical for applications. A more symmetric formulation would be to use the
  \emph{equitable setup} described in~\cite{RosenkranzSerwa2015}, where one starts from a whole
  family of ordinary Rota-Baxter operators~$\cum_{\!a}\colon \galg \to \galg$ whose induced
  evaluations~$\ev_a$ are required to satisfy the general shift
  relations~$\smash{[S_c, \cum_{\!a}] = \cum_{\!a}^{a+c}}$ for all~$a, c \in K$, where the
  right-hand integral is defined as above. In the asymmetric setup used here, these relations can be
  derived by a straightforward calculation.} by annihilating the given Rota-Baxter
operator~$\cum$. Using the evaluations, we can introduce \emph{shifted Rota-Baxter
  operators}~$\smash{\cum_{\!c}}\colon \galg \to \galg$ and \emph{definite
  integrals}~$\smash{\cum_{\!c}^d}\colon \galg \to K$ by
\begin{equation*}
  \cum_{\!c} := (1-\ev_c) \, \cum \qquad\text{and}\qquad \cum_{\!c}^d:=\ev_d\cum_{\!c}.
\end{equation*}
One checks immediately that~$\cum_{\!c} = S_{-c} \, \cum S_c$
and~$\cum_{\!c}^d = \cum_{\!c}-\cum_{\!d}$ are equivalent definitions. Obviously, each
$(\galg, \cum_c)$ is a Rota-Baxter algebra with evaluation~$\ev_c$ for~$c \in K$.

Let us now return to the task of defining the Rota-Baxter operator on~$\pcwext{\galg}$, assuming a
shifted Rota-Baxter algebra~$(\galg, \cum)$. Note first that every
element~$\zeta \in \pcwext{\galg}$ can be written uniquely as
\begin{equation}
  \label{eq:pcw-exp}
  \zeta = f + \sum_{a \in K} f_a H_a \qquad (f, f_a \in \galg)
\end{equation}
with almost all~$f_a$ zero. Hence it suffices to define~$\cum\colon \pcwext{\galg} \to \pcwext{\galg}$ as the unique extension
of~$\cum\colon \galg \to \galg$ such that
\begin{equation}
\label{eq:pcw-def1}
  \cum f H_a = (\cum_{\!\pospart{a}} f) \, H_a-(\cum_{\!\negpart{a}}^0 f)\, \bar{H}_a = (\cum_{\!a}
  \, f) \, H_a + \bar{H}(a) \, \cum_{\!0}^a \, f
\end{equation}
for all~$f \in \galg$ and~$a \in K$. For the sake of symmetry, let us also note that then
\begin{equation}
\label{eq:pcw-def2}
  \cum f \bar{H}_a = (\cum_{\!\negpart{a}} f) \, \bar{H}_a-(\cum_{\!\pospart{a}}^0 f)\, H_a
  = (\cum_{\!a} \, f) \, \bar{H}_a + H(a) \, \cum_{\!0}^a \, f.
\end{equation}
In fact, \eqref{eq:pcw-def1} and \eqref{eq:pcw-def2} are equivalent.\footnote{\label{fn:init-pt}Note
  that the choice of the splitting point~$0 \in K$ in~\eqref{eq:pcw-def1}--\eqref{eq:pcw-def2} is to
  some extent arbitrary. Any other point of~$K$ would yield the same
  operator~$\cum\colon \pcwext{\galg} \to \pcwext{\galg}$; in particular one could also choose the
  initialization point of the given Rota-Baxter operator~$\cum$ of~$\galg$. Here we have picked
  out~$0 \in K$ for convenience.}

\begin{figure}[ht]
  \centering
  \includegraphics[width=0.45\textwidth]{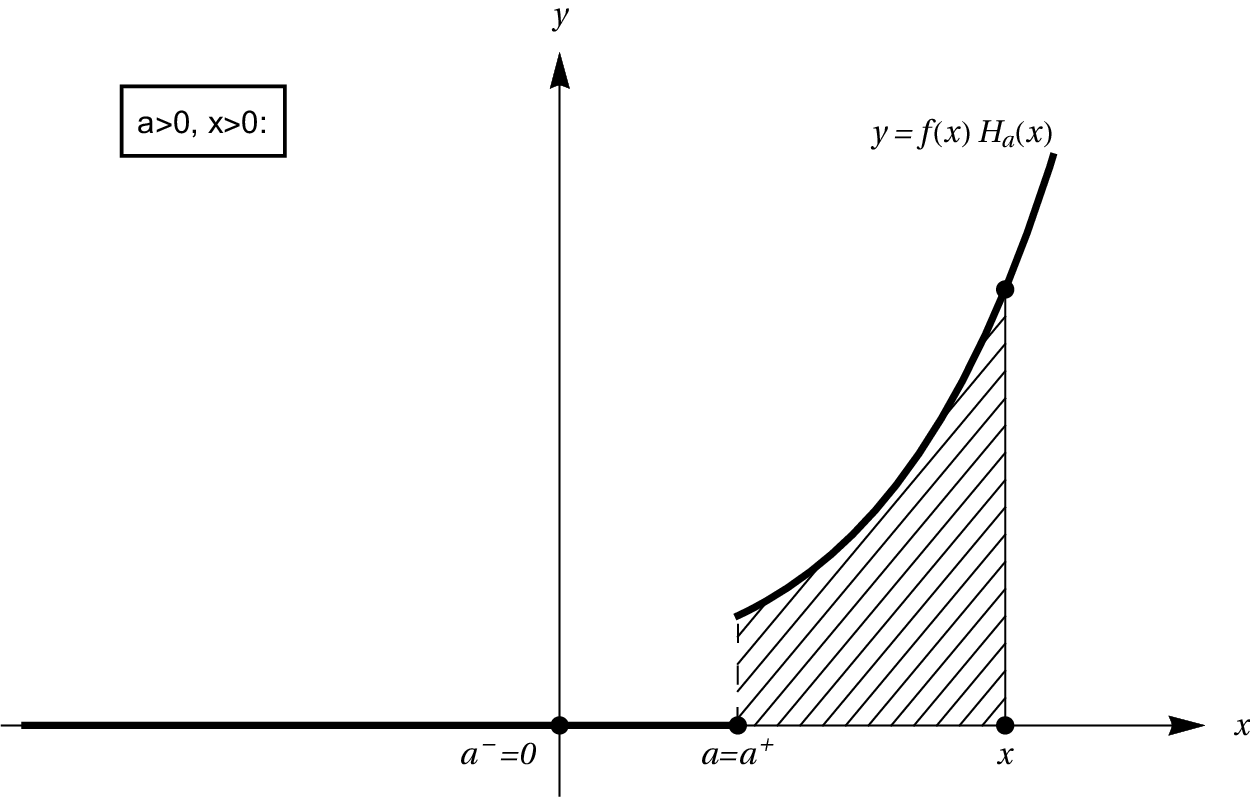}\hspace{0.025\textwidth}%
  \includegraphics[width=0.45\textwidth]{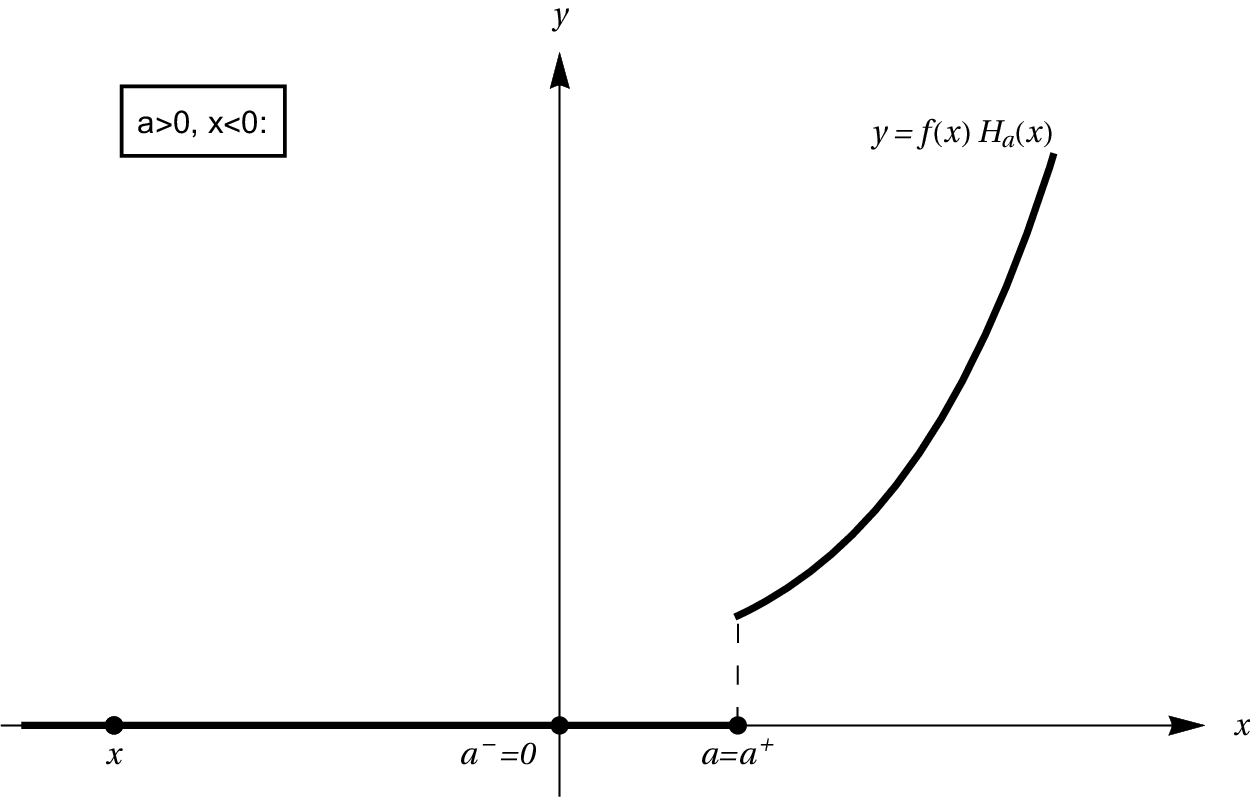}\\[2em]
  \includegraphics[width=0.45\textwidth]{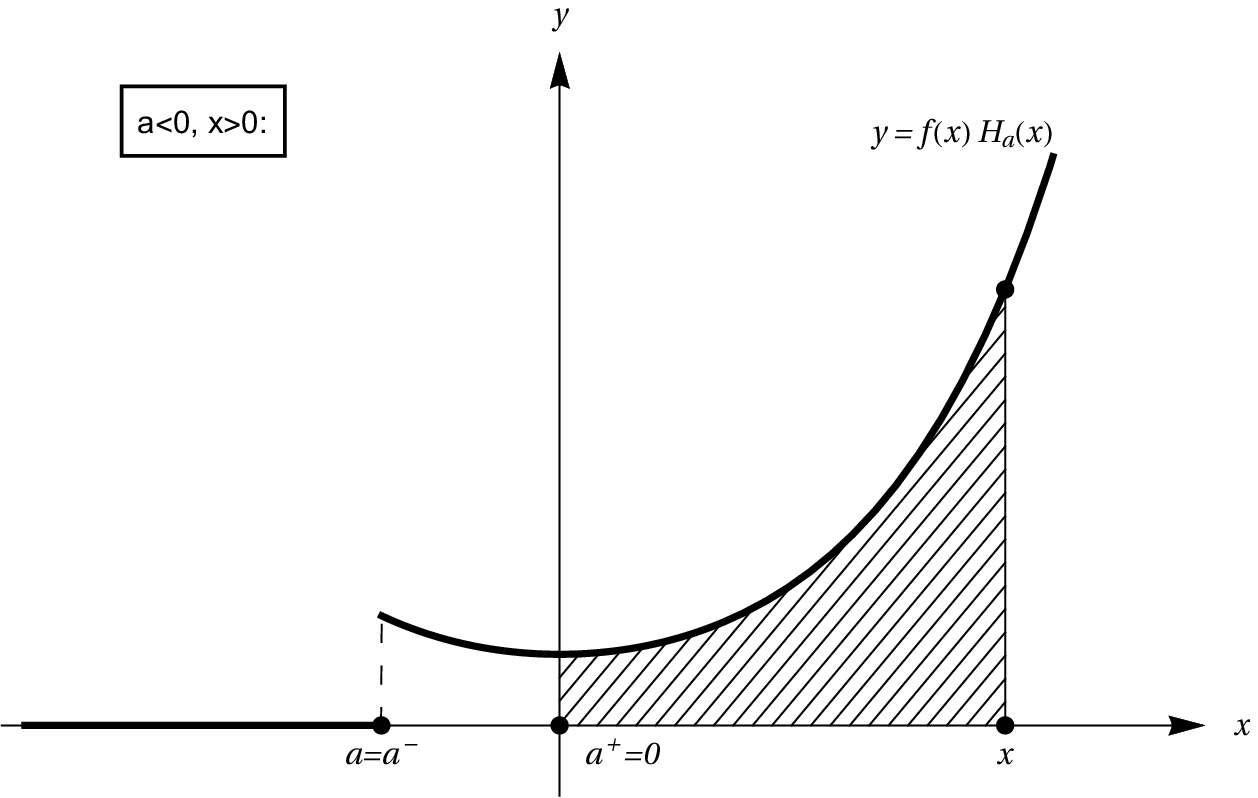}\hspace{0.025\textwidth}%
  \includegraphics[width=0.45\textwidth]{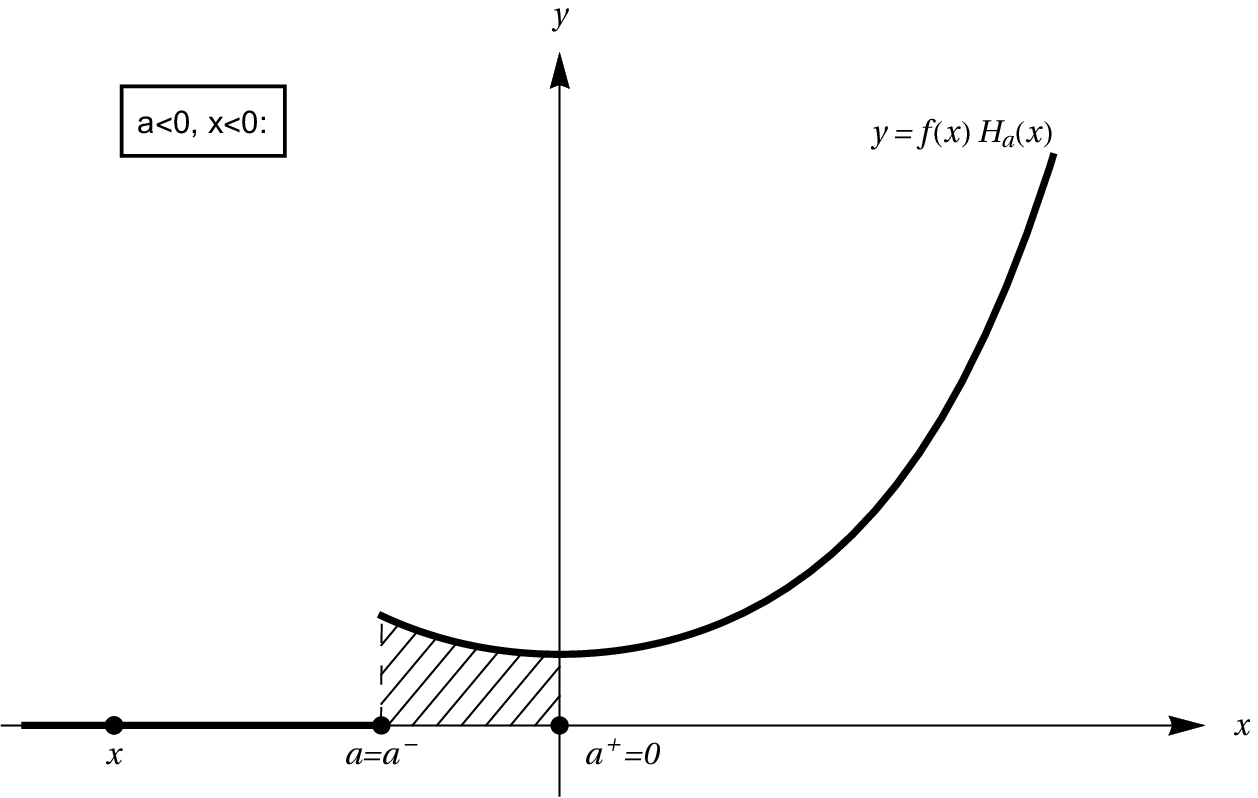}
  \caption{Integrating piecewise functions}
  \label{fig:pcw-int}
\end{figure}

The motivation for definition~\eqref{eq:pcw-def1} comes from the standard
example~$\galg = C^\infty(\RR)$ where it reproduces the usual \emph{Riemann
  integral}~$\cum = \cum_{\!0}^x$.  This is illustrated in Figure~\ref{fig:pcw-int}, where we have
visualized~$\cum f H_a$ with~$f(x) = \cosh{x}$ in the four different cases corresponding to the
signs of~$x$ and~$a$. While both forms of~\eqref{eq:pcw-def1} are obvious from the figure, their
identity is a general fact of ordinary shifted integro-differential algebras~$\galg$ as one can see
by a straightforward calculation using the generic relations~\eqref{eq:gen-rel} mentioned below.

In a similar way we can also define the \emph{shifts}~$S_a\colon \pcwext{\galg} \to
\pcwext{\galg}$.
Since in the standard examples, $f(x) \mapsto f(x+a)$ shifts the graph of~$f$ by $a$ units to the
left, we are led to~$S_a(H_b) := H_{b-a}$. This fixes~$S_a\colon \pcwext{\galg} \to \pcwext{\galg}$
in view of~\eqref{eq:pcw-exp} by requiring it to be an algebra homomorphism extending the given
shifts~$S_a\colon \galg \to \galg$. Obviously, the group law~$S_a \circ S_b = S_{a+b}$ is satisfied.

Finally, we define $\evl\colon \pcwext{\galg} \to K$ as the unique character
extending~$\evl\colon \galg \to K$ by the equi\-valent stipulations $\evl(H_a) = \bar{H}(a)$ or
$\evl(\bar{H}_a) = H(a)$. Using again $\ev_c := \evl \circ S_c\colon \pcwext{\galg} \to K$ as
shorthand notation, we have also~$\ev_c(H_a) = \bar{H}(a-c)$ and~$\ev_c(\bar{H}_a) = H(a-c)$. At
this point it should be noted that we could also use the right continuous convention for the
Heaviside operator~$H(a)$ but not the symmetric one: Indeed, the relation~$H_a H_b = H_{a \sqcup b}$
implies $\bar{H}(a) \, \bar{H}(b) = \bar{H}(a \sqcup b)$, which is satisfied by all three
conventions if~$(a,b) \neq (0,0)$; but the remaining case~$\bar{H}(0)^2 = \bar{H}(0)$ entails
$\bar{H}(0) \in \{ 0, 1 \}$ and thus rules out the symmetric convention.\footnote{We can only see
  one apparent advantage of the symmetric convention, in trying to build up a derivation with a
  Leibniz rule for Heavisides---but even this is ultimately doomed to fail: see
  Remark~\ref{rem:no-diffring}.} It is easy to check that~$\evl$ is indeed an \emph{evaluation}
on~$(\pcwext{\galg}, \cum)$.

\begin{proposition}
  \label{prop:pcw}
  Let~$(\galg, \cum)$ be an ordinary shifted Rota-Baxter algebra over an ordered field~$K$. Then~$(\pcwext{\galg}, \cum)$ 
  is a shifted Rota-Baxter algebra extending~$(\galg, \cum)$.
\end{proposition}

\begin{proof}
 Since~$K$ is of characteristic zero, by the
  polarization identity, it suffices to prove
  \begin{equation}
    \label{eq:bxt-pol}
    (\cum f H_a)^2 = 2 \cum f H_a \cum f H_a
  \end{equation}
  for~$f \in \galg$ and~$a \in K$. Since~$H_a$ is idempotent, the definition
  of~$\cum\colon \pcwext{\galg} \to \pcwext{\galg}$ and the Rota-Baxter axioms
  of~$\cum_{\!\pospart{a}}$ and~$\cum_ {\!\negpart{a}}$
  give~$2 \, (\cum_{\!\pospart{a}} f\cum_{\!\pospart{a}} f) \, H_a + 2 \, (\cum_{\!\negpart{a}}^0 f
  \cum_{\!\negpart{a}} f) \,\bar{H_a}$
  for the left-hand side of~\eqref{eq:bxt-pol}. Likewise, we get
  $2 \, (\cum_{\!\pospart{a}} f\cum_{\!\pospart{a}} f) \, H_a-2 \, (\cum_{\!\negpart{a}}^0 f
  \cum_{\!\pospart{a}} f) \,\bar{H_a}$
  on the right-hand side of~\eqref{eq:bxt-pol}, using twice the definition
  of~$\cum\colon \pcwext{\galg} \to \pcwext{\galg}$. It remains to check that the second terms are
  equal on both sides. For $a\geq 0$ both terms vanish while for $a<0$ the problem reduces to
  checking $\cum_{\!a}^0 f \cum_{\!a} f = \cum_{\!0}^a f \cum f $. Splitting the inner integral
  $\cum_{\!a} \, f=\cum_{\!a}^0 \, f+\cum \, f$ on the left-hand side yields
  $(\cum_{\!0}^a f )^2-\cum_{\!0}^a f \cum f$ since $\cum_{\!a}^0$ is $K$-linear and
  $\cum_{\!a}^0 f \in K$ by~$(\galg, \cum)$ being ordinary. Then the result follows from the
  Rota-Baxter axiom of $(\galg, \cum)$.

  As for any monoid algebra~\cite[p.~106]{Lang2002}, the
  map~$\galg \to \galg[K_\sqcup], \: f \mapsto f \cdot 1_{\pcwext{\galg}}$ is an embedding,
  hence~$\galg$ is a $K$-subalgebra of~$\pcwext{\galg} = \galg[K_\sqcup]$. Since the Rota-Baxter
  operator on~$\pcwext{\galg}$ has been defined as an extension, $(\pcwext{\galg}, \cum)$ is indeed
  a Rota-Baxter extension of~$(\galg, \cum)$.

  We have already seen that the~$S_a\colon \pcwext{\galg} \to \pcwext{\galg}$ defined above yield a
  shift map on~$\pcwext{\galg}$, and that the character~$\evl\colon \pcwext{\galg} \to K$ defined
  above is an evaluation on~$(\pcwext{\galg}, \cum)$. Hence it remains to prove the compatibility
  relation~$[S_c, \cum] = \ev_c \cum$ for the induced evaluations~$\ev_c = \evl \circ S_c$.
  By~\eqref{eq:pcw-exp}, we need only verify the relation on elements of the
  form~$f H_a \in \pcwext{\galg}$; we know it is satisfied for~$f \in \galg$ due to the shift
  relation on~$\galg$. The verification may be done by a four-fold case distinction based on the
  positivity of~$a$ and~$a-c$. For an alternative direct proof one employs the generic identities
  (valid for Rota-Baxter algebras over ordered fields)
  \begin{equation}
    \label{eq:gen-rel}
    \cum_{\!\pospart{s}} = \cum + H(s) \, \cum_{\!s}^0, \quad
    \cum_{\!\negpart{s}} = \cum + \bar{H}(s) \, \cum_{\!s}^0
    \qquad\text{and}\qquad
    \cum_{\!\pospart{s}}^0 = H(s) \, \cum_{\!s}^0, \quad \cum_{\!\negpart{s}}^0 = \bar{H}(s) \,
    \cum_{\!s}^0,
  \end{equation}
  together with the simple consequence~$\cum_{\!0}^{a-c} S_c f = \cum_{\!c}^a f$ of the shift
  relation on~$\galg$. Doing so
  yields~$\cum_{\!0}^c f + H(a) \, \cum_{\!a}^0 f + H(a-c) \, \cum_{\!c}^a f$ for both sides
  of~$[S_c, \cum] f H_a = \ev_c \cum f H_a$.
 \end{proof}

We have now an algebraic description of integration on rings of piecewise functions, constructed
from Heavisides. If all functions are piecewise smooth (cf.\@ Example~\ref{ex:pcw-smooth}), we can
add a derivation~$\der$ to obtain a differential Rota-Baxter algebra that is, however, \emph{not} an
integro-differential algebra (Proposition~\ref{prop:pcw-drb}). Normally, only in the latter case do
we speak of an \emph{induced evaluation}~$\evl := 1_\galg - \cum \der$, but since the analogous
concept is also useful in differential Rota-Baxter algebras we introduce this operation now in the
general context.

\begin{definition}
 \label{def:evaluation}
 Let $(\galg, \der, \cum)$ be a differential Rota-Baxter algebra. Then $\psevl = 1-\cum \circ \der$
 is called the \emph{induced pseudo-evaluation} of $(\galg, \der, \cum)$.
\end{definition}

Assume now that~$(\galg, \der, \cum)$ is a differential Rota-Baxter algebra such
that~$(\galg, \cum)$ satisfies the conditions of Proposition~\ref{prop:pcw}. Then we define a
\emph{derivation on the piecewise extension}~$\pcwext{\galg}$ by extending the derivation on~$\galg$
by zero. In other words, we set~$\der H_a = 0$ for all $a \in K$; then
$\der\colon \pcwext{\galg} \to \pcwext{\galg}$ is uniquely determined by the Leibniz rule. Note that
the ring of constants is enlarged to~$\ker (\partial) = K[H_a \mid a \in K]$. This reflects the
viewpoint of analysis that the derivative of the Heaviside function~$H(x-a) \in L^2(\RR)$
vanishes. Of course, this is in stark contrast to the more ambitious treatment via distributions
taken up in the next section (where the simple derivation from above is no longer in use).

\begin{proposition}
\label{prop:pcw-drb}
Let~$(\galg, \der, \cum)$ be an ordinary shifted differential Rota-Baxter algebra over the ordered
field~$(K, <)$. Then~$(\pcwext{\galg}, \der, \cum)$ is a shifted differential Rota-Baxter extension
algebra whose induced pseudo-evaluation
   \begin{equation}
     \label{eq:pcw-drb}
    \psevl(f H_a) = \ev_a(f) \, H_a + \ev_0(f) - \ev_{\negpart{a}}(f) = 
    \begin{cases}
    \ev_0(f) - \ev_a(f) \, \bar{H_a} & \text{if $a\leq 0$,}\\
    \ev_a(f) \, H_a & \text{if $a \geq 0$,}
    \end{cases}
  \end{equation}
  is not multiplicative. Hence~$(\pcwext{\galg}, \der, \cum)$ is not an integro-differential algebra.
\end{proposition}

\begin{proof}
  From the definition it is clear that~$\cum\colon \pcwext{\galg} \to \pcwext{\galg}$ is a section
  of~$\der\colon \pcwext{\galg} \to \pcwext{\galg}$, so $(\pcwext{\galg}, \der, \cum)$ is a
  differential Rota-Baxter algebra by Proposition~\ref{prop:pcw}. For showing that it is shifted, it
  remains to prove the compatibility relation~$[S_c, \der] = 0$. Since it is true by hypothesis
  on~$\galg$, we need only check that~$S_c \der \, f H_a = f'(x-c) \, H_{a-c} = \der S_c \, f H_a$.
  
 One checks immediately that the pseudo-evaluation of~$\pcwext{\galg}$ is given by~\eqref{eq:pcw-drb}, using the handy
 relation~$\ev_a = \ev_{\pospart{a}} - \ev_0 + \ev_{\negpart{a}}$. As for every integro-differential algebra, we
 have~$(K[x], \cum) \subseteq (\galg, \cum)$. Since~$K \supseteq \QQ$ is an ordered field, we have~$0 < 1 < 2$ so that
  \begin{equation*}
    \psevl(x H_1 \cdot x H_2) = \psevl(x^2 H_2) = 4 H_2 \neq
  2 H_2= H_1 \cdot 2 H_2 = \psevl(x H_1) \cdot \psevl(x H_2),
  \end{equation*}
  which shows that~$\psevl$ fails to be multiplicative.
\end{proof}

As a special case of~\eqref{eq:pcw-drb}, note that~$\psevl(H_a) = H_a$. This is in agreement with
the fact that the constants are given by~$\ker(\der) = K[H_a \mid a \in K]$; we see~$H_a$ as a
\emph{differential constant} that pseudo-evaluates to itself. Of course one must be careful not to
confuse the pseudo-evaluation with the dis\-tinguished evaluation~$\evl(H_a) = \bar{H}(a)$, which
has image~$K$ rather than~$K[H_a \mid a \in K]$.  This shows also that Proposition~\ref{prop:pcw}
cannot be strengthened to yield an \emph{ordinary} Rota-Baxter algebra~$(\pcwext{\galg},
\cum)$.
Indeed, even when~$\cum$ is injective as in Proposition~\ref{prop:pcw-drb}, the complement of its
image is larger than the ground field~$K$.

\begin{example}
  \label{ex:pcw-smooth}
  Let us show that for~$K = \RR$ and~$\galg = C(\RR)$, the piecewise extension~$\pcwext{C(\RR)}$
  yields the usual Rota-Baxter algebra of \emph{piecewise continuous functions}~$PC(\RR)$, up to a
  quotient.\footnote{Note the difference between $P$ and $\mathcal{P}$ in this example; the latter
    stands for the algebraic construction described above while the former denotes the standard
    notion of piecewise functions in real analysis.} Taking the subalgebra~$\galg = C^\infty(\RR)$,
  we obtain similarly the differential Rota-Baxter algebra $\pcwext{C^\infty(\RR)}$ as the algebraic
  counterpart of the \emph{piecewise smooth functions}~$PC^\infty(\RR)$.

  Let~$f\colon D \to \RR$ be continuous/smooth on an open set~$D \subseteq \RR$. Then we call~$f$
  \emph{piecewise continuous\,/smooth} if~$D$ has finite complement in~$\RR$ and~$f$ has one-sided
  limits at each~$x \in \RR\setminus D$. We call~$x \in \RR\setminus D$ regular if~$f$ extends to a
  continuous/smooth map~$f_x\colon D \cup \{x\} \to \RR$; in that
  case~$f_x(x) = \lim_{\xi \to x} f(\xi)$. For a piecewise function~$f\colon D \to \RR$ we shall
  write~$\tilde{f}\colon \tilde{D} \to \RR$ for its maximal continuous/smooth extension. We
  define~$PC(\RR)$ and~$PC^\infty(\RR)$ as the set of piecewise functions~$f\colon D \to \RR$ with
  $\tilde{f} = f$. They become rings by setting
  \begin{equation*}
    f_1 + f_2 := \widetilde{f_1 \oplus f_2}, \qquad f_1 \cdot f_2 =
    \widetilde{f_1 \odot f_2},
  \end{equation*}
  where~$f_1 \oplus f_2$ and~$f_1 \odot f_2$ denote the pointwise sum and product of
  functions~$f_i\colon D_i \to \RR$ after restricting each to their common domain~$D_1 \cap D_2$. We
  endow~$PC(\RR)$ and its subalgebra $PC^\infty(\RR)$, with the usual Rota-Baxter
  operator~$\cum = \cum_{\!0}^x$; it is clear that this yields Rota-Baxter
  algebras~$(PC^\infty(\RR), \cum) \subset (PC(\RR), \cum)$. Moreover, we can use the standard
  derivation~$\der = \tfrac{d}{dx}$ on the piecewise smooth functions, obtaining a differential
  Rota-Baxter algebra~$(PC^\infty(\RR), \der, \cum)$.

  There is an algebra homomorphism~$\pi\colon \pcwext{C(\RR)} \to PC(\RR)$ that fixes~$C(\RR)$ and
  that sends each $H_a\: (a \in \RR)$ to~$H(x-a) \in PC(\RR)$. Clearly, we have
  also~$\pcwext{C^\infty(\RR)} \to PC^\infty(\RR)$ by restriction. We show that both
  homomorphisms~$\pi$ are surjective: Each~$f \in PC(\RR)$ or $f \in PC^\infty(\RR)$ with regular
  part~$f\colon D \to \RR$ can be written as
  \begin{equation*}
    f(x) = \sum_{i=0}^n f_i(x) \, H(x-x_i) \, H(x_{i+1}-x)
  \end{equation*}
  where~$\RR \setminus D = \{ x_1 < \dots < x_n\}$ and~$f_i\colon \RR \to \RR$ is an arbitrary
  continuous/smooth extension of the \emph{function pieces}~$f |_{(x_i, x_{i+1})}$. Here we
  set~$x_0 = -\infty$ and~$x_{n+1} = +\infty$ with the understanding
  that~$H(x+\infty) = H(\infty-x) = 1$. With this choice of pieces~$f_0, \dots, f_n$ we
  have~$f = \pi\big(\sum_i f_i \, H_{x_{i}} \bar{H}_{x_{i+1}}\big)$, so~$\pi$ is indeed
  surjective. The
  ideals~$$\mathcal{R} := \ker\big(\pi\colon \pcwext{C(\RR)} \to PC(\RR)\big) \quad\text{and}\quad
  \mathcal{R}_\infty := \ker\big(\pi\colon \pcwext{C^\infty(\RR)} \to PC^\infty(\RR)\big)$$
  encode the algebraic relations between continuous\,/smooth functions and Heavisides, for instance
  $b(x) \, H(x-2) = 0$ where~$b(x)$ is any bump function supported in~$[-1,1]$. Hence we obtain the
  quotient representations~$PC(\RR) \cong \pcwext{C(\RR)}/\mathcal{R}$
  and~$PC^\infty(\RR) \cong \pcwext{C^\infty(\RR)}/\mathcal{R}_\infty$.
\end{example}

\begin{example}
  \label{ex:pcw-anl}
  The case of \emph{piecewise real-analytic functions} is essentially different since analytic
  continuation breeds multi-valued functions (or Riemann surfaces) whose proper treatment involves
  sheaf-theoretic methods combined with integro-differential structures. This would lead us too far
  afield but may provide interesting substance for future research.

  For keeping things simple, let us consider the complex algebra~$PC^\omega(\RR)$ of piecewise
  real-analytic functions, in the sense that each function piece~$f_i\colon (x_i, x_{i+1}) \to \CC$
  extends to an entire function.\footnote{This is a very restricted setting since
    even~$\tfrac{1}{x} \not\in PC^\omega(\RR)$. Indeed, keeping~$\tfrac{1}{x}$ creates analytic and
    algebraic complications: multi-valued logarithms and
    quasi-antiderivatives~\cite[Ex.~4.3]{GuoRegensburgerRosenkranz2012}, respectively.} Apart from
  this distinction, the construction of~$PC^\omega(\RR)$ is completely analogous to that
  of~$PC^\infty(\RR)$ in Example~\ref{ex:pcw-smooth}. Taking now the algebra~$\galg = C^\omega(\RR)$
  of global real-analytic functions (real restrictions of entire functions) as coefficient algebra,
  we can apply the construction of Example~\ref{ex:pcw-smooth}. But now the relation
  ideal~$\mathcal{R}$ is trivial because each real-analytic function piece extends \emph{uniquely}
  to a global real analytic function, and we obtain~$PC^\omega(\RR) \cong \pcwext{C^\omega(\RR)}$.
\end{example}

Piecewise defined functions are a major motivation for introducing \emph{distributions}, via
generalized derivatives. In particular, we will no longer view~$\der H_a$ as identically zero but as
a ``Dirac delta''~$\delta_a$, sometimes written~$\delta_a(x) = \delta(x-a)$.  Again we shall pursue
a purely algebraic route to introduce these quantities along with an integro-differential structure.

\section{Construction of the Distribution Module}
\label{sec:distmod}

The basic property of the Dirac distribution~$\delta_a$ concentrated at a source point~$a$ is that
its only nonzero ``value'' is assumed for~$x=a$, in the sense that~$f \delta_a$ vanishes identically
when~$f(a) = 0$. In other words, $f \delta_a$ only depends on~$f(a)$ and not on all of~$f$,
and one has the \emph{sifting property}
\begin{equation}
  \label{eq:extract-value}
  f \delta_a = f(a) \, \delta_a
\end{equation}
for ``extracting'' the source value. This will be the basis of our algebraic
construction.\footnote{Using~$\pcwext{\galg} \, \{\delta_a \mid a \in K\}_1$ instead
  of~$\galg\{H_a \mid a \in K\}_1$ may seem more natural and incremental, but it runs into problems
  with the Leibniz rule: see Remark~\ref{rem:no-diffring}.}

\begin{definition}
\label{def:distmod}
Let $(\galg, \der)$ be a differential algebra over a ring~$K$. We define the \emph{distribution
  module}~$(\distmod{\galg}, \vder)$ as the differential
$\galg$-module~$\galg\{H_a \mid a \in K\}_1/Z$, where~$Z$ denotes the differential $\galg$-submodule
generated by~$\{ f \, \delta_a - \ev_a(f)\, \delta_a \mid f \in \galg, a \in K\}$.
\end{definition}

Recall that~$\galg\{X\}_1$ denotes the module of affine differential polynomials in~$X$. We have
also employed the abbreviation~$\delta_a := H_a'$, which we shall continue to use throughout this
paper (of course derivatives~$\vder\phi$ of~$\phi \in \distmod{\galg}$ are also denoted
by~$\phi'$). The order on~$K$ induces an elimination ranking~$\prec$
on~$\galg \{ H_a \mid a \in K \}$ and thus a Noetherian term order on the
$\galg$-module~$\galg\{H_a \mid a \in K\}_1$. We have~$H_a^{(m)} \prec H_b^{(n)}$ iff~$a < b$ or
otherwise~$a=b$ and~$m < n$. In the sequel we shall always employ this term order on the free
differential module underlying~$\distmod{\galg}$. It is easy to get a kind of Gr{\"o}bner basis
for~$Z$ with respect to this term order. Moreover, the direct decomposition
\begin{equation*}
  \galg\{H_a \mid a \in K\}_1 = \bigoplus_{a \in K} \galg\{H_a\}_1
\end{equation*}
of differential $\galg$-modules induces the direct decomposition $Z = \bigoplus Z_a$, and we
write~$$\zeta = \sum_{a \in K} \zeta_a \quad (\zeta_a \in Z_a)$$ for the corresponding sum
representation of an arbitrary~$\zeta \in Z$. Let us now proceed to the crucial \emph{Presentation
  Lemma} for exhibiting the Gr{\"o}bner basis.

\begin{lemma} \label{lem:distmod-rel}
 The differential~$\galg$-module~$Z$ in Definition~\ref{def:distmod} is generated as an $\galg$-module by
\begin{equation}
  \label{eq:gen-dist-mod}
  \Big\{
  f \, \delta_a^{(k)} - \sum_{i=0}^k \binom{k}{i} (-1)^i
  \ev_a(f^{(i)}) \, \delta_a^{(k-i)} \,\Big|\,
  a \in K, f \in \galg, k \ge 0 \Big\} ,
\end{equation}
which forms a Gr{\"o}bner basis of~$Z$. For every element~$\zeta \in Z$, the leading
coefficient~$f_a$ of each~$\zeta_a$ has the property~$\ev_a(f_a) = 0$. Relative to this Gr{\"o}bner
basis, the elements~$\phi + Z \in \distmod{\galg}$ of the quotient have the canonical
representatives
\begin{equation}
\label{eq:dist-can}
\phi =  f + \sum_{a \in K} \, f_a H_a +  \sum_{a \in K} \sum_{k\geq 0} \lambda_{a, k} \, \delta_a^{(k)}
\quad (f, f_a \in \galg; \: \lambda_{a,k}\in K)
\end{equation}
with only finitely many~$f_a$ and $\lambda_{a,k}$ nonzero.
\end{lemma}

\begin{proof}
We split the proof in several steps.

 \begin{enumerate}
 \item Let us first show that~$Z$ contains the $\galg$-module generated
   by~\eqref{eq:gen-dist-mod}. Since the components~$Z_a$ are independent, we fix an~$a \in K$ and
   abbreviate the corresponding elements of~\eqref{eq:gen-dist-mod} by~$\zeta_{f,k}$. We prove by
   induction on~$k$ that all~$\zeta_{f,k}$ are contained in~$Z$. For~$k=0$ this is clear
   since~$\zeta_{f,0}$ is a (differential) generator of~$Z$. Assume that all~$\zeta_{f,j}$
   with~$j<k$ and arbitrary~$f \in \galg$ are contained in~$Z$; we show that~$\zeta_{f,k} \in Z$ for
   a fixed~$f \in \galg$. Differentiating an arbitrary
   generator~$f \, \delta_a - \ev_a(f) \, \delta_a$ of~$Z$, we obtain
\begin{equation*}
  \vder^k \zeta_{f,0} = f \, \delta_a^{(k)} + \sum_{i=0}^{k-1} \binom{k}{i}
  (\der^{k-i} f) \, \delta_a^{(i)} - \ev_a(f) \, \delta_a^{(k)} \in Z.
\end{equation*}
Eliminating the terms~$f^{(i)} \, \delta_a^{(k-i)}$ yields
\begin{equation*}
  \vder^k \zeta_{f,0} - \sum_{i=0}^{k-1} \binom{k}{i} \zeta_{f^{(k-i)}, i} =
  f \, \delta_a^{(k)}
  + \sum_{j=0}^{k-1} \sum_{i=j}^{k-1} \binom{k}{i}
  \binom{i}{j} (-1)^{i+j} \ev_a(f^{(k-j)}) \, \delta_a^{(j)}
  - \ev_a(f) \, \delta_a^{(k)}
\end{equation*}
after an index transformation. The double sum simplifies to
\begin{align*}
  \quad\qquad \sum_{j=0}^{k-1} & \sum_{i=j}^{k-1} \cdots = 
  \sum_{j=0}^{k-1} (-1)^j \ev_a(f^{(k-j)}) \, \delta_a^{(j)} 
  \sum_{i=j}^{k-1} \binom{k}{i} \binom{i}{j} (-1)^i\\
  &= (-1)^{k+1} \sum_{j=0}^{k-1} \binom{k}{j} (-1)^j \ev_a(\der^{k-j}
  f) \, \delta_a^{(j)}
  = - \sum_{j=1}^k \binom{k}{j} (-1)^j
  \ev_a(f^{(j)}) \, \delta_a^{(k-j)},
\end{align*}
using the fact that the inner sum above evaluates
to~$(-1)^{k+1} \tbinom{k}{j}$. Extending the range of the last sum
to include~$j=0$ incorporates the remaining term so that
\begin{equation*}
  \vder^k \zeta_{f,0} - \sum_{i=0}^{k-1} \binom{k}{i} \zeta_{f^{(k-i)}, i} = f \, \delta_a^{(k)} - \sum_{j=0}^k \binom{k}{j} (-1)^j
  \ev_a(f^{(j)}) \, \delta_a^{(k-j)} = \zeta_{f,k},
\end{equation*}
which shows that~$\zeta_{f,k} \in Z$ since all~$\zeta_{\der^{k-i} f, i} \in
Z$ by the induction hypothesis.

\item For establishing the converse inclusion that~$Z$ is contained in the $\galg$-module generated by~\eqref{eq:gen-dist-mod}, it suffices to show that all the derivatives~$\vder^k \zeta_{f,0}$ are $\galg$-linear combination of the~$\zeta_{f,j}$. But this is clear from the last identity of the previous item.

\item We proceed now to the statement about the leading coefficients.
To this end, we rewrite the module generators as
\begin{align*}
 \zeta_{f,k}= \, (f \, - \ev_a f) \, \delta_a^{(k)} - \sum_{i=1}^k \binom{k}{i} (-1)^i
  \ev_a(f^{(j)}) \, \delta_a^{(k-i)},
\end{align*}
from which the claim is evident.

\item Next we must show that~\eqref{eq:gen-dist-mod} forms a Gr{\"o}bner bases for the $\galg$-module~$Z$. This involves a slight variation of the usual setting of Gr{\"o}bner bases for commutative polynomials~\cite{Buchberger2006} since we have infinitely many indeterminates and the coefficient ring~$\galg$ may have zero divisors (it is certainly not a field). Since we need only the linear fragment of the polynomial ring, we may use the approach of~\cite[\S9.5a]{Bergman1978}, which also allows for infinitely many generators. In the notation of~\cite[\S9.5a]{Bergman1978}, we set~$k = K$ and $R = \galg$ with trivial presentation (every element of~$\galg$ is a generator, and there are no relations) and the module~$M=Z$ with generators~$\smash{\delta_a^{(k)}}$ and relations~\eqref{eq:gen-dist-mod}. The only S-polynomials~$\sigma$ arise from the self-overlaps of~\eqref{eq:gen-dist-mod}, namely $f \! \bar{f} \, \smash{\delta_a^{(k)}}$, and this yields
\begin{align*}
  \quad\qquad \sigma &= \sum_{i=0}^k \binom{k}{i} (-1)^i \Big( \ev_a(f^{(i)}) \, \bar{f} - \ev_a(\bar{f}^{(i)}) \, f \, \Big) \, \delta_a^{(k-i)}\\
  & \to \sum_{i=0}^k \sum_{j=0}^{k-i} \binom{k}{i} \binom{k-i}{j} (-1)^{i+j} \ev_a(f^{(i)} \, \bar{f}^{(j)} - \bar{f}^{(i)} \,f^{(j)}) \, \delta_a^{(k-i-j)} = \sum_{i+j \le k} e_{ij} \, \eta_{ij},
\end{align*}
which vanishes since the summation is over a triangle~$i+j \le k$, symmetric with respect to~$i \leftrightarrow j$, while the evaluation term~$e_{ij} = \ev_a(\ldots)$ is antisymmetric and the trinomial term~$\eta_{ij} = k! \smash{\bigm/} i! j! (k-i-j)! \, (-1)^{i+j} \, \delta_a^{(k-i-j)}$ symmetric.

\item The analog of the Diamond Lemma in~\cite[\S9.5a]{Bergman1978} ensures that the normal forms
  of~\eqref{eq:gen-dist-mod} are canonical representatives of the congruence
  classes~$\phi + Z \in \distmod{\galg}$. Hence it suffices to characterize the normal forms of an
  arbitrary (noncanonical) representative~$\phi$. Clearly, every such~$\phi$ is reducible as long as
  it contains any~$\delta_a^{(k)}$ with a coefficient in~$\galg \setminus K$; hence we can
  achieve~\eqref{eq:dist-can}, which is clearly irreducible with respect to~\eqref{eq:gen-dist-mod}.
\end{enumerate}
This completes the proof of the Presentation Lemma.
\end{proof}

We identify the Heavisides~$H_a \in \distmod{\galg}$ with the
corresponding~$H_a \in \pcwext{\galg}$. As a consequence, we
have~$\pcwext{\galg} \subset \distmod{\galg}$ as plain $\galg$-modules\footnote{The total order
  presupposed in the definition of~$\pcwext{\galg}$ is irrelevant for the identification of
  modules: It is only needed for the ring structure of~$\pcwext{\galg}$, which is momentarily
  ignored but incorporated later (Remark~\ref{rem:no-diffring}).} but not as differential
$\galg$-modules: Indeed, the derivation~$\der\colon \pcwext{\galg} \to \pcwext{\galg}$ just
annihilates the Heavisides, $\der H_a = 0$, whereas
$\vder\colon \distmod{\galg} \to \distmod{\galg}$ sends them to~$\vder H_a = \delta_a$. The
situation for the \emph{Rota-Baxter structure} is very different---in fact, we shall see
that~$\pcwext{\galg} \subset \distmod{\galg}$ as Rota-Baxter $\galg$-modules. To this end we
define~$\vcum\colon\distmod{\galg} \to \distmod{\galg}$ as an extension of
$\cum\colon \pcwext{\galg} \to \pcwext{\galg}$ via the recursion
\begin{equation}
  \label{eq:rb-dist}  
   \vcum f \delta_a^{(k)} =
   \begin{cases}
    \ev_a(f) \, \vcum \delta_a & \text{for $k=0$,}\\
    f\delta_a^{(k-1)}-\vcum f'\delta_a^{(k-1)} & \text{for $k > 0$,}
    \end{cases}
\end{equation}
where~$\ev_a$ denotes the evaluation in~$\galg$ and the integral in the base case is given in terms of the (rising or falling) Heaviside function via
\begin{equation}
  \label{eq:dirac-ader}
  \vcum \delta_a = H_a - \bar{H}(a) = H(a) - \bar{H}_a,
\end{equation}
which may also be written symmetrically as~$\vcum \delta_a = H(a) \, H_a - \bar{H}(a) \,
\bar{H}_a$. Setting~$f=1$ in~\eqref{eq:rb-dist}, we obtain the higher Dirac antiderivatives~$\vcum
\smash{\delta_a^{(k)}} = \smash{\delta_a^{(k-1)}}$ for~$k > 0$. Hence the \emph{induced evaluation} $\mevl = 1_{\distmod{\galg}} - \vcum \vder$ of the module generators is given by
\begin{equation}
  \label{eq:mod-eval}
  \mevl(H_a)= \bar{H}(a)  \qquad\text{and}\qquad
  \mevl(\delta_a^{(k)})= 0 \quad (k \geq 0),
\end{equation}
which---unlike in the piecewise extension---do go to the ground field~$K$. This should be contrasted
to the pseudo-evaluation~$\psevl(H_a) = H_a$ we introduced earlier (after
Proposition~\ref{prop:pcw-drb}). We shall come back
to~$\mevl\colon \distmod{\galg} \to \distmod{\galg}$ in due course (see
Lemmas~\ref{lem:char-intdiff-mod} and~\ref{prop:intdiff-mod}).

\begin{remark}
  \label{rem:alt-der-distmod}
  The definition of the Rota-Baxter operator~$\vcum\colon \distmod{\galg} \to \distmod{\galg}$ 
  in~\eqref{eq:pcw-def1} and~\eqref{eq:rb-dist}--\eqref{eq:dirac-ader} may be rephrased more
  economically by joining~\eqref{eq:pcw-def1} with the single formula
  \begin{equation}
    \label{eq:alt-der-distmod}
    \vcum f H_a^{(k+1)} = f H_a^{(k)} - \vcum f' H_a^{(k)}  \quad (k \in \NN) .
  \end{equation}
  While this is evident for~$k>0$, it requires a small calculation to confirm in the case~$k=0$. The
  main point is to use~\eqref{eq:pcw-def1} in conjunction with the
  relation~$\ev_a = \ev_{\pospart{a}} - \ev_0 + \ev_{\negpart{a}}$ already used in the proof of
  Proposition~\ref{prop:pcw-drb} and the simple fact
  that~$f(a^+) = f(a) \, H(a) + f(0) \, \bar{H}(a)$.  We have chosen the split
  definition~\eqref{eq:rb-dist}--\eqref{eq:dirac-ader} above since we find it more intuitive.
\end{remark}


\begin{wrapfigure}[6]{R}{0.35\textwidth}
\centering\vspace{-1.95em}

$\xymatrix @M=0.6pc @=2.25pc @C=5pc {\pcwext{\galg} \ar@{^(->}^{\iota}[r] & \distmod{\galg}\\
\galg \ar^{u_{\pcwext{}}}[u] \ar_{u_{\distmod{}}}[ur]}$
\end{wrapfigure}

Our main result states that the distribution module $\distmod{\galg}$ is an extension of the ground
algebra~$\galg$ that contains the piecewise extension~$\pcwext{\galg}$ qua Rota-Baxter module; see
the figure nearby, where~$\iota$ is the embedding of Rota-Baxter $\galg$-modules
while~$u_{\pcwext{}}$ and~$u_{\distmod{}}$ are the structure maps of the~$\galg$-modules
$\pcwext{\galg}$ and $\distmod{\galg}$, respectively.

\begin{theorem}
  \label{thm:dist-drbm}
  Let~$(\galg,\der, \cum)$ be an ordinary shifted integro-differential algebra. Then the distribution module~$(\distmod{\galg}, \vder, \vcum)$ is a differential Rota-Baxter module over~$\galg$ that extends $(\pcwext{\galg}, \cum)$ as a Rota-Baxter module.
\end{theorem}
 
\begin{proof}
  It suffices to prove the following statements:
 \begin{enumerate}
 \item\label{it:rb-wdef} \emph{The map~$\vcum\colon \distmod{\galg} \to \distmod{\galg}$ is well-defined.} For this we have to show that~$\vcum Z\subseteq Z $, which we do by the aid (and with the notation) of Lemma~\ref{lem:distmod-rel}. So for~$a \in K$ fixed, we prove~$\vcum \zeta_{f,k} \in Z$ for all~$f \in \galg$ and~$k \ge 0$. Using induction on~$k$, the base case~$k=0$ follows immediately from~\eqref{eq:rb-dist}. For the induction step it suffices to prove that $\vcum \zeta_{f,k+1} = \zeta_{f,k} - \vcum \zeta_{f',k}$ for all $f \in \galg$. Using the generators~\eqref{eq:gen-dist-mod} we have
\begin{equation*}
  \vcum \zeta_{f,k+1} = \vcum f \, \delta_a^{(k+1)} - \sum_{i=0}^{k+1} \binom{k+1}{i} (-1)^i
  \ev_a(f^{(i)}) \, \vcum \delta_a^{(k-i+1)},
\end{equation*} 
which simplifies by~\eqref{eq:rb-dist} and the binomial recursion~$\binom{k+1}{i}= \binom{k}{i}+ \binom{k}{i-1}$ to
\begin{equation*}
  f \, \delta_a^{(k)} - \sum_{i=0}^{k} \binom{k}{i} (-1)^i \ev_a(f^{(i)}) \, \delta_a^{(k-i)}
   {} - \bigg( \vcum f' \, \delta_a^{(k)} - \sum_{i=0}^{k} \binom{k}{i} (-1)^i \ev_a(f'^{\,(i)}) \, \vcum \delta_a^{(k-i)} \bigg) 
  = \zeta_{f,k} - \vcum \zeta_{f',k}
\end{equation*}
and thus completes the induction.
\item \emph{The map~$\vcum\colon \distmod{\galg} \to \distmod{\galg}$ is a Rota-Baxter operator.} Hence we must prove, for any~$f,g \in \galg$ and~$a \in K$ and~$k \ge 0$, the Rota-Baxter axiom
\begin{equation}
\label{eq:rb-mod}
\cum f \cdot \vcum g \, \delta_a^{(k)} = \vcum f \, \vcum g \, \delta_a^{(k)} + \vcum (\cum f)  g \, \delta_a^{(k)}.
\end{equation}
We fix~$a \in K$ and use induction on $k$ to prove~\eqref{eq:rb-mod} for all $f, g \in \galg$. In the base case, exploring definition~\eqref{eq:rb-dist} reveals that~$\ev_a(g)$ factors on both sides of~\eqref{eq:rb-mod}; hence it suffices to take~$g=1$. The left-hand side is then~$\cum f \cdot \vcum \delta_a$ while we obtain
\begin{align*}
\big( H(a) \, \cum f H_a - \bar{H}(a) \, \cum f \bar{H}_a \big) + \cum_{\!0\,}^a f \cdot \vcum \delta_a
\end{align*}
for the right-hand side. Using the definition~\eqref{eq:pcw-def1}, \eqref{eq:pcw-def2} of the
Rota-Baxter operator on the piecewise extension~$\pcwext{\galg} \subset \distmod{\galg}$ and
properties of the Heaviside operator, the first parenthesized term
becomes~$\cum_{\!a\,} f \cdot \vcum \delta_a$ and then combines with the remaining term
to~$\cum f \cdot \vcum \delta_a$; this completes the base case of the induction. Assume now
that~\eqref{eq:rb-mod} holds for~$k$; we show that it holds for~$k+1$. Using the
definition~\eqref{eq:rb-dist} once, the left-hand side
is~$\cum f \cdot (g \delta_a^{(k)}- \vcum g' \, \delta_a^{(k)})$. On the right-hand side we
use~\eqref{eq:rb-dist} on each summand to get
\begin{align*}
\vcum  f g & \, \delta_a^{(k)} - \vcum f \, \vcum g' \, \delta_a^{(k)} + (\cum f) g \, \delta_a^{(k)} - \vcum fg \, \delta_a^{(k)} - \vcum (\cum f) g' \, \delta_a^{(k)}\\
& = (\cum f) g \, \delta_a^{(k)} - \vcum f \, \vcum g' \, \delta_a^{(k)} - \vcum (\cum f) g' \, \delta_a^{(k)} .
\end{align*}
Canceling the first terms on both sides, we end up with~\eqref{eq:rb-mod} where~$g$ is replaced by~$g'$, and this holds by the induction hypothesis.

\item \emph{The map~$\vder\colon \distmod{\galg} \to \distmod{\galg}$ is a well-defined derivation.}
  In fact, it suffices to prove well-definedness since the derivation property then follows
  immediately from the definition of~$\distmod{\galg}$ as a quotient of a differential module. Hence
  we must prove~$\der Z \subset Z$, but this follows directly
  from~$\vder \zeta_{f,k} = \zeta_{f,k+1} + \zeta_{f',k}$, obtained by differentiating the
  identity of Item~\eqref{it:rb-wdef}.

\item \emph{The Rota-Baxter operator~$\vcum$ is a section of the derivation~$\vder$.} We start by
  showing that $\vder \vcum \, f H_a = f H_a$ holds for all $f \in \galg$. Using
  definition~\eqref{eq:pcw-def1} for the Rota-Baxter operator on~$\pcwext{\galg}$ and the Leibniz
  rule together with the basic relation~$f \delta_a = \ev_a(f) \, \delta_a$ of~$Z$ yields
\begin{equation}
\label{eq:sec-ax}
\vder \vcum \, f H_a = f H_a + \big( \cum_{\!\pospart{a}}^a f \big) \, \delta_a  + \big( \cum_{\!\negpart{a}}^0 f \big) \, \delta_a
\end{equation}
whose last two terms combine to~$0+0$ in the case~$a \ge 0$ and again to~$\cum_0^a f + \cum_a^0 f = 0$ in the case~$a \le 0$. Hence the right-hand side of~\eqref{eq:sec-ax} is indeed~$f H_a$. Now for elements of the form~$f\delta^{(k)}_a$ we use induction on~$k$. In the base case we have
$$\vder \vcum \, f \delta_a = \ev_a(f) \, \vder \big(H_a - H(a) \big) = \ev_a(f) \, \delta_a = f \delta_a,$$
where the last step uses again the basic relation of~$Z$. Now assume~$\vder \vcum \, f \delta_a^{(k)} = f \delta_a^{(k)}$ for a fixed~$k$. Then we have
$$\vder \vcum \, f \delta_a^{(k+1)} = \vder (f \delta_a^{(k)}) - \vder \vcum \, f' \delta_a^{(k)} = f \delta_a^{(k+1)},$$
where the last step uses the Leibniz rule for~$\vder$ and the induction hypothesis. This completes the proof of the section axiom for~$\vcum$.
\end{enumerate}\vspace*{-\baselineskip}
\end{proof}

Before analyzing some further properties of~$\distmod{\galg}$, let us digress briefly for addressing
an important ``\emph{design question}'' that has come up repeatedly in the course of building up the
algebraic structure of Heaviside functions and Dirac distributions.

\begin{remark}
  \label{rem:no-diffring}
  It sounds tempting to introduce distributions as a \emph{differential ring extension}
  of~$\pcwext{\galg}$. However, the famous negative result~\cite{Schwartz1954} serves as a warning
  signal that we should not be overly optimistic in that respect. In the algebraic setup, we see
  that things are in a sense worse---we cannot even expect a Leibniz rule that involves Heavisides:
  Since $H_a^2 = H_a$ in~$\pcwext{\galg}$, differentiation would
  yield~$2 H_a \, \delta_a = \delta_a$ as a new relation. Hence we would need~$\distmod{\galg}$ to
  be a module over~$\pcwext{\galg}$, though it would not be a differential module since the
  derivation does not restrict to a map~$\vder\colon \pcwext{\galg} \to \pcwext{\galg}$.
  Furthermore, we would now expand the relations~$Z$ of Definition~\ref{def:distmod} to
  include~$F \, \delta_a - \ev_a(F) \, \delta_a$ for all~$F \in \pcwext{\galg}$ and not just
  for~$F \in \galg$; in conjunction with the new relation this forces on us the symmetric convention
  for the Heaviside operator. We have discarded the latter (see before Proposition~\ref{prop:pcw})
  solely for ensuring multiplicative evaluations on~$\pcwext{\galg}$, so let us momentarily assume
  the symmetric convention. At any rate, the new relation~$2 H_a \, \delta_a = \delta_a$ implies
  that
  \begin{equation*}
    \delta_a = 2 H_a \, \delta_a = 2 H_a \, (2 H_a \, \delta_a) = (4 H_a^2) \, \delta_a = 4 H_a \,
    \delta_a = 2 \delta_a,
  \end{equation*}
  which means~$\delta_a = 0$ and hence~$\distmod{\galg} = \pcwext{\galg}$.

  It is now clear why our construction of~$\distmod{\galg}$ was based on a free differential module
  over~$\galg$ rather than some module over~$\pcwext{\galg}$. On the other hand, it is clear
  that~$\pcwext{\galg} \subset \distmod{\galg}$, and we may export the \emph{product structure of
    the piecewise extension}~$\pcwext{\galg}$ to the distribution module~$\distmod{\galg}$. Hence we
  may say~$H_a^2 = H_a \in \distmod{\galg}$ but we are barred from differentiating this relation
  since~$\distmod{\galg}$ is a differential module over~$\galg$ and not over~$\pcwext{\galg}$.\qed
\end{remark}

In Theorem~\ref{thm:dist-drbm} we use the rather strong assumption that the ground
algebra~$(\galg, \der, \cum)$ is an ordinary integro-differential algebra since this is what we need
in our applications. This has the nice consequence that the distribution module itself has similar
properties. However, for a general differential Rota-Baxter module one must distinguish between the
\emph{strong Rota-Baxter axiom}~\eqref{eq:str-rb-axiom} for coefficients and for module elements
(whether one may pull out constants of either kind from the integral). In the sequel, we shall
write~$\mevl := 1_M - \vcum \vder$ for the induced (pseudo)evaluation in an arbitrary differential
Rota-Baxter module~$(M, \vder, \vcum)$.

\def\myequiv{$\quad\Leftrightarrow\quad$}

\begin{lemma}
  \label{lem:char-intdiff-mod}
  Let~$(M, \vder, \vcum)$ be a differential Rota-Baxter module over the integro-differential algebra~$(\galg, \der, \cum)$. Then we have the following equivalences (where~$f, c \in \galg$ and~$\phi, \gamma \in M$):
  \begin{enumerate}
  \item $\vcum c \phi = c \, (\vcum \phi) \quad (\text{for all } c \in \ker{\der})$ \myequiv $\vcum f \phi = f \, \vcum \phi - \vcum f' \, \vcum \phi$
  \item $\vcum f \gamma = (\cum f) \, \gamma \quad (\text{for all } \gamma \in \ker{\vder})$ \myequiv $\vcum f \phi = (\cum f) \, \phi - \vcum (\cum f) \, \phi'$
  \item $\mevl(f \phi) = \evl(f) \, \mevl(\phi)$ \myequiv (1a) \textsl{\&} (2a) \myequiv (1b) \textsl{\&} (2b)
  \end{enumerate}
  If~$M$ is ordinary, property~(1a) and hence~(1b) is automatic; if~$\galg$ is ordinary, the same holds for properties~(2a) and~(2b).
\end{lemma}

\begin{proof}
  The implications are similar to the corresponding ones given
  in~\cite[Thm.~2.5]{GuoRegensburgerRosenkranz2012} for noncommutative rings, provided one splits
  the properties of the ring into its left-hand and right-hand versions.
  
  Let us start with~(1). The implication from right to left is obvious, so assume the homogeneity condition~(1a) for~$c \in \ker{\der}$. Then we have
  \begin{align*}
    f \, \vcum \phi = (f - \cum f') \, \vcum \phi + (\cum f')(\vcum \phi) = \vcum f\phi - \vcum (\cum f') \phi + (\cum f')(\vcum \phi),
  \end{align*}
  where we have used the homogeneity condition for~$c = f - \cum f' \in \ker{\der}$. By the (plain) Rota-Baxter axiom the last term above is~$(\cum f')(\vcum \phi) = \vcum (\cum f') \phi + \vcum f' \vcum \phi$, hence one immediately obtains~(1b). The proof of the equivalence~(2a) $\Leftrightarrow$ (2b) is completely analogous. Turning to~(3), let us first assume the multiplicativity condition~$\mevl(f \phi) = \evl(f) \, \mevl(\phi)$. Specializing to~$f = c \in \ker{\der}$ yields~$\vcum c\phi' = c \, \vcum \phi'$, which is~(1a) since~$\vder$ is surjective; likewise specializing to~$\phi = \gamma \in \ker{\vder}$ gives~$\vcum f' \gamma = (\cum f') \, \gamma$, which is~(2a) since~$\der$ is surjective as well. For the converse statement, we may assume~(1b) and~(2b) to prove the multiplicativity condition for the evaluations. From the plain Rota-Baxter axiom we have
  \begin{align*}
    (\cum f')(\vcum \phi') = \vcum (\cum f') \, \phi' + \vcum f' \vcum \phi' = \bigg( (\cum f') \, \phi - \vcum f' \phi \bigg) 
    + \bigg( f \, \vcum \phi' - \vcum f \phi' \bigg),
  \end{align*}
  where the first and the second parenthesized terms come from applying~(2b) and~(1b), respectively. Subtracting~$f \phi$ from both sides of the above identity and rearranging, one obtains exactly~$\mevl(f \phi) = \evl(f) \, \mevl(\phi)$.
\end{proof}

If~$(M, \vder, \vcum)$ satisfies the multiplicativity requirement of~(3) above, we shall call it an
\emph{integro-differential module} (similar terms could be introduced for the weaker properties~(1)
and~(2) but will not be needed for our purposes). It is now easy to see that the distribution
module~$\distmod{\galg}$ of Theorem~\ref{thm:dist-drbm} is indeed an ordinary integro-differential
module in this sense.

\begin{proposition}
  \label{prop:intdiff-mod}
  If~$(\galg,\der, \cum)$ is an ordinary shifted integro-differential algebra,
  $(\distmod{\galg}, \vder, \vcum)$ is an ordinary integro-differential module over~$\galg$.
\end{proposition}

\begin{proof}
  Let us first prove that~$\distmod{\galg}$ is ordinary, meaning~$\ker{\vder} = K$. Hence assume~$\vder \phi = 0$ for an arbitrary element~$\phi \in \distmod{\galg}$. By Lemma~\ref{lem:distmod-rel} we may assume
  \begin{equation*}
    \phi = f + \sum_{a \in K} f_{a} H_a +  \sum_{a \in K} \sum_{k \geq 0} \lambda_{a,k} \, \delta_a^{(k)}
  \end{equation*}
  for some~$f, f_a \in \galg$ and~$\lambda_{a,k} \in K$ so that
  \begin{equation*}
    f' + \sum_{a \in K} (f_a' H_a + f_a \delta_a) + \sum_{a \in K} \sum_{k \geq 0} \lambda_{a, k} \, \delta_a^{(k+1)} = 0.
  \end{equation*}
  Since the above representation is canonical by Lemma~\ref{lem:distmod-rel}, we
  obtain~$f' = f_a' = f_a = \lambda_{a,k} = 0$. But then we have~$\phi = f \in \ker{\der} = K$, so
  the differential module~$(\distmod{\galg}, \vder)$ is ordinary. From
  Lemma~\ref{lem:char-intdiff-mod} it follows immediately that~$(\distmod{\galg}, \vder, \vcum)$ is
  also an integro-differential module.
\end{proof}

The distribution module~$(\distmod{\galg}, \vder, \vcum)$ over the ordinary shifted
integro-differential algebra $(\galg, \der, \cum)$ can also be characterized in terms of a
\emph{universal mapping property}. First we encapsulate the minimal requirements for adjoining a
family of distributions~$\delta_a \; (a \in K)$ to the given integro-differential
algebra~$(\galg, \der, \cum)$. Algebraically, they are characterized by the sifting
property~\eqref{eq:extract-value}, the integro-differential relation
\begin{equation*}
  \delta^{(k)} \overset{\scriptstyle\vder}{\underset{\text{\tiny$\vcum$}}{\rightleftarrows}} \delta^{(k+1)} 
\end{equation*}
for~$k \ge 0$, and the stipulation that~$\delta_a$ has the Heaviside function~$H_a$ as its
antiderivative with integration constant~$-\bar{H}(a)$. From the latter stipulation, it is clear
that the resulting structure must contain the Rota-Baxter submodule~$\pcwext{\galg}$. Finally, we
hold fast to the analysis tradition of barring multiplication of distributions (see
Remark~\ref{rem:no-diffring} for the algebraic view of this proscription). The universal property
stated below can then be construed as exhibiting the distribution
module~$(\distmod{\galg}, \vder, \vcum)$ as the most economic solution to the task of adjoining
Dirac distributions subject to these minimal requirements.

\begin{definition}
  \label{def:dirac-module}
  Let~$(\galg, \der, \cum)$ be an integro-differential algebra. An integro-differential module
  $(\dirm, \vderd, \vcumd)$ over~$\galg$ is called a \emph{Dirac module}
  if~$\pcwext{\galg} \hookrightarrow \dirm$ as Rota-Baxter modules such
  that~\eqref{eq:extract-value} holds and~$\delta_a := \vderd H_a$
  satisfies~$\vcumd\, \delta_a = H_a - \bar{H}(a)$ as well
  as~$\smash{\vcumd\, \delta_a^{(k+1)} = \delta_a^{(k)}}$, for all~$a \in K$ and~$k \ge 0$.
\end{definition}

\begin{proposition}
  \label{prop:dirac-module}
  The differential Rota-Baxter module~$(\distmod{\galg}, \vder, \vcum)$ is the universal Dirac
  module over $(\galg, \der, \cum)$ that extends $(\pcwext{\galg}, \cum)$ as a Rota-Baxter
  module. In other words, for every Dirac module~$\dirm$ there is a unique integro-differential
  morphism~$\Phi\colon \distmod{\galg} \to \dirm$ that respects the canonical embedding
  of~$\pcwext{\galg}$.
\end{proposition}

\begin{proof}
Let~$\kappa\colon \pcwext{\galg} \hookrightarrow \dirm$ be the embedding of Rota-Baxter modules from Definition~\ref{def:dirac-module}, and let~$u_{\pcwext{}}$, $u_{\distmod{}}$, $\iota$ be as in the diagram before Theorem~\ref{thm:dist-drbm}.

\begin{wrapfigure}{R}{0.5\textwidth}
  \centering\vspace{-1.6em}
$\xymatrix @M=0.6pc @=2.25pc @C=5pc %
 {\galg \ar_{u_{\pcwext{}}}[d] \ar^{u_{\distmod{}}}[drr] \ar@{-}[dr]\\
 \pcwext{\galg} \ar@{^(->}^(0.25){\iota}[rr] \ar@{^(->}^{\kappa}[]!<0.5ex,-2.5ex>;[drr] & \ar^{u_{\dirm}}[dr] & \distmod{\galg} \ar@{-->}^{\Phi}[d] \\
 && \dirm}$
\vspace{-1em}
\end{wrapfigure}

Furthermore, we will write~$u_{\dirm}$ for the structure map of the $\galg$-module~$\dirm$. We construct a morphism of integro-differential modules $\Phi\colon \distmod{\galg} \to \dirm$ that makes the right-hand diagram commute. It suffices to show~$\Phi \iota = \kappa$ since then~$\Phi u_{\distmod{}} = u_{\dirm}$ follows from the module structures~$\iota u_{\pcwext{}} = u_{\distmod{}}$ and $\kappa u_{\pcwext{}} = u_{\dirm}$.

If the required map~$\Phi$ exists, it must be $\galg$-linear and send~$(\iota H_a)^{(k)}$
to~$(\kappa H_a)^{(k)}$. But this defines~$\Phi$ uniquely because $\distmod{\galg}$ is generated
by~$(\iota H_a)^{(k)}$ as an $\galg$-module. Defining
first~$\tilde{\Phi}\colon \galg \{H_a \mid K \}_1 \to \dirm$ by these requirements, it follows at
once that~$\tilde\Phi$ is in fact a morphism of differential $\galg$-modules. For seeing that it
lifts to a map~$\Phi\colon \distmod{\galg} \to \dirm$, we must show~$\tilde\Phi(Z) =
0$.
Since~$\Phi$ respects the derivation, it suffices to prove that~$\Phi$ annihilates the differential
generators~$f \, \delta_a - \ev_a(f) \, \delta_a$ or, more precisely, the corresponding elements
$u_{\distmod{}}(f) \, \iota(H_a)' - \ev_a(f) \, \iota(H_a)'$. But this follows immediately from the
sifting property~\eqref{eq:extract-value} of the Dirac module~$\dirm$.

We have now a differential morphism~$\Phi\colon \distmod{\galg} \to \dirm$ that clearly satisfies the required commutation property~$\Phi \iota = \kappa$. Moreover, it is clear from the construction that~$\Phi$ is unique. Hence it only remains to prove that~$\Phi$ is also a morphism of Rota-Baxter algebras over~$\galg$. To this end, we show first that
\begin{equation}
\label{eq:univ-rb-pcwext}
\vcum_{\!\dirm} \Phi(f \, \iota H_a) = \Phi \, \vcum (f \, \iota H_a).
\end{equation}
Note that the left-hand side may be written as~$\vcum_{\!\dirm} \kappa (f \, \iota H_a)$ since~$\Phi \iota = \kappa$. Since by hypothesis we have~$\pcwext{\galg} \hookrightarrow \dirm$ as Rota-Baxter $\galg$-modules, we may now apply~$\vcum_{\!\dirm} \kappa = \kappa \cum$ and then expand the integral~$\cum$ of~$\pcwext{\galg}$ to obtain
\begin{equation*}
  \kappa \Big( (\cum_{\!\pospart{a}} f) \, H_a-(\cum_{\!\negpart{a}}^0 f)\, \bar{H_a} \Big) = \Phi \Big( (\cum_{\!\pospart{a}} f) \, \iota H_a-(\cum_{\!\negpart{a}}^0 f)\, \iota \bar{H_a} \Big)
\end{equation*}
for the left-hand side of~\eqref{eq:univ-rb-pcwext}, using again~$\Phi \iota = \kappa$ for the last step. Recalling that~$\vcum$ on~$\distmod{\galg}$ was defined as an extension of~$\cum$ on~$\distmod{\galg}$, this yields the right-hand side of~\eqref{eq:univ-rb-pcwext}. It remains to prove
\begin{equation}
\label{eq:univ-rb-dirac}
  \vcum_{\!\dirm} \Phi(f \, \delta_a^{(k)}) = \Phi \, \vcum (f \, \delta_a^{(k)})
\end{equation}
for all~$k \ge 0$. By the sifting property~\eqref{eq:extract-value}, valid in~$\distmod{\galg}$ as well as~$\dirm$, we may replace~$f$ by~$\ev_a(f)$ on both sides of~\eqref{eq:univ-rb-dirac}. Hence we may set~$f=1$ for the proof of~\eqref{eq:univ-rb-dirac}. For~$k=0$, we use the antiderivative relation of the Dirac module~$\dirm$ in the precise form~$\vcum (\kappa H_a)' = \kappa H_a - \bar{H}(a)$ to obtain
\begin{equation*}
  \vcum_{\!\dirm} \Phi \delta_a = \vcum (\kappa H_a)' = \kappa \big( H_a - \bar{H}(a) \big) = \Phi (\iota H_a - \bar{H}(a)) = \Phi (\vcum \delta_a)
\end{equation*}
as required. For~$k > 0$, Equation~\eqref{eq:univ-rb-dirac} follows immediately from~$\vcum_{\!\dirm} (\kappa H_a)^{(k)} = (\kappa H_a)^{(k-1)}$, which holds since~$\dirm$ is a Dirac module.
\end{proof}

As in the piecewise extension~$(\pcwext{\galg}, \cum)$, we can also provide \emph{shifted
  evaluations} on the distribution module~$(\distmod{\galg}, \vcum)$ if we have a shift
map~$S\colon K \to \Aut(\galg)$ on the ground algebra~$(\galg, \cum)$. Then we
define~$\mshift\colon K \to \Aut(\distmod{\galg})$ by extending~$\mshift_c \, H_a = H_{a-c}$
and~$\mshift_c \, \delta_a = \delta_{a-c} \; (a, c \in K)$ through linearity and
multiplicativity. It is immediate that~$\mshift$ is a shift map on the distribution
module~$\distmod{\galg}$. The latter is an ordinary integro-differential module if~$\galg$ is an
ordinary integro-differential algebra (Proposition~\ref{prop:intdiff-mod}), hence we get shifted
evaluations on~$\distmod$ by setting~$\mev_c := \mevl \circ \mshift_c$. Clearly, this
yields~$\mev_c H_a = \bar{H}(a-c)$ and~$\mev_c \delta_a^{(k)} = 0$ on the generators as
per~\eqref{eq:mod-eval}. As usual we write~$\vcum_b \: (b \in K)$ for the resulting shifted
Rota-Baxter operators.

\begin{theorem}
  \label{thm:intdiff-mod-shifted}
  If~$(\galg,\der, \cum)$ is an ordinary shifted integro-differential algebra,
  $(\distmod{\galg}, \vder, \vcum)$ is an ordinary shifted integro-differential module
  over~$\galg$. Its shifted Rota-Baxter operators are given by the recursion~\eqref{eq:rb-dist},
  with~$\vcum$ replaced by~$\vcum_b$, and by the base case~\eqref{eq:dirac-ader}, with~$\bar{H}(a)$
  replaced by~$\bar{H}(a-b)$ or~$H(a)$ replaced by~$H(a-b)$.
\end{theorem}

\begin{proof}
  The recursive description of the shifted Rota-Baxter operators follows immediately from the
  definition~$\vcum_b := (1-\mev_b) \cum$.  In view of Proposition~\ref{prop:intdiff-mod}, it then
  remains to prove the compatibility relations~$[\mshift_c, \vcum] = \mev_c \vcum$ and
  $[\mshift_c, \vder] = 0$.  Let us start with the former.

  Since~$\mshift_c$ and~$\vcum$ as well as~$\mev_c$ agree
  on~$\pcwext{\galg} \subset \distmod{\galg}$ by definition, it suffices to consider elements of the
  form~$f \delta_a^{(k)} \; (k \ge 0)$. We apply induction on~$k$. For the base case~$k = 0$, we
  obtain~$\ev_a(f) \, \big(\bar{H}(a-c) - \bar{H}(a) \big)$ for both left-hand and right-hand side
  of the relation~$[\mshift_c, \vcum] = \mev_c \vcum$ applied to~$f \delta_a$.  Now assume the
  relation for all~$f \delta_a^{(k)}$ with fixed~$k \ge 0$; we must show it
  for~$f \delta_a^{(k+1)}$. A straightforward computation, using the induction hypothesis
  on~$\vcum f' \delta_a^{(k)}$, yields~${} -\mev_c \vcum f' \delta_a^{(k)}$ for both sides
  of~$[\mshift_c, \vcum] = \mev_c \vcum$ as applied to~$f \delta_a^{(k+1)}$.

  Let us now turn to the commutation identity~$\mshift_c \vder = \vder \mshift_c$. Since~$\galg$ is
  a shifted integro-differential algebra by hypothesis, we need only consider elements of the
  form~$f H_a^{(k)} \: (k \ge 0)$. For those one obtains
  indeed~$\mshift_c \vder \, f H_a^{(k)} = \vder \mshift_c \, f H_a^{(k)} = S_c(f') \,
  \delta_{a-c}^{(k)} + S_c(f) \, H_{a-c}^{(k+1)}$,
  making use of the commutation identity on~$\galg$.
\end{proof}

\def\vvcum{\text{$\vcum$}}
\def\shiftarrow{\ar^*-<2.5ex>{^{\sim}}_*+<1.25ex>{_{{S_{c}}}}[r]}
\def\modwithdots#1{\stackrel{\vdots\strut}{\distmod{\galg}_#1^{(2)}}}
\begin{wrapfigure}[14]{r}{0.35\textwidth}
\centering\vspace{-2em}

$\xymatrix @M=0.6pc @=2.25pc @C=5pc {
  \modwithdots{a} \ar@<0.5ex>^{\vvcum}[d] \shiftarrow & 
  \modwithdots{b} \ar@<0.5ex>^{\vvcum}[d]\\
  \distmod{\galg}_a^{(1)} \shiftarrow \ar@<0.5ex>^{\vvcum}[d] \ar@<0.5ex>^{\vder}[u] &
 \distmod{\galg}_b^{(1)} \ar@<0.5ex>^{\vvcum}[d] \ar@<0.5ex>^{\vder}[u]\\
  \distmod{\galg}_a^{(0)} \shiftarrow \ar@<0.5ex>^{\vder}[u] & \distmod{\galg}_b^{(0)}
  \ar@<0.5ex>^{\vder}[u]\\
  \space}$
\end{wrapfigure}
 
It is gratifying that all the required properties of the ground algebra~$\galg$ are inherited by the
module~$\distmod{\galg}$: the integro-differential structure, ordinariness, and the shift structure.
We end this section by endowing the distribution module~$\distmod{\galg}$ with an \emph{ascending
  filtration}. Indeed, let us start by writing $\distmod{\galg}_a$ for the differential
$\galg$-submodule generated by~$H_a$. By~\eqref{eq:dist-can}, its elements have the canonical
form~$f + f_a H_a + \sum_k \lambda_{a,k} \, \smash{\delta_a^{(k)}}$ with~$f, f_a \in \galg$
and~$\lambda_{a,k} \in K$. A glance at~\eqref{eq:rb-dist} confirms at once that such elements are
also closed under the Rota-Baxter operator, so we have a differential Rota-Baxter
submodule~$(\distmod{\galg}_a, \vder, \vcum)$ and indeed a direct
sum~$\distmod{\galg} = \bigoplus_a \distmod{\galg}_a$ of differential Rota-Baxter submodules. The
$\distmod{\galg}_a$ are of course not shifted submodules, but the shift map restricts to
isomorphisms~$S_c\colon \distmod{\galg}_a \overset{\sim}{\to} \distmod{\galg}_b$, for any~$c \in K$
and~$b := a-c$. Next we define~$\distmod{\galg}_a^{(k)} \subseteq \distmod{\galg}_a$ as the
$\galg$-submodule generated by all~$H_a^{(j)}$ with~$j \le k$; note also that the piecewise
extension is given by~$\pcwext{\galg} = \bigoplus_a \distmod{\galg}_a^{(0)}$. It is obvious
that~$\vder$ maps~$\distmod{\galg}_a^{(k)}$ to~$\distmod{\galg}_a^{(k+1)}$, and one see
from~\eqref{eq:rb-dist} that~$\vcum$ restricts to a map from~$\distmod{\galg}_a^{(k+1)}$
to~$\distmod{\galg}_a^{(k)}$. It is also clear that~$(\distmod{\galg}_a^{(k)})_{k \ge 0}$ forms an
ascending $\galg$-module filtration of~$\distmod{\galg}_a$. We conclude that
each~$\distmod{\galg}_a$ as well as the entire distribution module~$\distmod{\galg}$ is a filtered
differential Rota-Baxter module (see the figure above). Moreover, the restricted shift
maps~$S_{c}\colon \distmod{\galg}_a \overset{\sim}{\to} \distmod{\galg}_b$ restrict further
to~$S_{c}\colon \distmod{\galg}_a^{(k)} \overset{\sim}{\to} \distmod{\galg}_b^{(k)}$.

For some purposes one needs only a few Heavisides (and Diracs), rather than the whole
gamut~$H_a \; (a \in K)$; in the extreme case one gets the \emph{slim distribution
  module}~$\slimmod{\galg}$, which is differentially generated by a single Heaviside that we shall
denote by~$\hat{H}$, its derivative being written~$\hat{\delta} := \hat{H}'$. The whole construction
given in this section may obviously be repeated verbatim to obtain the differential Rota-Baxter
module~$\slimmod{\galg}$. Alternatively, one may achieve the same result by slimming the hierarchy
of the above figure, namely by setting~$\slimmod{\galg} = \distmod{\galg} / N_{\distmod{}}$
with $\hat{H} := H_0 + N_{\distmod{}}$, where~$N_{\distmod{}} \subset \distmod{\galg}$ is the
differential Rota-Baxter submodule generated by the set~$N := \{ H_a \mid a \in \nonzero{K} \}$.
Similarly, one gets the \emph{slim piecewise extension}
$\slimext{\galg} = \pcwext{\galg} / N_{\pcwext{}}$ where~$N_{\pcwext{}} \subset \pcwext{\galg}$ is
the ideal generated by~$N$. Obviously, we may view~$\slimmod{\galg}$ as a module
over~$\slimext{\galg}$. We shall need the slim distribution module~$\slimmod{\galg}$ and the slim
piecewise extension~$\slimext{\galg}$ in the next section for obtaining the bivariate ``diagonal''
distribution~$\delta(x-\xi)$. In fact, we shall only need the $K$-subspace generated by~$\hat{H}$
and its derivatives; let us denote this space by~$\slimmod{K} \subset \slimmod{\galg}$. Likewise, we
shall write~$\slimext{K} \subset \slimext{\galg}$ for the $K$-subalgebra generated by~$\hat{H}$
alone.

\section{Bivariate Distributions}
\label{sec:biv-distributions}

Since one of our main applications in Section~\ref{sec:appl} will be to provide an algebraic model
of the \emph{bivariate Green's function} corresponding to a given boundary problem, it is now
necessary to expand the distribution module~$\distmod{\galg}$. While the latter contains only
univariate Heavisides~$H(x-a)$ and their derivatives (with $a \in K$ fixed), we shall also need
their counterparts~$H(\xi-a)$ in another variable~$\xi$, and moreover the ``diagonal'' Heaviside
function~$H(x-\xi)$ with its derivatives.\footnote{\label{fn:biv-dist}Our present treatment of
  bivariate distributions is very limited. A more comprehensive algebraic theory will allow more
  general distributions, containing at least~$\delta(a_1 x_1 + \cdots + a_n x_n)$.  While such a
  development should properly be given in an LPDE context, we are here only interested in LODE
  boundary problems where the three distribution
  families~$\delta_a(x), \delta_a(\xi), \delta(x-\xi)$ turn out to be sufficient
  (Section~\ref{sec:appl}).} Let us first concentrate on the former.

We start with the tensor product~$\galg_2 := \galg \otimes_K \galg$, writing its
elements~$f_1 \otimes f_2$ as~$f_1(x) \, f_2(\xi)$.  Note that~$\galg_2$ is
an~$\galg$-\emph{bimodule} with two derivations and two Rota-Baxter operators
\begin{align*}
  \der_x (f_1 \otimes f_2) = (\der f_1) \otimes f_2,\qquad
  \der_\xi (f_1 \otimes f_2) = f_1 \otimes (\der f_2),\\
  \cum^x (f_1 \otimes f_2) = (\cum f_1) \otimes f_2,\qquad
  \cum^\xi (f_1 \otimes f_2) = f_1 \otimes (\cum f_2).
\end{align*}
We have two embeddings~$\iota_x, \iota_\xi\colon \galg \to \galg_2$ with~$\iota_x(f) = f \otimes 1$
and~$\iota_\xi(f) = 1 \otimes f$; we denote their images by~$\galg_x$ and~$\galg_\xi$,
respectively. For a ground element~$f\in \galg$, their embeddings are also written
as~$f(x) := \iota_x(f) \in \galg_x$ and~$f(\xi) := \iota_\xi(f) \in \galg_\xi$.

Note that both~$(\galg_2, \der_x, \cum^x)$ and~$(\galg_2, \der_\xi, \cum^\xi)$ are
\emph{integro-differential algebras} over~$K$, though not ordinary ones
since~$\ker{\der_x} = \galg_\xi$ and~$\ker{\der_\xi} = \galg_x$. In addition to the duplex
differential Rota-Baxter structure, $\galg_2$ has two shift
operators~$\bivsh{x}{a} (f_1 \otimes f_2) := (S_a f_1) \otimes f_2$
and~$\bivsh{\xi}{a} (f_1 \otimes f_2) := f_1 \otimes (S_a f_2)$.
If~$\tau\colon \galg_2 \to \galg_2$ is the usual exchange
automorphism~$\tau(f_1 \otimes f_2) = f_2 \otimes f_1$, the derivations, Rota-Baxter and shift
operators are conjugate under~$\tau$, meaning~$\der_\xi = \tau \der_x
\tau$,~$\:
\smash{\cum^\xi} = \tau \smash{\cum^x} \tau$, $\bivsh{\xi}{a} = \tau \bivsh{x}{a} \tau$.

\begin{definition}
  \label{def:pure-distmod}
  The \emph{pure distribution modules} are introduced by
  $\distmod_x{\galg} := \distmod{(\galg_2, \der_x, \smash{\cum^x})}$ and
  $\distmod_\xi{\galg} := \distmod{(\galg_2, \der_\xi, \smash{\cum^\xi})}$. We
  write~$H(x-a) \in \distmod_x{\galg}$ and~$H(\xi-a) \in \distmod_\xi{\galg}$ for the corresponding
  differential generators ($a \in K$).
\end{definition}

In this context, we will revive our abbreviations~$H(a-x) \in \distmod_x{\galg}$
and~$H(a-\xi) \in \distmod_\xi{\galg}$. By virtue of~$\delta = H'$, we have
likewise~$\delta(x-a) \in \distmod_x{\galg}$ and~$\delta(\xi-a) \in \distmod_\xi{\galg}$. Finally,
we define the action of~$\vder_x, \vcum^x, \bivmsh{x}{a}$ on~$\distmod_\xi{\galg}$ by regarding
the~$H(\xi-a)$ as constants, meaning we
set $\vder_x \, f \smash{H_a^{(k)}} := (\der_x f) \, H_a^{(k)}$,
$\: \vcum^x f H_a^{(k)} := (\cum^x f) \, \smash{H_a^{(k)}}$
and~$\bivmsh{x}{a} \, f \smash{H_a^{(k)}} := (\bivsh{x}{a} f) \, H_a^{(k)}$.  The action
of~$\vder_\xi, \vcum^\xi, \bivmsh{\xi}{a}$ on~$\distmod_x{\galg}$ is defined analogously. Altogether
we obtain the two duplex shifted differential Rota-Baxter modules
$(\distmod_x{\galg}, \vder_x, \vder_\xi, \vcum^x, \vcum^\xi)$
and~$(\distmod_\xi{\galg}, \vder_x, \vder_\xi, \vcum^x, \vcum^\xi)$. Their induced evaluations are
written as~$\mevl_x := 1 - \vcum^x \, \vder_x$ and~$\mevl_\xi := 1 - \vcum^\xi \, \vder_\xi$, along
with the shifted versions~$\bivmev{x}{a} := \mevl_x \, \bivmsh{x}{a}$
and~$\bivmev{\xi}{a} := \mevl_\xi \, \bivmsh{\xi}{a}$.

Note that both pure distribution modules contain the corresponding piecewise extension
algebras~$\pcwext_x{\galg} \subset \distmod_x{\galg}$
and~$\pcwext_\xi{\galg} \subset \distmod_\xi{\galg}$.  These rings can be combined into the
\emph{bivariate piecewise
  extension}~$\pcwext_{x\xi}{\galg} := \pcwext_x{\galg} \otimes_\galg \pcwext_\xi{\galg}$, which is
useful for representing the characteristic functions\footnote{As in~\cite{RosenkranzSerwa2015} we
  use the \emph{Iverson bracket notation}~\cite[\S2.2]{GrahamKnuthPatashnik1994} for characteristic
  functions of intervals.} of a rectangle~$(x, \xi) \in [a, b] \times [c, d]$ by the tensor product
$[a \le x \le b] \otimes [c \le \xi \le d]$ with Heaviside
factors~$[a \le x \le b] := H(x-a) \, H(b-x) \in \pcwext_x{\galg}$
and~$[c \le \xi \le d] := H(\xi-c) \, H(d-\xi) \in \pcwext_\xi{\galg}$; this is needed in
Section~\ref{sec:appl}. By analogy to the situation in~$\galg_2$, we shall drop the $\otimes$
symbol, thus writing~$H(x-a) \, H(\xi-b)$ for what is strictly
speaking~$H_a \otimes H_b \in \pcwext_{x\xi}{\galg}$. Note that $\pcwext_{x\xi}{\galg}$ is a duplex
shifted differential Rota-Baxter algebra over~$\galg_2$, analogous to the univariate case.

We will now combine the univariate distribution modules~$\distmod_x{\galg}$
and~$\distmod_\xi{\galg}$ along with the bivariate piecewise extension~$\pcwext_{x\xi}{\galg}$ into
a single module. To this end, note that both~$\distmod_x{\galg} \otimes_\galg \pcwext_\xi{\galg}$
and~$\pcwext_x{\galg} \otimes_\galg \distmod_\xi{\galg}$ contain isomorphic copies of the
$\galg$-submodule~$\pcwext_{x\xi}{\galg}$ with which they are identified. With this identification,
we form the direct sum of~$\distmod_x{\galg} \otimes_\galg \pcwext_\xi{\galg}$
and~$\pcwext_x{\galg} \otimes_\galg \distmod_\xi{\galg}$ which we call the \emph{tensorial
  distribution module} and denote by~$\distmod_{x\xi}{\galg}$. Regarding the ``foreign'' tensor
factors as constants (see the comments after Definition~\ref{def:pure-distmod}), all structures
combine into a duplex shifted differential Rota-Baxter module
$(\distmod_{x\xi}{\galg}, \vder_x, \vder_\xi, \vcum^x, \vcum^\xi)$ over~$\galg_2$, which is
simultaneously a module over~$\pcwext_{x\xi}{\galg}$. Thus far, the situation is parallel to that of
Theorem~\ref{thm:dist-drbm}.

The algebraic description of the \emph{diagonal Heavisides}~$H(x-\xi)$ and diagonal
Diracs~$\delta(x-\xi)$ is somewhat more complicated. At the level of elements, we insert them
essentially by tacking a slim distribution module on top of~$\distmod_{x\xi}{\galg}$. However, the
crucial question is how to combine the diagonal Heavisides with the univariate ones to form a
uniform Rota-Baxter structure on the resulting module. The required relation is easy to find if we
want to keep touch with analysis. Indeed, for a moment let us think of~$K = \RR$ with a
fixed~$a \in \RR$ and variables~$x, \xi$ ranging over~$\RR$. We have~$x \ge a \land x \ge \xi$ iff
$(x \ge a \land a \ge \xi) \lor (x \ge \xi \land a \le \xi)$ since we may split the cases~$a > \xi$
and~$a < \xi$, the remaining possibility~$a = \xi$ holding in both cases above. Translating into
Heavisides, this yields
\begin{equation}
  \label{eq:biv-heaviside}
  H(x-a) \, H(x-\xi) = H(x-a) \, H(a-\xi) + H(x-\xi) \, H(\xi-a)
\end{equation}
or~$H_a(x) \, \hat{H} = H_a(x) \, \bar{H}_a(\xi) + H_a(\xi) \, \hat{H}$ in our algebraic
language. We can formulate this into a proper definition of the module providing diagonal Heavisides
and Diracs. Here we must take recourse to our earlier interpretations
$H_a(x) := H_a \otimes 1 \in \pcwext_{x\xi}{\galg}$
and~$H_a(\xi) := 1 \otimes H_a \in \pcwext_{x\xi}{\galg}$.

\begin{definition}
  \label{def:diagdist}
  Let~$\hat{Z}$ be the $\pcwext_{x\xi}{\galg}$-submodule
  of~$\pcwext_{x\xi}{\galg} \otimes_K \slimmod{K}$ that is generated by the
  set $\{ \big( H_a(x) - H_a(\xi) \big) \hat{H} - H_a(x) \, \bar{H}_a(\xi) \mid a \in K \}$. Then
  the $\pcwext_{x\xi}{\galg}$-module
  \begin{equation*}
    \distmod_{x-\xi}{\galg} := \frac{\pcwext_{x\xi}{\galg} \otimes_K \slimmod{K}}{\hat Z}
  \end{equation*}
  is called the \emph{diagonal distribution module}. We shall denote the (congruence class of) its
  slim generator~$\hat{H} \in \slimmod{K}$ by~$H(x-\xi)$, and its
  derivative~$\hat{\delta} \in \slimmod{K}$ by~$\delta(x-\xi)$. Analogously to the univariate
  case, we set also~$H(\xi-x) := 1 - \hat{H}$.
\end{definition}

It should also be emphasized that the submodule~$\hat{Z}$ is \emph{not differentially generated}. In
other words, one is not supposed to differentiate the relation~\eqref{eq:biv-heaviside} as this
would once again lead to inconsistencies. (The situation is completely analogous to the univariate
case where one is not supposed to differentiate the relation~$H_a^2 = H_a$; confer
Remark~\ref{rem:no-diffring}.)

At this point we have two $\pcwext_{x\xi}{\galg}$-modules~$\distmod_{x\xi}{\galg}$
and~$\distmod_{x-\xi}{\galg}$. Since~$\galg_2 \subset \pcwext_{x\xi}$, we may also view them
as~$\galg_2$-modules. It is easy to see that as such they are \emph{free modules} just
as~$\pcwext_{x\xi}$ itself is free as an~$\galg_2$-module. Indeed, the bivariate piecewise
extension~$\pcwext_{x\xi}$ has the $\galg_2$-basis
$\mathcal{B} := \{ 1, H_a(x), H_a(\xi), H_a(x) \, H_b(\xi) \mid a, b \in K \}$, while the tensorial
distribution module~$\distmod_{x\xi}{\galg}$
has~$\mathcal{B}_{x\xi} := \mathcal{B} \cup \{ H_a(x) \, \delta^{(n)}(b-\xi), H_a(\xi) \,
\delta^{(n)}(b-x) \mid a, b \in K; n \in \NN \}$
as an $\galg_2$-basis. Finally, using the relation~\eqref{eq:biv-heaviside}, the diagonal
distribution module~$\distmod_{x-\xi}{\galg}$ can be equipped with the ``left-focused''
$\galg_2$-basis
$\mathcal{B}_x := \mathcal{B} \cup \{ H^{(n)}(x-\xi), H_a(x) \, H^{(n)}(x-\xi) \mid a \in K, n \ge 0
\}$
or with its ``right-focused'' companion
$\mathcal{B}_\xi := \mathcal{B} \cup \{ H^{(n)}(x-\xi), H_a(\xi) \, H^{(n)}(x-\xi) \mid a \in K, n
\ge 0 \}$.

We can now put together the tensorial and the diagonal distribution module to obtain the full
\emph{bivariate distribution module}. The latter is already equipped with a duplex differential
Rota-Baxter structure, which we shall soon extend to the whole bivariate distribution module in such
a way that~$\distmod_{x\xi}{\galg}$ but not~$\distmod_{x-\xi}{\galg}$ will occur as a duplex
differential Rota-Baxter submodule.

\begin{definition}
  \label{def:bivdistmod}
  The \emph{bivariate distribution module} is given
  by~$\bivmod{\galg} := \distmod_{x\xi}{\galg} \oplus \distmod_{x-\xi}{\galg}$, as a direct sum of
  $\pcwext_{x\xi}{\galg}$-modules.
\end{definition}

Let us first extend the two
\emph{derivations}~$\vder_x, \vder_\xi\colon \distmod_{x\xi} \to \distmod_{x\xi}$ to the diagonal
distribution module~$\distmod_{x-\xi}{\galg}$. For defining~$\vder_x$ we use the
$\galg_2$-basis~$\mathcal{B}_\xi$ to set~$\vder_x H^{(n)}(x-\xi) := H^{(n+1)}(x-\xi)$. Regarding
the~$H_a(\xi)$ as constants, the map~$\vder_x\colon \distmod_{x-\xi}{\galg} \to \bivmod{\galg}$ is
uniquely determined as an extension
of~$\vder_x\colon \pcwext_{x\xi} \to \distmod_{x\xi} \subset \bivmod{\galg}$.  Analogously, the
map~$\vder_\xi\colon \distmod_{x-\xi}{\galg} \to \bivmod{\galg}$ is introduced as an extension
of~$\der_\xi\colon \pcwext_{x\xi} \to \distmod_{x\xi} \subset \bivmod{\galg}$
with~$\vder_\xi H^{(n)}(x-\xi) := - H^{(n+1)}(x-\xi)$, via the \mbox{$\galg_2$-basis}
$\mathcal{B}_x$. The resulting
maps~$\vder_x, \vder_\xi\colon \distmod_{x-\xi}{\galg} \to \bivmod{\galg}$ are now combined with the
existing derivations~$\vder_x, \vder_\xi\colon \distmod_{x\xi} \to \bivmod{\galg}$ on the tensorial
distribution modules into the canonical derivations on the direct
sum~$\vder_x, \vder_\xi\colon \bivmod{\galg} \to \bivmod{\galg}$. It is clear
that~$(\bivmod{\galg}, \vder_x, \vder_\xi)$ is then a duplex differential module just
over~$\galg_2$, although~$\bivmod{\galg}$ is a module over the duplex differential
ring~$(\pcwext_{x\xi}{\galg}, \der_x, \der_\xi)$; this is completely analogous to the univariate
structures~$(\distmod{\galg}, \vder)$ and~$(\pcwext{\galg}, \der)$.

For defining\footnote{The standard approach uses the
  isomorphism~$\distmod_2{\galg} \cong \big( \distmod_{x\xi}{\galg} \oplus \hdistmod_{x-\xi}{\galg}
  \big) / (0 \oplus \hat{Z})$
  from~\cite[\S VI.6.14]{MacLaneBirkhoff1968}, defining a Rota-Baxter operator~$\wcum^x$
  on~$\hdistmod_{x-\xi}{\galg} := \pcwext_{x\xi}{\galg} \otimes_K \slimmod{K}$ and hence on the
  numerator, then proving that~$0 \oplus \hat{Z}$ is invariant under~$\wcum^x$ so that~$\vcum^x$ is
  the induced map on the quotient. We bypass this laborious procedure by using the
  bases~$\mathcal{B}_x$ and~$\mathcal{B}_\xi$, proving the Rota-Baxter axiom directly in
  Proposition~\ref{prop:pcw-biv}.} the \emph{Rota-Baxter operators}~$\vcum^x$ and~$\vcum^\xi$
on~$\bivmod{\galg}$, it suffices to define them on the diagonal summand~$\distmod_{x-\xi}{\galg}$,
using the existing Rota-Baxter operators on the tensorial summand $\distmod_{x\xi}{\galg}$. Thus we
define first~$\vcum^x\colon \distmod_{x-\xi}{\galg} \to \bivmod{\galg}$ using the
$\galg$-basis~$\mathcal{B}_\xi$ of~$\distmod_{x-\xi}{\galg}$,
and~$\vcum^\xi\colon \distmod_{x-\xi}{\galg} \to \bivmod{\galg}$ using~$\mathcal{B}_x$. For the
former, we view~$f(\xi) \, H(\xi-a)$ with~$f(\xi) \in \galg_\xi$ as constants, for the
latter~$g(x) \, H(x-a)$ with~$g(x) \in \galg_x$. Hence it suffices to define~$\vcum^x$ for elements
of the form~$f(x) \, H^{(n)}(x-\xi)$ with~$f(x) \in \galg_x$ and likewise~$\vcum^\xi$ for
elements~$g(\xi) \, H^{(n)}(x-\xi)$ with~$g(\xi) \in \galg_\xi$. As in the univariate case, we give
a recursive definition. In analogy to~\eqref{eq:pcw-def1}--\eqref{eq:pcw-def2} in their second form,
the base case~$n = 0$ is
\begin{align}
  \label{eq:biv-pcw-def1}
  \vcum^x f(x) \, H(x-\xi) &:= \Big( \cum_{\!\xi}^x \, f(x) \Big) \, H(x-\xi) + \Big( \cum^\xi
                             f(x) \Big) \bar{H}_0(\xi),\\
  \label{eq:biv-pcw-def2}
  \vcum^\xi\, g(\xi) \, H(x-\xi) &:= \Big( \cum_{\!x}^\xi \, g(\xi) \Big) \, H(x-\xi) + \Big(
                                   \cum^x g(\xi) \Big) H_0(x),
\end{align}
where we abbreviate~$\cum^\xi f(x) := \tau \big( \cum^x f(x) \big) \in \galg_\xi$ and
$\cum^x g(\xi) := \tau \big( \cum^\xi g(\xi) \big) \in \galg_x$. By our usual convention, we have
then~$\smash{\cum_{\!\xi}^x} \, f(x) = (1-\tau) \, \smash{\cum^x} f(x)$
and~$\smash{\cum_{\!x}^\xi} \, g(\xi) = (1-\tau) \, \smash{\cum^\xi} g(\xi)$. Here it is important
to distinguish carefully~$H(x-\xi) = \hat{H} \in \distmod_{x\xi}{\galg}$
and~$H(\xi-x) = 1-\hat{H} \in \distmod_{x\xi}{\galg}$
from~$H_0(\xi) = 1 \otimes H_0 \in \pcwext_{x\xi}{\galg} \subset \distmod_{x\xi}{\galg}$
and~$H_0(x) = H_0 \otimes 1 \in \pcwext_{x\xi}{\galg} \subset \distmod_{x\xi}{\galg}$.  Furthermore,
it should be noted that while the $x$-integral~\eqref{eq:biv-pcw-def1} corresponds
to~\eqref{eq:pcw-def1}, the $\xi$-integral~\eqref{eq:biv-pcw-def2} corresponds
to~\eqref{eq:pcw-def2} since~$H(x-\xi)$ behaves like~$\bar{H}_x(\xi)$ from the $\xi$ perspective;
this is the reason for having~$H_0(x)$ in~\eqref{eq:biv-pcw-def2} as opposed to~$\bar{H}_0(\xi)$
in~\eqref{eq:biv-pcw-def1}.

Before returning to the definition
of~$\vcum^x, \vcum^\xi\colon \distmod_{x\xi}{\galg} \to \bivmod{\galg}$, we present a bivariate
analog of Proposition~\ref{prop:pcw}, introducing the \emph{diagonal piecewise extension} as the
$\pcwext_{x\xi}{\galg}$-submodule
\begin{equation}
  \label{eq:diagpcwext}
  \pcwext_{x-\xi}{\galg} := \frac{\pcwext_{x\xi}{\galg} \otimes_K K \! \hat{H}}{\hat Z}
  \: \subset \: \distmod_{x-\xi}{\galg},
\end{equation}
where $K \! \hat{H} \subset \slimmod{K}$ is the $K$-subspace generated
by~$\hat{H} = H(x-\xi) \in \slimmod{K}$. It is clear that~$\pcwext_{x-\xi}{\galg}$ is free
over~$\galg_2$ with
basis~$\mathcal{B}_x^0 := \{ H(x-\xi), H_a(x) \, H(x-\xi) \mid a \in K \} \subset \mathcal{B}_x$ or
again
alternatively $\mathcal{B}_\xi^0 := \{ H(x-\xi), H_a(\xi) \, H(x-\xi) \mid a \in K \} \subset
\mathcal{B}_\xi$.

We note that both~$\pcwext_{x\xi}{\galg}$ and~$K \! \hat{H}$ are endowed with a multiplication but
unlike the former, $K \! \smash{\hat{H}}$ is a \emph{nonunitary} $K$-algebra. In fact, its
unitarization is just the slim piecewise extension $\smash{\slimext{K}} = K \oplus K \! \hat{H}$. At
any rate, the numerator of~\eqref{eq:diagpcwext} is naturally a nonunitary $K$-algebra, and it turns
out that the whole quotient module is as well.

\begin{lemma}
  The diagonal piecewise extension~$\pcwext_{x-\xi}{\galg}$ is a nonunitary $K$-algebra.
\end{lemma}
\begin{proof}
  It suffices to prove that~$\hat{Z}$ is an ideal in the nonunitary
  ring~$\pcwext_{x\xi}{\galg} \otimes_K K \! \hat{H}$. Hence let us take an arbitrary
  $\pcwext_{x\xi}{\galg}$-generator~$\gamma_a := H_a(x) \, \hat{H} - H_a(\xi) \, \hat{H} - H_a(x) \,
  \bar{H}_a(\xi) \in \hat{Z}$
  and show~$\gamma_a \hat{Z} \subseteq \hat{Z}$.
  Since~$\pcwext_{x\xi}{\galg} \otimes_K K \! \hat{H}$ is generated over~$\pcwext_{x\xi}{\galg}$
  by~$\hat{H}$, we need only verify~$\gamma_a \hat{H} \in \hat{Z}$. One checks immediately that
  $\gamma_a \smash{\hat{H}} = - \bar{H}_a(x) \, H_a(\xi) \, \smash{\hat{H}} = \gamma_a \bar{H}_a(x)
  \in \smash{\hat{Z}}$.
\end{proof}

\begin{wrapfigure}[6]{r}{0.36\textwidth}
\centering\vspace{-1.5em}

$\xymatrix @M=0.5pc @=2.5pc @C=-0.5pc {
  \distmod_2{\galg} & = & \distmod_{x\xi}{\galg} & \oplus & \distmod_{x-\xi}{\galg}\\
  \pcwext_2{\galg}\ar@{^(->}[u] & = & \pcwext_{x\xi}{\galg}\ar@{^(->}[u] & \oplus & 
  \pcwext_{x-\xi}{\galg}\ar@{^(->}[u] }$
\end{wrapfigure}

In analogy to Definition~\ref{def:bivdistmod}, the \emph{bivariate piecewise
  \mbox{extension}}~$\pcwext_2{\galg} := \pcwext_{x\xi}{\galg} \oplus \pcwext_{x-\xi}{\galg}$ is
a~$\pcwext_{x\xi}{\galg}$-module consisting of tensorial and diagonal components. But we may also
view~$\pcwext_{x-\xi}{\galg}$ as a nonunitary algebra over~$\pcwext_{x\xi}{\galg}$, and as such its
unitarization is~$\pcwext_2{\galg}$. Therefore the latter is \mbox{naturally} a (unitary)
$\pcwext_{x\xi}{\galg}$-algebra. It is free over~$\galg_2$ with
basis~$\mathcal{B} \cup \mathcal{B}_x^0 \subset \mathcal{B}_x$ or equivalently with
basis~$\mathcal{B} \cup \mathcal{B}_\xi^0 \subset \mathcal{B}_\xi$. Moreover, it is clear
that~\eqref{eq:biv-pcw-def1}--\eqref{eq:biv-pcw-def2} restrict to yield
operators~$\smash{\vcum^x}, \smash{\vcum^\xi}\colon \pcwext_2{\galg} \to \pcwext_2{\galg}$, which
turn out to be Rota-Baxter operators. Thus we obtain the following partial bivariate analog to
Proposition~\ref{prop:pcw}.

\begin{proposition}
  \label{prop:pcw-biv}
  Let~$(\galg, \cum)$ be an ordinary shifted Rota-Baxter algebra over an ordered
  field~$K$. Then~$(\pcwext_2{\galg}, \vcum^x, \vcum^\xi)$ is a duplex Rota-Baxter algebra that
  extends~$(\galg_2, \cum^x, \cum^\xi)$.
\end{proposition}

\begin{proof}
  The extension property is immediate from the definition of~$\vcum^x$ and~$\vcum^\xi$.
  Since~\eqref{eq:biv-pcw-def1}--\eqref{eq:biv-pcw-def2} are symmetric under exchange of~$x$
  and~$\xi$, it suffices to verify the Rota-Baxter axiom for~$\vcum^x$, say. For this purpose it
  will be proficient to view~$\pcwext_2{\galg}$ as a free module over~$\pcwext_\xi{\galg}$ having
  the basis~$\{ b(x), b(x) \, H_a(x), b(x) \, H(x-\xi) \mid a \in K, \; b(x) \in \mathcal{B}_0 \}$,
  where~$\mathcal{B}_0$ is an arbitrary but fixed \mbox{$K$-basis} of~$\galg_x$. Since the
  Rota-Baxter~\eqref{eq:rb-axiom} axiom is bilinear and symmetric in the arguments~$f$ and~$g$, it
  suffices to let both arguments range over the basis given above. This leads to $3 + 2 + 1 = 6$
  cases. The three cases without~$H(x-\xi)$ are covered since~$(\pcwext_{x\xi}{\galg}, \vcum^x)$ is
  a Rota-Baxter ring. Hence it suffices to consider~$f = f(x) \, H(x-\xi)$ for
  arbitrary~$f(x) \in \galg_x$ in conjunction with the
  cases~$g = g(x), g(x) \, H_a(x), g(x) \, H(x-\xi)$ for arbitrary~$g(x) \in \galg_x$. We can
  subsume the first two cases for~$g$ by allowing an arbitrary~$g(x) \in \pcwext_x{\galg}$.

  Let us start with the diagonal-univariate case~$f = f(x) \, H(x-\xi)$ and~$g \in \pcwext_x{\galg}$
  of the Rota-Baxter axiom~\eqref{eq:rb-axiom}. Using the defining equation~\eqref{eq:biv-pcw-def1},
  the left-hand side~$\vcum^x f \cdot \vcum^x g$ is given by
  \begin{equation}
    \label{eq:rb-bivpcwext-lhs}
    \Big( \cum_{\!\xi}^x \, f(x) \cdot \cum^x g(x) \Big) \, H(x-\xi)
    + \Big( \cum^\xi f(x) \cdot \cum^x g(x) \Big) \, \bar{H}_0(\xi)
  \end{equation}
  where~$\cum^x$ denotes the Rota-Baxter operator of~$\pcwext_x{\galg}$ and $\cum_{\!\xi}^x$ is as
  defined after~\eqref{eq:biv-pcw-def2}. Likewise, a single application of~\eqref{eq:biv-pcw-def1}
  determines the term~$\vcum^x f \, \vcum^x g$ on the right-hand side of~\eqref{eq:rb-axiom} as
  \begin{equation*}
     \Big( \cum_{\!\xi}^x \, f(x) \, \cum^x g(x) \Big) \, H(x-\xi)
     + \Big( \cum^\xi f(x) \, \cum^x g(x) \Big) \, \bar{H}_0(\xi) .
  \end{equation*}
  Thus it remains to compute the other term~$\vcum^x g \, \vcum^x f$ on the right-hand side
  of~\eqref{eq:rb-axiom}. Here we invoke~\eqref{eq:biv-pcw-def1} three times to get
  \begin{equation*}
      \Big( \cum_{\!\xi}^x \, g(x) \, \cum^x f(x)
      - \cum^\xi f(x) \cdot \cum_{\!\xi}^x \, g(x) \Big) \, H(x-\xi)
      + \Big( \cum^\xi g(x) \, \cum^x f(x)
      + \cum^\xi f(x) \cdot \cum_{\!\xi}^x \, g(x) \Big) \, \bar{H}_0(\xi)
  \end{equation*}
  Adding the last two equations yields~\eqref{eq:rb-bivpcwext-lhs} as one sees by a straightforward
  calculation using the Rota-Baxter axiom of~$\pcwext_x{\galg}$, suitably combined with the action
  of~$\tau$. Hence the Rota-Baxter axiom~\eqref{eq:rb-axiom} is verified in this case.

  We are left with the diagonal-diagonal case~$f = f(x) \, H(x-\xi)$ and~$g = g(x) \, H(x-\xi)$,
  which is of a more symmetric nature. Calculating according to~\eqref{eq:biv-pcw-def1}, we obtain
  for~$\vcum^x f \, \vcum^x g$ the expression
  \begin{equation*}
    \Big( \cum_{\!\xi}^x \, f(x) \, \cum^x g(x) - \cum_{\!\xi}^x \, f(x) \cdot \cum^\xi g(x)
    + \bar{H}_0(\xi) \, \cum_{\!\xi}^x \, f(x) \cdot \cum^\xi g(x) \Big) \, H(x-\xi)
     + \Big( \cum^\xi f(x) \, \cum^x g(x) \Big) \, \bar{H}_0(\xi) ,
  \end{equation*}
  which combined with its symmetric counterpart~$\vcum^x g \: \vcum^x f$ yields
  \begin{align*}
    \quad \Big( & \cum^x f(x) \cdot \cum^x g(x) - \cum^\xi f(x) \cdot \cum^\xi g(x) \Big) \,
                  H(x-\xi)  + \Big( \cum^\xi f(x) \cdot \cum^\xi g(x) \Big) \, \bar{H}_0(\xi)\\
                & {} + \Big( 2 \, \cum^\xi f(x) \cdot \cum^\xi g(x) - \cum^x f(x) \cdot \cum^\xi g(x) -
                  \cum^\xi f(x) \cdot \cum^x g(x) \Big) \, H_0(\xi) \, H(x-\xi) ,
  \end{align*}
  where we have again applied the Rota-Baxter axiom of~$\pcwext_x{\galg}$ for merging nested
  integrals. Expanding the product~$\vcum^x f \cdot \vcum^x g$ one confirms that this indeed
  coincides with the expression above, so the Rota-Baxter axiom~\eqref{eq:rb-axiom} is again
  verified.
\end{proof}

We return now to the definition of the Rota-Baxter operators~$\vcum^x$ and~$\vcum^\xi$ on the
bivariate distribution module~$\distmod_{x\xi}{\galg}$, which is in fact dictated by the Rota-Baxter
axiom for modules. Having settled the base case in~\eqref{eq:biv-pcw-def1}--\eqref{eq:biv-pcw-def2},
we apply the reasoning of Remark~\ref{rem:alt-der-distmod} to continue the definition by setting
\begin{align}
  \label{eq:biv-dist-def1}
  \vcum^x f(x) \, H^{(n+1)}(x-\xi) &:= f(x) \, H^{(n)}(x-\xi) - \vcum^x f'(x) \, H^{(n)}(x-\xi),\\
  \label{eq:biv-dist-def2}
  {}-\vcum^\xi\, g(\xi) \, H^{(n+1)}(x-\xi) &:= g(\xi) \, H^{(n)}(x-\xi)  -
                                           \vcum^\xi g'(\xi) \, H^{(n)}(x-\xi)
\end{align}
for all~$f(x) \in \galg_x, g(\xi) \in \galg_\xi$ and~$n \in \NN$. Note the distinct sign
in~\eqref{eq:biv-dist-def2}, due to the fact that~$\vder_\xi = - \vder_x$ on the diagonal
distribution module~$\distmod_{x-\xi}{\galg}$. We obtain now the following kind of analog to
Theorem~\ref{thm:intdiff-mod-shifted}.

\begin{theorem}
  \label{thm:biv-distmod}
  Let~$(\galg,\der, \cum)$ be an ordinary shifted integro-differential algebra. Then the bivariate
  distribution module~$(\distmod_2{\galg}, \vder_x, \vder_\xi, \smash{\vcum^x}, \smash{\vcum^\xi})$
  is a duplex differential Rota-Baxter module containing two isomorphic
  copies $(\distmod_x{\galg}, \vder_x, \smash{\vcum^x})$
  and~$(\distmod_\xi{\galg}, \vder_\xi, \smash{\vcum^\xi})$ of the given~$(\galg,\der, \cum)$. As a
  duplex Rota-Baxter module, $\distmod_2{\galg}$ extends~$\pcwext_2{\galg}$.
\end{theorem}

\begin{proof}
  It is obvious from the construction of~$\distmod_{x\xi}{\galg}$ that it contains the two
  isomorphic copies $(\distmod_x{\galg}, \vder_x, \smash{\vcum^x})$
  and~$(\distmod_\xi{\galg}, \vder_\xi, \smash{\vcum^\xi})$, so
  clearly~$\distmod_2{\galg} \supset \distmod_{x\xi}{\galg}$ contains them as well. Furthermore, the
  Rota-Baxter module extension~$\distmod_2{\galg} \supset \pcwext_2{\galg}$ is immediate from the
  definition of~$\vcum^x$ and~$\vcum^\xi$.

  By symmetry, it suffices to consider the other claims for~$\vcum^x, \vder^x$, say. As mentioned
  after Definition~\ref{def:bivdistmod}, $(\distmod_2, \vder_x, \vcum^x)$ is a differential module
  over~$\galg_2$ and hence over~$K$. We check now that it is also a Rota-Baxter module
  over~$\galg_2$, meaning it satisfies the module Rota-Baxter axiom
  \begin{equation}
    \label{eq:mod-rb-axiom}
    \cum^x f \cdot \vcum^x \phi = \vcum^x f \, \vcum^x \phi + \vcum^x (\cum^x f) \phi
  \end{equation}
  for all~$f \in \galg_2$ and~$\phi \in \distmod_2{\galg}$. Since both~$\cum^x$ and~$\vcum^x$
  treat~$\galg_\xi < \pcwext_\xi{\galg}$ as constants, it suffices to consider~$f \in \galg_x$.
  Moreover, we may also restrict ourselves
  to~$\phi \in \distmod_x{\galg} \oplus \distmod_{x-\xi}{\galg} < \distmod_2{\galg}$ since~$\vcum^x$
  treats~$\distmod_\xi{\galg}$ as constants (here we
  view~$\distmod_x{\galg} = \distmod_x{\galg} \otimes K < \distmod_{x\xi}{\galg}$). The first case
  $\phi \in \distmod_x{\galg}$ is already settled since we know from
  Theorem~\ref{thm:intdiff-mod-shifted}
  that~$(\distmod_x{\galg}, \vcum^x) \cong (\distmod{\galg}, \vcum)$ is a Rota-Baxter module. Thus
  remains to consider the diagonal case~$\phi \in \distmod_{x-\xi}{\galg}$, and we can use the
  $\galg_2$-basis~$\mathcal{B}_\xi$. By definition, $\vcum^x$ coincides for basis elements
  in~$\mathcal{B} \subset \mathcal{B}_\xi$ with the Rota-Baxter operator~$\cum^x$
  on~$\pcwext_{x\xi}{\galg}$; hence we may restrict ourselves to~$\phi = H^{(n)}(x-\xi)$
  and~$\phi = H_a(\xi) \, H^{(n)}(x-\xi)$ with arbitrary~$n \ge 0$ and~$a \in K$. But the latter
  case follows immediately from the former since the~$H_a(\xi)$ are constants for~$\vcum^x$. We are
  now left to prove~\eqref{eq:mod-rb-axiom} for~$f \in \galg_x$ and~$\phi = H^{(n)}(x-\xi)$, which
  we do by induction on~$n$. The base case~$n=0$ is covered by Proposition~\ref{prop:pcw-biv}, so we
  consider~\eqref{eq:mod-rb-axiom} with~$\phi = H^{(n+1)}(x-\xi)$ for the induction
  step. From~\eqref{eq:biv-dist-def1} we get~$\vcum^x \phi = H^{(n)}(x-\xi)$, so the first summand
  on the right of~\eqref{eq:mod-rb-axiom} cancels the second term of
  expanding~\eqref{eq:biv-dist-def1} with~$\cum^x f(x)$ in place of~$f(x)$; the remaining
  term~$\cum^x f(x) \cdot H^{(n)}(x-\xi)$ equals the left-hand side of~\eqref{eq:mod-rb-axiom}.
\end{proof}

Note that we have not set up shift operators on the diagonal distributions
of~$\distmod_{x-\xi}{\galg}$ since $\bivmsh{x}{a} H(x-\xi) = H(x-\xi-a)$ would take us outside
of~$\distmod_{x-\xi}{\galg}$. While it is certainly possible to set up a larger domain allowing this
(cf.\@ Footnote~\ref{fn:biv-dist}), we do not need it for our present purposes. However, we will
need \emph{evaluation operators}~$\mev^x_a$ and~$\mev^\xi_a$ on~$\distmod_{x-\xi}{\galg}$; since
such operators are already defined on~$\distmod_{x\xi}{\galg}$, this determines~$\mev^x_a$
and~$\mev^\xi_a$ on the bivariate distribution module~$\distmod_2{\galg}$ by linearity. The
intuitive idea is to define~$\mev^x_a\colon \distmod_{x-\xi}{\galg} \to \distmod_\xi{\galg}$ by
analogy to the univariate definition~$\ev_a \, H_\xi := \bar{H}(\xi-a) \in K$
for~$H_\xi \in \pcwext{\galg}$ given earlier (cf.\@ Proposition~\ref{prop:pcw} and the paragraph
above it): Since for evaluation it should not play a role whether one views~$\xi$ as a parameter or
as a variable, we can interpret~$H_\xi$ heuristically as~$H(x-\xi) \in \pcwext_{x-\xi}{\galg}$ and
the right-hand side as~$\bar{H}(\xi-a) = 1 \otimes \bar{H}_a \in \pcwext_\xi{\galg}$. For the
evaluation with respect to~$\xi$, the reasoning is analogous. Hence we give the definitions
\begin{equation}
  \label{eq:def-diag-eval-pcw}
  \mev^x_a \, H(x-\xi) := \bar{H}(\xi-a), \qquad
  \mev^\xi_a \, H(x-\xi) := H(x-a) .
\end{equation}
For evaluating diagonal Diracs, we use again the analogy to our earlier
definition~$\mev_a \, \delta_\xi^{(k)} := 0$ set up earlier (see the paragraph before
Theorem~\ref{thm:intdiff-mod-shifted}). Thus we set
\begin{equation}
  \label{eq:def-diag-eval-dirac}
  \mev^x_a \, \delta^{(k)}(x-\xi) = 0, \qquad
  \mev^\xi_a \, \delta^{(k)}(x-\xi) = 0 ,
\end{equation}
completing the definition of~$\mev^x_a$ and~$\mev^\xi_a$ on the diagonal distribution
module~$\distmod_{x-\xi}{\galg} \subset \distmod_2{\galg}$ and hence yielding evaluation
operators~$\mev^x_a\colon \distmod_2{\galg} \to \distmod_\xi{\galg}$
and~$\mev^\xi_a\colon \distmod_2{\galg} \to \distmod_x{\galg}$.

\section{Application to Boundary Problems}
\label{sec:appl}

As mentioned earlier, the treatment of \emph{boundary problems} (for linear ordinary differential
equations) is a major application area for our algebraic approach to piecewise smooth functions and
Dirac distributions. We refer to \cite{RosenkranzRegensburger2008a,RegensburgerRosenkranz2009} for
basic notions and algorithms in the algebraic theory of boundary problems. Consider a regular
boundary problem~$(T, \bspc)$ over an ordinary shifted integro-differential
algebra~$(\galg, \der, \cum)$ and let~$G := (T, \bspc)^{-1}$ be its Green's operator. Assuming a
well-posed two-point boundary value problem, classical analysis~\cite[\S3]{StakgoldHolst2011} will
inform us that $G$ is an integral
operator~$Gf(x) = \cum_{-\infty}^\infty\, g(x,\xi) \, f(\xi) \, d\xi$ with the so-called Green's
function~$g(x,\xi)$ as its integral kernel. We shall denote the initialization point of~$\cum$
by~$o \in K$ so that~$\cum = \cum_{\!o}$ and~$\evl = \ev_o$.

We distinguish now \emph{three essentially independent applications} of the algebraic theory
developed in Sections~\ref{sec:intdiffalg}--\ref{sec:biv-distributions} to such boundary problems,
which we elaborate in this section:
\begin{enumerate}
\item The Green's function~$g(x,\xi)$ is a (bivariate) piecewise smooth function, usually described
  by a \emph{case distinction}; we would like to express it in the algebraic language of Heaviside
  functions. For ill-posed boundary problems, $g(x, \xi)$ may be a Dirac distribution that we wish
  to express in terms of the distribution module.
\item The very \emph{definition of the Green's function}~$g_\xi(x) := g(x,\xi)$ is typically cast in
  the language of distributions~\cite[(3.3.4)]{StakgoldHolst2011}. Subject to suitable smoothness
  constraints, it is described uniquely by requiring it, as a function of~$x$, to satisfy the
  differential equation~$T g_\xi = \delta_\xi$ and the boundary
  conditions~$\beta(g_\xi) = 0 \; (\beta \in \bspc)$.
\item A specific instance of the boundary problem~$(T, \bspc)$ arises by choosing a \emph{forcing
    function}~$f$. Thus one wants to find~$u \in \galg$ such that~$Tu = f$
  and~$\beta(u) = 0 \; (\beta \in \bspc)$. In terms of the Green's operator~$G$, the solution is
  expressed by the action~$u = Gf$, which has been defined when~$f \in \galg$. For a piecewise
  smooth\footnote{\label{fn:crazy-func}This is a sensible hypothesis for applications. The usual
    requirement is piecewise continuity~\cite[\S3.1.1]{StakgoldHolst2011}, but continuous functions
    failing $C^\infty$ except on isolated singularities are bizarre (Weierstrass function).} forcing
  function~$f$, no choice of integro-differential algebra~$\galg$ will enable~$f \in \galg$ since
  piecewise smooth functions do not form an integro-differential algebra
  (Proposition~\ref{prop:pcw-drb}).
\end{enumerate}
For a still more ambitious generalization, see our remarks in the Conclusion.

Let us return to the given regular boundary problem~$(T, \bspc)$. We allow~$(T, \bspc)$ to be an
arbitrary \emph{Stieltjes boundary problem}~\cite{RosenkranzSerwa2015}, meaning: (1) It may have
more than two evaluation points; (2) it may involve definite integrals in the boundary conditions;
(3) it may be ill-posed. We assume now that~$(\galg, \der, \cum)$ is an ordinary shifted
integro-differential algebra over the ordered field~$K$; then all the results of
Sections~\ref{sec:pcwext} and~\ref{sec:distmod} on~$\pcwext{\galg} \subset \distmod{\galg}$ are
available. The corresponding set of evaluations will be denoted
by~$\Phi := \{ \evl_a \mid a \in K\}$. We may form the standard integro-differential operator
ring~$\intdiffop$ and its equitable variant~$\eqintdiffop$, as described
in~\cite{RosenkranzSerwa2015}. Let~$J = \{ a_1, \dots, a_k \} \subseteq K$ be the evaluations
actually occurring in the boundary conditions~$\bspc$, in the sense that all~$\beta \in \bspc$ are
contained in the right ideal generated by the evaluations~$\evl_a \; (a \in J)$. Picking
an~$a \in \{ a_1, \dots, a_k\}$ as initialization point~$o$ of the Rota-Baxter operator~$\cum$
on~$\galg$ will avoid spurious case distinctions in~$g(x,\xi)$, but this is not required for correct
extraction~\cite[Rem.~1]{RosenkranzSerwa2015}.

The setting described in~\cite{RosenkranzSerwa2015} took the standard integro-differential
algebra~$\galg = C^\infty(\RR)$ over the real field~$K = \RR$ as a starting point for an algorithm
\emph{extracting the Green's function}~$g(x,\xi)$ from the Green's operator~$G$, which may itself be
computed as in~\cite{RosenkranzRegensburger2008a}. Since~$g(x,\xi)$ is at best piecewise smooth (for
well-posed problems) and in general even distributional (for ill-posed problems), a concrete
distribution module\footnote{It is essentially a certain $C^\infty(\RR)$-submodule of the dual space
  of the smooth compactly supported test functions, namely the one is generated by the
  Heavisides~$H_a$ and the Diracs~$\delta_a$ for all~$a \in J$.} from analysis was chosen.  For the
algebraic framework of boundary problems~$(T, \bspc)$ it is more appropriate to provide a purely
algebraic construction for accommodating the Green's function. We shall now show that we may indeed
consider~$g(x, \xi) \in \bivmod{\galg}$ for regular Stieltjes boundary problems
(Theorem~\ref{thm:extr-corr}) and~$g(x, \xi) \in \pcwext_2{\galg}$ for well-posed problems
(Proposition~\ref{prop:gf-well-posed}).

The procedure to achieve this goal is rather straightforward: The algorithm
of~\cite{RosenkranzSerwa2015} can be used~\emph{verbatim}, provided we interpret all Heavisides and
Diracs in the sense of~$\bivmod{\galg}$. We need only prove that the latter have the properties
required for the proof of the Structure Theorem for Green's
Functions~\cite[Thm.~1]{RosenkranzSerwa2015}. We start with the \emph{extraction map}
$\eta\colon \intdiffop \to \bivmod{\galg}$, which we shall write~$G \mapsto G_{x\xi}$ as in the
corresponding definition given in~\cite[\S5]{RosenkranzSerwa2015} before Lemma~1 (but we forgo the
modified equitable form, which may sometimes lead to further simplifications). For convenience, we
write out the definition of~$\eta$ in Table~\ref{tab:ex-map} below, using the natural $K$-basis
of~$\intdiffop$.

\begin{table}[ht!]
  \begin{equation*}
    \renewcommand{\arraystretch}{1.25}
    \begin{array}{|l|l|}
      \hline
      G \in \intdiffop & G_{x\xi} \in \bivmod{\galg}\\\hline
      u \, \der^i & u(x) \, \delta^{(i)}(x-\xi)\\
      u \cum v & u(x) \, v(\xi) \, [o \le \xi \le x]_{\pm}\\
      u \, \ev_a \der^i & (-1)^i \, u(x) \, \delta^{(i)}(\xi-a)\\
      u \, \ev_a \cum v & u(x) \, v(\xi) \, [o \le \xi \le a]_{\pm}\\\hline
    \end{array}
  \end{equation*}
  \caption{Extraction Map~$\eta\colon \intdiffop \to \bivmod{\galg}$}
  \label{tab:ex-map}
\end{table}

Here we have employed the abbreviation~$[a \le \xi \le b] := H(\xi-a) \, \bar{H}(\xi-b)$ for the
characteristic function of the interval~$[a,b]$ and~$[a \le \xi \le x] := H(\xi-a) \, H(x-\xi)$ for
that of~$[a,x]$. Note that this presupposes~$a < b$ and~$a < x$. While~$a$ and~$b$ are on an equal
footing, we must define the characteristic
function~$[x \le \xi \le a] := \bar{H}(\xi-a) \, \bar{H}(x-\xi)$ separately for the interval~$[x,a]$
with~$x < a$. For the extraction one actually needs signed versions for recording the relative
order, namely~$[a \le \xi \le b]_{\pm} := [a \le \xi \le b] - [b \le \xi \le x]$
and~$[a \le \xi \le x]_{\pm} := [a \le \xi \le x] - [x \le \xi \le a]$. One checks immediately that
this simplifies to~$[a \le \xi \le b]_{\pm} = H(\xi-a)-H(\xi-b)$
and to $[a \le \xi \le x]_{\pm} = H(x-\xi) + H(\xi-a) - 1$, respectively.

\begin{remark}
  The first row in Table~\ref{tab:ex-map} might make the impression of missing an \emph{alternating
    sign}, which was indeed---erroneously---present in our original
  formulation~\cite{RosenkranzSerwa2015}. Acting on a function~$f(x)$ and setting~$u(x)=1$ for
  simplicity, the rule is~$\vcum_{\!\alpha}^\beta \, \delta^{(i)}(x-\xi) \, f(\xi) = f^{(i)}(x)$. In
  analysis this is usually written as
  \begin{equation}
    \label{eq:sifting-der}
    \cum_{\!-\infty}^\infty \, \delta^{(i)}(\xi-x) \, f(\xi) \, d\xi = (-1)^i \, f^{(i)}(x)
  \end{equation}
  or~$\delta^{(i)}_x[f] = (-1)^i \, f^{(i)}(x)$ in the language of functionals. In the
  example~$\galg = C^\infty(\RR)$ with standard integro-differential structure, both formulations
  are equivalent by the well-known Dirac
  symmetry~$\delta^{(i)}(x-\xi) = (-1)^i \, \delta^{(i)}(\xi-x)$. In contrast, the alternating sign
  in the third row of \mbox{Table~\ref{tab:ex-map}} \emph{cannot} be avoided since this corresponds
  directly to~\eqref{eq:sifting-der}. It should also be noted that the sign in the fourth row of
  Table~\ref{tab:ex-map} has been corrected with respect to~\cite{RosenkranzSerwa2015},
  where~$\sgn(a)$ was used instead of the signed characteristic function.
\end{remark}

For simplicity we assume the \emph{initialization point} to be~$o = 0$. Otherwise some computations
would only become more cumbersome without providing additional insight (producing various
intermediate terms that eventually all cancel out). A nonzero initialization point~$o$ is best
handled by adapting the splitting point of~\eqref{eq:pcw-def1}
and~\eqref{eq:biv-pcw-def1}--\eqref{eq:biv-pcw-def2}; confer Footnote~\ref{fn:init-pt}.

Let us now prove that the extraction procedure preserves the meaning of the Green's operator. Since
the following result is applicable to the standard example~$(C^\infty(\RR), \der, \cum)$
over~$K = \RR$, it includes the setting of~\cite{RosenkranzSerwa2015} and may be seen as an
algebraic abstraction of the distribution setup customarily used in this context. For achieving a
smooth formulation, let us introduce an algebraic generalization of \emph{functional equality
  restricted to intervals}~$[\alpha, \beta]$ of the real line. Given piecewise
functions~$f, g \in \pcwext{\galg}$ and~$\alpha < \beta \in K$, we say that~$f = g$ \emph{on
  $[\alpha, \beta]$} iff~$f \equiv g \pmod{Z_{[\alpha, \beta]}}$ where~$Z_{[\alpha, \beta]}$ is the
ideal of~$\pcwext{\galg}$ generated by~$\bar{H}_\alpha$ and~$H_\beta$. In the standard
example~$\galg = C^\infty(\RR)$ this corresponds to the familiar notion of analysis. Since
$\pcwext{\galg}$ is isomorphic to the rings~$\pcwext_x{\galg}$ and~$\pcwext_\xi{\galg}$, we may
apply analogous interval restriction to the two latter rings. This allows us to give a precise
meaning to the colloquial statement: ``The Green's function provides a faithful realization of the
Green's operator.''

\begin{theorem}
  \label{thm:extr-corr}
  Let~$\galg$ be an ordinary shifted integro-differential algebra over any ordered field~$K$, and
  let~$\eta\colon \intdiffop \to \bivmod{\galg}$ be as in Table~\ref{tab:ex-map}.
  Choose~$\alpha, \beta \in K$ with~$\alpha \le a_1 < \dots < a_k \le \beta$.
  Writing~$\vcum := \smash{\vcum^\xi}$ for brevity, we have
  \begin{equation}
    \label{eq:extr-corr}
    G\!f (x) = \vcum_{\!\alpha}^\beta \, g(x, \xi) \, f(\xi) \quad\in\quad \galg_x
  \end{equation}
  on~$[\alpha, \beta]$, for all~$f \in \galg$ and~$G \in \intdiffop$ with
  extraction~$g(x,\xi) := G_{x\xi}$. If~$G$ is the Green's operator of a regular Stieltjes boundary
  problem, $g(x, \xi)$ is thus its Green's function.
\end{theorem}

\begin{proof}
  Let us start by recalling the exact meaning of~\eqref{eq:extr-corr}, for an arbitrary Green's
  operator~$G \in \intdiffop$ arising from a regular Stieltjes boundary problem and a forcing
  function $f \in \galg$. On the left-hand side we have the usual action of the operator
  ring~$\intdiffop$ on the underlying integro-differential algebra~$\galg$; thus~$G\!f \in \galg$
  and~$G\!f(x) = \iota_x(G\!f) \in \galg_x$ via the
  embedding~$\iota_x\colon \galg \hookrightarrow \galg_2$. On the right-hand side
  of~\eqref{eq:extr-corr} we have the associated Green's function~$g(x,\xi) \in \distmod_2{\galg}$
  and the given function~$f(\xi) = \iota_\xi(f) \in \galg_\xi$ via the 
  embedding~$\iota_\xi\colon \galg \hookrightarrow \galg_2$; their
  product~$g(x, \xi) \, f(\xi) \in \distmod_2{\galg}$ is then integrated
  via~$\smash{\vcum_{\!\alpha}^\beta} := \smash{\vcum_{\!\alpha}^\xi} - \smash{\vcum_{\!\beta}^\xi}$
  where we have as usual set~$\smash{\vcum_{\!c}^\xi} := (1-\bivmev{\xi}{c}) \, \smash{\vcum^\xi}$
  and~$\bivmev{\xi}{c} := \mevl_\xi \, \bivmsh{\xi}{c}$ for arbitrary~$c \in K$.

  Let us now go through the rows of Table~\ref{tab:ex-map}. The first case is~$G = u \, \der^i$ so
  that we obtain immediately~$G\!f(x) = u(x) \, f^{(i)}(x)$ for the left-hand side
  of~\eqref{eq:extr-corr}. Since $u(x)$ is constant with respect to~$\smash{\vcum^\xi}$, the
  right-hand side is given
  by~$u(x) \, \smash{\vcum_{\!\alpha}^\beta} \, f(\xi) \, \delta^{(i)}(x-\xi)$, and it suffices to
  show
  \begin{equation}
    \label{eq:ex-case1}
    f^{(i)}(x) = \smash{\vcum_{\!\alpha}^\beta} \, f(\xi) \, \delta^{(i)}(x-\xi)
    \qquad \text{on~$[\alpha, \beta]$} ,
  \end{equation}
  which one does by induction on~$i$. For the base case~$i=0$ we compute
  \begin{align*}
    \vcum^\xi f(\xi) & \, \delta(x-\xi) = \vcum^\xi f'(\xi) \, H(x-\xi) - f(\xi) \, H(x-\xi)\\
                     & {} = \Big( \cum_{\!x}^\xi \, f'(\xi) \Big) \, H(x-\xi) 
                       + \Big( \cum^x f'(\xi) \Big) \, H_0(x) - f(\xi) \, H(x-\xi) ,
  \end{align*}
  using first~\eqref{eq:biv-dist-def2} and then~\eqref{eq:biv-pcw-def2}.
  Since~$\cum_{\!x}^\xi \, f'(\xi) = f(\xi) - f(x)$ and~$\cum^x f'(\xi) = f(x) - f(0)$, this
  simplifies to~$\vcum^\xi f(\xi) \, \delta(x-\xi) = - f(x) \, H(x-\xi) + r(x)$ where the
  term~$r(x) \in \pcwext_x{\galg}$ is invariant under~$\mev^\xi_a$. Hence the latter term cancels
  in~$\vcum_{\!\alpha}^\beta = (\mev^\xi_\beta - \mev^\xi_\alpha) \, \vcum^\xi$ so that
  \begin{equation*}
    \vcum_{\!\alpha}^\beta \, f(\xi) \, \delta(x-\xi) = \Big( \mev^\xi_\alpha - \mev^\xi_\beta \Big)
    \, f(x) \, H(x-\xi) = f(x) \, \Big( H(x-\alpha) - H(x-\beta) \Big) ,
  \end{equation*}
  where in the last step we have applied~\eqref{eq:def-diag-eval-pcw}. As a consequence,
  $f(x) - \vcum_{\!\alpha}^\beta \, f(\xi) \, \delta(x-\xi)$ is given
  by~$f(x) \big( \bar{H}_\alpha(x) + H_\beta(x) \big) \in Z_{[\alpha, \beta]} \subset
  \pcwext_x{\galg}$,
  which means that~$f(x) = \smash{\vcum_{\!\alpha}^\beta} \, f(\xi) \, \delta(x-\xi)$
  on~$[\alpha, \beta]$ as claimed. For the induction step, we compute
  \begin{equation*}
    \smash{\vcum_{\!\alpha}^\beta} f(\xi) \, \delta^{(i+1)}(x-\xi) = 
    \Big( \mev^\xi_\beta - \mev^\xi_\alpha \Big) \, 
    \Big( \vcum^\xi f'(\xi) \, \delta^{(i)}(x-\xi) - f(\xi) \, \delta^{(i)}(x-\xi) \Big)
  \end{equation*}
  by~\eqref{eq:biv-dist-def2}, which reduces
  to~$\smash{\vcum_{\!\alpha}^\beta} \, f'(\xi) \, \delta^{(i)}(x-\xi)$ since diagonal Diracs
  evaluate to zero by~\eqref{eq:def-diag-eval-dirac}. Using the induction
  hypothesis~\eqref{eq:ex-case1} with~$f'$ in place of~$f$, the latter integral
  equals~$f^{(i+1)}(x)$ on~$[a,b]$; this is indeed~\eqref{eq:ex-case1} for~$i+1$.

  Next we treat the second row of Table~\ref{tab:ex-map}. Since~$u(x)$ is constant with respect
  to~$\vcum^\xi$, it suffices to
  show~$\cum^x v(x) \, f(x) = \smash{\vcum_{\!\alpha}^\beta} \, v(\xi) \, [0 \le \xi \le x]_{\pm} \,
  f(\xi)$
  on~$[\alpha, \beta]$. Obviously, we may set~$v = 1$ without loss of generality.
  Using~$[0 \le \xi \le x]_{\pm} = H(x-\xi) + H_0(\xi) - 1$ we have
  \begin{align*}
    \vcum^\xi f(\xi) \, [0 & \le \xi \le x]_{\pm} = \vcum^\xi f(\xi) \, H(x-\xi) + \cum^\xi f(\xi) \,
    H_0(\xi) - \cum^\xi f(\xi)\\
    & {} = \Big( \cum_{\!x}^\xi \, f(\xi) \Big) \, H(x-\xi) + \Big( \cum^x f(x) \Big) \,
      H_0(x) - \Big( \cum^\xi f(\xi) \Big) \bar{H}_0(\xi)
  \end{align*}
  by applying~\eqref{eq:biv-pcw-def2} and~\eqref{eq:pcw-def1}. The middle summand cancels
  in~$\vcum_{\!\alpha}^\beta = (\mev^\xi_\beta - \mev^\xi_\alpha) \, \vcum^\xi$, and one obtains
  after a few simplifications
  \begin{equation*}
    \vcum_{\!\alpha}^\beta \, f(\xi) \, [0 \le \xi \le x]_{\pm} = 
    \big( \cum_{\!0}^\alpha \, f \big) \, \bar{H}_\alpha(x) + \big( \cum_{\!0}^\beta \, f \big) \,
    H_\beta(x) + \big( \cum^x f \big) \big( H_\alpha(x) - H_\beta(x) \big) .
  \end{equation*}
  Here we have used the facts~$H(-\alpha) = 1$ and~$H(-\beta) = 0$, which follow from our assumption
  that the interval~$[\alpha, \beta]$ contains the initialization point~$o = 0$ so
  that~$\alpha < 0 < \beta$. Since the first two summands on the right-hand side above are
  in~$Z_{[\alpha, \beta]} \subset \pcwext_x{\galg}$, we obtain finally
  \begin{equation*}
    \cum^x f - \vcum_{\!\alpha}^\beta \, f(\xi) \, [0 \le \xi \le x]_{\pm} \equiv
    \big( \cum^x f \big) \big( \bar{H}_\alpha(x) + H_\beta(x) \big) \equiv 0
    \pmod{Z_{[\alpha, \beta]}} ,
  \end{equation*}
  and~$\smash{\vcum_{\!\alpha}^\beta} \, f(\xi) \, [0 \le \xi \le x]_{\pm}$ is indeed equal
  to~$\cum^x f(x)$ on~$[\alpha, \beta]$, as was claimed.

  Turning to the third row of Table~\ref{tab:ex-map}, we can again set~$u(x) = 1$ without loss of
  generality. Hence we must show
  that~$f^{(i)}(a) = \smash{\vcum_{\!\alpha}^\beta} \, \delta^{(i)}(\xi-a) \, f(\xi)$ holds
  on~$[\alpha, \beta]$. In fact, it turns out to hold without constraints. We can now work purely
  in~$\distmod_\xi{\galg}$ and use~\eqref{eq:rb-dist} to calculate
  \begin{equation*}
    \vcum^\xi \delta^{(i)}(\xi-a) \, f(\xi) = \sum_{k=0}^i (-1)^k \, f^{(k)} \, H^{(i-k)}(\xi-a) 
    - (-1)^i \, \vcum^\xi f^{(i+1)}(\xi) \, H(\xi-a)
  \end{equation*}
  by a straightforward induction on~$i \ge 0$. Applying
  again~$\vcum_{\!\alpha}^\beta = (\mev^\xi_\beta - \mev^\xi_\alpha) \, \vcum^\xi$, all terms in the
  sum except for~$k=i$ cancel since~$\smash{\mev^\xi_\alpha}$ and~$\smash{\mev^\xi_\beta}$
  annihilate the Diracs, and we get
  \begin{equation*}
    \vcum_{\!\alpha}^\beta \, \delta^{(i)}(\xi-a) \, f(\xi) = (-1)^i \Big( f^{(i)}(\beta) \,
    \bar{H}(a-\beta) - f^{(i)}(\alpha) \, \bar{H}(a-\alpha) - \vcum_{\!\alpha}^\beta \,
    f^{(i+1)}(\xi) \, H(\xi-a) \Big) .
  \end{equation*}
  By our assumption~$\alpha < a < \beta$ we have~$\bar{H}(a-\alpha) = 0$ and~$\bar{H}(a-\beta) =
  1$. Now we compute the remaining integral according to~\eqref{eq:pcw-def1} to obtain
  \begin{equation*}
    (-1)^i \, \vcum_{\!\alpha}^\beta \, \delta^{(i)}(\xi-a) \, f(\xi) = f^{(i)}(\beta)
    + \Big( \mev^\xi_\alpha - \mev^\xi_\beta \Big) \Big( \big( \cum_{\!a}^\xi \, f^{(i+1)} \big)
    H(\xi-a) + \bar{H}(a) \, \cum_{\!0}^a \, f^{(i+1)} \Big) .
  \end{equation*}
  The last term in the right parenthesis cancels since it is invariant under both~$\mev^\xi_\alpha$
  and~$\mev^\xi_\beta$. Since we
  have~$\smash{\cum_{\!a}^\xi} \, f^{(i+1)} = f^{(i)}(\xi) - f^{(i)}(a)$, we get
  for~$(-1)^i \, \smash{\vcum_{\!\alpha}^\beta} \, \delta^{(i)}(\xi-a) \, f(\xi)$ the expected
  result
  \begin{equation*}
    f^{(i)}(\beta) + \bar{H}(a-\alpha) \, \Big( f^{(i)}(\alpha) - f^{(i)}(a) \Big) - 
    \bar{H}(a-\beta) \, \Big( f^{(i)}(\beta) - f^{(i)}(a) \Big) = f^{(i)}(a),
  \end{equation*}
  using again~$\bar{H}(a-\alpha) = 0$ and~$\bar{H}(a-\beta) = 1$.

  It remains to consider the fourth row of Table~\ref{tab:ex-map}. As for the second row, we may
  omit~$u(x)$ and~$v(x)$ without loss of generality, and it suffices to
  prove~$\cum_{\!0}^a \, f = \smash{\vcum_{\!\alpha}^\beta} f(\xi) \, [0 \le \xi \le
  a]_{\pm}$.
  Using $[0 \le \xi \le a]_{\pm} = H_0(\xi) - H_a(\xi)$, we compute
  first~$\smash{\vcum^\xi} f(\xi) \, [0 \le \xi \le a]_{\pm}$ as
  \begin{equation*}
    \cum^\xi f(\xi) \, H_0(\xi) - \cum^\xi f(\xi) \, H_a(\xi) = (\cum^\xi f) \, H_0(\xi) -
    (\cum_{\!a}^\xi \, f) \, H_a(\xi) - \bar{H}(a) \, \cum_{\!0}^a \, f
  \end{equation*}
  according to~\eqref{eq:pcw-def1}. As before, the last term cancels when
  computing~$\vcum_{\!\alpha}^\beta = (\mev^\xi_\beta - \mev^\xi_\alpha) \, \vcum^\xi$, and we
  obtain the desired equality
  \begin{align*}
    \smash{\vcum_{\!\alpha}^\beta} & f(\xi) \, [0 \le \xi \le a]_{\pm} = 
    (\cum_{\!0}^\beta f) \, \bar{H}(-\beta) - (\cum_{\!0}^\alpha f) \, \bar{H}(-\alpha) + \Big(
                                     \cum_{\!0}^a \, f - \cum_{\!0}^\beta \, f \Big) \, \bar{H}(a-\beta)\\
    & {} + \Big( \cum_{\!0}^\alpha \, f - \cum_{\!0}^a \, f \Big) \, \bar{H}(a-\alpha) =
      \cum_{\!0}^a f,
  \end{align*}
  using again~$\bar{H}(-\alpha) = \bar{H}(a-\alpha) = 0$ and~$\bar{H}(-\beta) = \bar{H}(a-\beta)
  = 1$.
\end{proof}

\begin{proposition}
  \label{prop:gf-well-posed}
  Let~$\galg$ and~$\eta\colon \intdiffop \to \bivmod{\galg}$ be as in
  Theorem~\ref{thm:extr-corr}. If the regular boundary problem~$(T, \bspc)$ is well-posed, then we
  have~$g(x,\xi) \in \pcwext_2{\galg}$ for the Green's function~$g(x,\xi) = G_{x\xi}$ extracted from
  its Green's operator~$G = (T,\bspc)^{-1 }$.
\end{proposition}

\begin{proof}
  We use the crucial fact~\cite[Thm.~1]{RosenkranzSerwa2015} that the Green's
  function $g(x,\xi) = \tilde{g}(x,\xi) + \hat{g}(x,\xi)$ splits into a functional
  part~$\tilde{g}(x,\xi) \in \pcwext_2{\galg}$ and a distributional
  part $\hat{g}(x,\xi) \in \distmod_2{\galg} \setminus \pcwext_2{\galg}$. Hence it suffices to show
  that~$\hat{g}(x,\xi) = 0$. In the proof of~\cite[Thm.~1]{RosenkranzSerwa2015}, the splitting of
  the Green's function is induced by a corresponding splitting of the Green's
  operator~$G = \tilde{G} + \hat{G}$ into a functional part~$\tilde{G}$
  with~$\tilde{G}_{x\xi} = \tilde{g}(x,\xi)$ and a distributional part~$\hat{G}$
  with~$\hat{G}_{x\xi} = \hat{g}(x,\xi)$, so our goal is to show~$\hat{G} = 0$. From the proof
  of~\cite[Lem.~2]{RosenkranzSerwa2015} we see that the only possible contributions to~$\hat{G}$
  come from terms of the form~$f \ev_\alpha \der^k$ in the kernel projector~$P$. Moreover, such a
  term will go to~$\tilde{G}$ if~$k < n$ since in this case the second sum
  in~\cite[Lem.~1]{RosenkranzSerwa2015} is absent, as has been observed after Equation~(8)
  of~\cite{RosenkranzSerwa2015}.

  Thus it suffices to prove that~$k<n$ for all terms~$f \ev_\alpha \der^k$ occurring in the kernel
  projector~$P$. But this is clear from the form of~$P$ as given e.g.\@ in the proof of Theorem~26
  in~\cite{RosenkranzRegensburger2008a}.  Indeed, if~$u = (u_1, \dots, u_n) \in \galg^n$ is a
  fundamental system for~$\ker{T}$ and~$\beta = (\beta_1, \dots, \beta_n) \in (\Phi)^n$ a basis of
  Stieltjes conditions for the boundary space~$\bspc$, the kernel projector is given
  by~$P = \trp{u} \cdot \beta(u)^{-1} \cdot \beta \in \intdiffop$, and since~$(T,\bspc)$ is
  well-posed by hypothesis, all local terms~$\ev_\alpha \der^k$ occurring in the boundary
  conditions~$\beta_1, \dots, \beta_n$ must have~$k < n$.
\end{proof}

We turn now to the application labeled~(2) in the above introduction. We continue to assume
that~$(\galg, \der, \cum)$ is an ordinary shifted integro-differential algebra. Recall that any
integro-differential algebra~$(\galg, \der, \cum)$ contains an isomorphic copy of the
\emph{polynomial ring}~$(K[x], \der, \cum)$ with its standard integro-differential
structure~\cite[Prop.~3]{BuchbergerRosenkranz2012}; we may use the identification~$x := \cum 1$.
Since~$(\galg, \der, \cum)$ is ordinary, it is not difficult to see
that~$\ker{\der^n} = [1, \dots, x^{n-1}]$ so that~$\dim \ker{\der^n} = n$, as one would
expect~\cite[Lem.~3.21]{RegensburgerRosenkranz2009a}. However, we need an additional condition to
ensure similar behavior for arbitrary differential operators (including nonmonic ones). Hence let us
call a differential algebra~$(\galg, \der)$ \emph{strongly ordinary} if~$\dim \ker{T} < \infty$ for
any~$T \in \galg[\der]$. Note that all the usual examples of ordinary differential algebras in
analysis are strongly ordinary, in particular our standard example~$C^\infty(\RR)$. In fact, there
is an explicit upper bound in the real-analytic theory~\cite[Thm.~1.3.6]{KatoStruppa1999},
namely~$\dim \ker{T} \le m + d$, where~$m$ is the order of the differential operator~$T$ and $d$
counts the zeros of its leading coefficient with multiplicities (using hyperfunctions the estimate
becomes an identity).

We must first ensure that we can uniquely recover Dirac distributions. To achieve this we will need
some analytic assumption that plays the role of the Fundamental Lemma of the Variational Calculus,
often ascribed to Paul du Bois-Reymond. Hence we say that~$\nglb \in \galg$ is \emph{degenerate on}
$[\alpha, \beta]$ if~$\cum_{\!\alpha}^\beta \, \nglb(\xi) \, f(\xi) = 0$ for all~$f \in \galg$. In
the language of~\cite[\S3]{RosenkranzKorporal2013}, this says that the Stieltjes
condition~$\cum_{\!\alpha}^\beta \, \nglb$ is degenerate. In the same vein, we
call~$k(x,\xi) \in \distmod_2{\galg}$ nondegenerate it does not contain any
degenerate~$\nglb(\xi) \in \galg_\xi$; we do not need a similar condition on its $\galg_x$
parts. This restriction is clearly no loss of generality since our goal is to integrate
over~$[\alpha, \beta]$, so degenerate functions may as well be discarded from the outset.

\begin{proposition}
  \label{prop:only-delta}
  Let~$(\galg, \der, \cum)$ be a strongly ordinary shifted integro-differential algebra and choose
  any bivariate distribution~$k(x,\xi) \in \distmod_2{\galg}$ that is nondegenerate
  on~$[\alpha, \beta]$.  If one has $\smash{\vcum_{\!\alpha}^\beta} \, k(x, \xi) \, f(\xi) = f(x)$
  on~$[\alpha, \beta]$ for all~$f \in \galg$, then necessarily~$k(x,\xi) = \delta(x-\xi)$.
\end{proposition}

\begin{proof}
  We
  have~$\distmod_2{\galg} = (\distmod_x{\galg} \otimes_\galg \pcwext_\xi{\galg}) \oplus
  (\pcwext_x{\galg} \otimes_\galg \distmod_\xi{\galg}) \oplus \distmod_{x-\xi}{\galg}$
  by the definition of bivariate distributions, hence we may assume
  \begin{equation}
    \label{eq:kernel}
    k(x,\xi) = \sum_{i,a} \Xi_{i,a}(\xi) \, \delta^{(i)}(x-a) + \sum_{i,a} X_{i,a}(x)
    \, \delta^{(i)}(\xi-a) + \sum_{i=0}^N M_i(x,\xi) \, \delta^{(i)}(x-\xi),
  \end{equation}
  where~$N \in \NN$, the summations are over~$i \in \NN$ and~$a \in K$, containing only
  finitely many nonzero coefficients~$\Xi_{i,a}(\xi) \in \pcwext_\xi{\galg}$,
  $X_{i,a}(x) \in \pcwext_x{\galg}$ and~$M_i(x,\xi) \in \pcwext_{x\xi}{\galg}$. Since the
  latter ring is by
  definition~$\pcwext_{x\xi}{\galg} = \pcwext_x{\galg} \otimes_\galg \pcwext_\xi{\galg}$, we may
  also assume~$M_i(x,\xi) = L_i(x) \, R_i(\xi)$ with~$L_i(x) \in \pcwext_x{\galg}$
  and~$R_i(\xi) \in \pcwext_\xi{\galg}$. In fact, we can further restrict
  to~$R_i(\xi) \in \galg_\xi$ due to the relations contained in the
  $\pcwext_{x\xi}{\galg}$-submodule~$\smash{\hat{Z}}$ of Definition~\ref{def:diagdist}. From
  Table~\ref{tab:ex-map} we can read off the action of each term in~\eqref{eq:kernel} on the
  left-hand side of the given
  identity~$\smash{\vcum_{\!\alpha}^\beta} \, k(x, \xi) \, f(\xi) = f(x)$.  Thus we obtain
  \begin{equation}
    \label{eq:kernel-action}
    \sum_{i,a} \Xi_{i,a} \, \delta^{(i)}(x-a) + \sum_{i,a} (-1)^i \,
    X_{i,a}(x) \, f^{(i)}(a) + \sum_{i=0}^N L_i(x) \, \der^i \Big( R_i(x) \, f(x) \Big) = f(x),
  \end{equation}
  on~$[\alpha, \beta]$
  where~$\Xi_{i,a} := \cum_{\!\alpha}^\beta \, \Xi_{i,a}(\xi) \, f(\xi) \in K$. Since the
  distributions~$\delta^{(i)}(x-a)$ are by construction linearly independent from any element
  of~$\pcwext_x{\galg}$, the assumption~$\bar{f}(x) = f(x)$ forces~$\Xi_{i,a} = 0$ and hence
  also~$\Xi_{i,a}(\xi) = 0$ by the hypothesis on nondegeneracy.

  As noted above, we have~$K[x] \subset \galg$. It is easy to see that in such circumstances, the
  well-known algorithm for \emph{Hermite interpolation} applies to construct polynomials~$p(\xi)$
  with arbitrary values prescribed for~$p^{(j-1)}(\xi_k)$, where~$j \in \{ 1, \dots, m \}$
  and~$\{ \xi_1, \dots, \xi_n \} \subset K$.  Here~$m$ and~$n$ may be arbitrarily large positive
  integers since we can construct polynomials of indefinitely high degree involving powers of~$x-a$
  for any~$a \in K$. In particular, we can construct polynomials~$f_c(\xi)$ with~$f^{(i)}(a) = 0$
  for all those~$(i,a)$ that occur in~\eqref{eq:kernel-action} with nonzero
  coefficients~$X_{i,a}$. We may further assume that they carry~$l$ arbitrary
  parameters~$c_1, \dots, c_l$ by adding suitable interpolation data for ``unused'' higher
  derivatives; the parameters are collected into~$c = (c_1, \dots, c_l) \in K^l$.  Note
  that~$l \in \NN$ is arbitrary since we may add indefinitely many interpolation values some of
  which may be frozen to zero if needed.

  Substituting the interpolation polynomial thus obtained into~\eqref{eq:kernel-action}, also the
  middle sum now vanishes by our construction, and we are left with the differential equation
  \begin{equation}
    \label{eq:diffeq-dubois}
    \sum_{i=0}^N L_i(x) \, \der^i \Big( R_i(x) \, f_c(x) \Big) = f_c(x) 
  \end{equation}
  on~$[\alpha, \beta]$. Let us
  write~$L_i(x) = l_i(x) + \sum_{b \in K} l_{i,b}(x) \, H(x-b)$
  with~$l_i(x), l_{i,b}(x) \in \galg_x$. Then the
  set~$B := \{ b \in K \mid \exists_{i=0,\dots,N} \: L_{i,b}(x) \neq 0 \}$ is clearly
  finite and contained in~$[\alpha,\beta]$, so we may rewrite~\eqref{eq:diffeq-dubois} as
  \begin{equation*}
    \sum_{i=0}^N l_i(x) \, \der^i \Big( R_i(x) \, f_c(x) \Big) 
    + \sum_{b \in B} H(x-b) \sum_{i=0}^N l_{i,b}(x) \, \der^i \Big( R_i(x) \,
    f_c(x) \Big) = f_c(x) + z_{\alpha,\beta}
  \end{equation*}
  for some~$z_{\alpha,\beta} \in \hat{Z}_{\alpha, \beta}$. But the set
  $\{ H(x-a), H(x-\beta) \} \cup \{ H(x-b) \mid b \in B \}$ is linearly
  independent over~$\galg$; hence~\eqref{eq:diffeq-dubois} splits into the~$|B|+1$ separate
  differential equations
  \begin{equation*}
    \sum_{i=0}^N l_i(x) \, \der^i \Big( R_i(x) \, f_c(x) \Big) = f_c(x), \qquad
    \sum_{i=0}^N l_{i,b}(x) \, \der^i \Big( R_i(x) \,
    f_c(x) \Big) = 0 \quad (b \in B) .
  \end{equation*}
  Unless the underlying differential operators vanish, each of these differential equations has an
  infinite-dimensional solution space containing all~$f_c(x)$ with~$c \in K^l$ for $l \in N$
  indefinitely large. Since~$(\galg, \der) \cong (\galg_x, \der_x)$ is assumed to be strongly
  ordinary, we conclude that we must in fact have~$l_0(x) = R_i(x) = 1$ and~$l_i(x) = 0$ for~$i > 0$
  as well as~$l_{i,b}(x) = 0$.

  At this point we have reduced~\eqref{eq:kernel}
  to~$k(x,\xi) = \sum_{i,a} X_{i,a}(x) \, \delta^{(i)}(\xi-a) + \delta(x-\xi)$, hence
  the action yields~$\sum_{i,a} (-1)^i \, X_{i,a}(x) \, f^{(i)}(a) = 0$ for
  all~$f \in \galg$. Since for each fixed~$i$, there are only finitely many nonzero
  coefficients~$\{ X_{i,a} \mid a \in A_i \}$ in~\eqref{eq:kernel}, and these may be
  assumed to be linearly independent over~$K$ since we may always combine them if needed. But
  applying again Hermite interpolation, one may choose~$f \in K[x] \subset \galg_x$ such
  that~$f^{(i')}(a') = 1$ for some fixed pair~$(i',a')$ and~$f^{(i)}(a) = 0$ for all
  other pairs~$(i,a)$. This implies immediately that~$X_{i',a'} = 0$. Since~$(i',a')$
  is arbitrary, we conclude that indeed~$k(x,\xi) = \delta(x-\xi)$.
\end{proof}

For an integro-differential operator~$U \in \intdiffop$ we shall
write~$U_x\colon \distmod_2{\galg} \to \distmod_2{\galg}$
and $U_\xi\colon \distmod_2{\galg} \to \distmod_2{\galg}$ for the two operators induced by the
\emph{action} with respect to~$x$ and with respect to~$\xi$. In other words: If $U = f \in \galg$
then~$U_x$ is the multiplication operator induced by~$f(x) \in \galg_x$ while~$U_\xi$ is induced
by~$f(\xi) \in \galg_\xi$. Likewise, if $U = \der$ then~$U_x$ acts as~$\vder_x$ and~$U_\xi$
as~$\vder_\xi$, and if~$U = \smash\cum$ then $U_x$ acts as~$\smash{\vcum^x}$ and~$U_\xi$ accordingly
as~$\smash{\vcum^\xi}$. Finally, for an evaluation~$U = \ev_\alpha$, the action~$U_x$
is~$\mev^x_\alpha$ and the action~$U_\xi$ is correspondingly~$\mev^\xi_\alpha$.

\begin{theorem}
  \label{thm:dist-bvp}
  Let~$(\galg, \der, \cum)$ be a strongly ordinary shifted integro-differential algebra and
  $(T, \bspc)$ a regular Stieltjes boundary problem over~$(\galg, \der, \cum)$. Then there exists a
  bivariate distribution $g(x,\xi) \in \distmod_2{\galg}$ such that
  \begin{equation}
    \label{eq:dist-bvp}
    \bvp{T_x \, g(x,\xi) = \delta(x-\xi),}{\beta_x \, g(x,\xi) = 0 \quad (\beta \in \bspc).}
  \end{equation}
  Moreover, this~$g(x,\xi)$ coincides with the Green's function of Theorem~\ref{thm:extr-corr}.
\end{theorem}

\begin{proof}
  With~$G$ the Green's operator of the boundary problem~$(T, \bspc)$, set~$g(x,\xi) := G_{x\xi}$.
  For existence it suffices to show that~$g(x,\xi)$ satisfies the distributional boundary
  problem~\eqref{eq:dist-bvp}, and we may furthermore assume that~$g(x,\xi)$ is nondegenerate (we
  may discard any degenerate functions occurring in it since the induced action still represents the
  same~$G$). Since~$G$ is the Green's operator of~$(T, \bspc)$, the function~$u := Gf \in \galg$
  satisfies
  \begin{equation}
    \label{eq:green-func-calc}
    f(x) = T_x \, u(x) = T_x \, \vcum_{\!\alpha}^\beta \, g(x, \xi) \, f(\xi)
    = \vcum_{\!\alpha}^\beta \, \Big( T_x \, g(x, \xi) \Big) \, f(\xi) \qquad \text{on $[\alpha, \beta]$},
  \end{equation}
  where the second step follows from Theorem~\ref{thm:extr-corr} and the last step from the fact
  that~$\vder_x$ and all~$g(x) \in \galg_x$ commute with~$\smash{\vcum^\xi}$ and the
  evaluations~$\mev^\xi_\alpha, \mev^\xi_\beta$. Note that~$T_x \, g(x, \xi)$ is still nondegenerate
  since~$T_x$ does not affect the functions of~$\galg_\xi$. Hence we may apply
  Proposition~\ref{prop:only-delta} to~\eqref{eq:green-func-calc} to
  obtain~$T_x \, g(x,\xi) = \delta(x-\xi)$, which is the first line of~\eqref{eq:dist-bvp}. For
  verifying the second line, take any~$\beta \in \bspc$.  Again we have~$\beta(u) = 0$ since~$G$ is
  the Green's operator of~$(T, \bspc)$. But
  then~$0 = \beta_x \, \smash{\vcum_{\!\alpha}^\beta} \, g(x,\xi) \, f(\xi) =
  \smash{\vcum_{\!\alpha}^\beta} \, \big( \beta_x \, g(x,\xi) \big) \, f(\xi)$,
  which implies the second line of~\eqref{eq:dist-bvp} since the action of~$\beta_x$ again
  preserves the nondegeneracy of~$g(x,\xi)$.
\end{proof}

Note that the Green's function~$g(x,\xi)$ of Theorems~\ref{thm:extr-corr} and~\ref{thm:dist-bvp} is
\emph{unique on}~$[\alpha,\beta]$, meaning unique after discarding all degenerate
functions~$\nglb(\xi) \in \galg_\xi$. This is clear since if~$\tilde{g}(x,\xi)$ is another such
Green's function then~$k(x,\xi) := g(x,\xi)-\tilde{g}(x,\xi)$ would also be nondegenerate but since
they induce the same Green's operator we have~$\vcum_{\!\alpha}^\beta \, k(x, \xi) \, f(\xi) = 0$
for all~$f \in \galg$, and this implies~$k(x,\xi) = 0$ so~$g(x,\xi) = \tilde{g}(x,\xi)$.

Finally, let us now turn to the last goal~(3) outlined at the opening of this section. It is
relatively easy to achieve using the tools we have now at hand. If we have computed a Green's
operator in the usual setting of the ordinary shifted integro-differential
algebra~$(\galg, \der, \cum)$, we may immediately apply it to a \emph{piecewise forcing function} by
restricting the action defined above to~$G_x\colon \pcwext_x{\galg} \to \pcwext_x{\galg}$. In this
case, however, we should restrict ourselves to well-posed boundary problems so
that~$g(x,\xi) \in \pcwext_2{\galg}$ by Proposition~\ref{prop:gf-well-posed}. Otherwise the Green's
function~$g(x,\xi)$ would contain Diracs whose multiplication with the Heavisides
of~$f(\xi) \in \pcwext_\xi{\galg}$ in Theorem~\ref{thm:extr-corr} is undefined.\footnote{Also the
  operator interpretation is at best dubious in this case: The Diracs~$\delta^{(k)}(\xi-\alpha)$
  engender evaluations~$\ev_\alpha \, \der^k$ whose action on Heavisides is questionable. This
  reflects the problematic nature of a boundary problem constraining derivative ``values'' for
  solutions that will be distributional. Only the borderline case of piecewise solutions---obtaining
  when the order of the boundary conditions reaches but does not exceed the order of the
  differential equation---may still make sense when interpreted with caution.}

With these reservations in mind, we can now make a simple but precise statement about piecewise
forcing functions. The basic message is that we can use essentially the same method as for the usual
forcing functions taken from the ground algebra~$\galg$. In particular, \emph{existence and
  uniqueness} go through unscathed.

\begin{proposition}
  \label{prop:pcw-forcing}
  Let~$(\galg, \der, \cum)$ be a strongly ordinary shifted integro-differential algebra and
  let~$(T, \bspc)$ be a well-posed Stieltjes boundary problem. Then~\eqref{eq:bvp} admits exactly
  one solution $u \in \pcwext{\galg}$ for any given forcing function~$f \in \pcwext{\galg}$.
  If~$G \in \intdiffop$ is the corresponding Green's operator with Green's function~$g(x,\xi)$, we
  can compute the solution either via~$u = Gf$ or
  via~$u(x) = \smash{\vcum_{\!\alpha}^\beta} \, g(x, \xi) \, f(\xi)$.
\end{proposition}

\begin{proof}
  From~\cite[Lem.~2]{RosenkranzSerwa2015} we know that~$G \in \galg[\cum_\Phi]$, so the action
  of~$G$ involves only integral operators and multiplication by elements of~$\galg$. But this means
  that the correctness proof for Green's operators~\cite[Thm.~26]{RosenkranzRegensburger2008a} is
  applicable, and all required reduction rules for the operator ring~$\intdiffop$ are valid
  on~$\pcwext{\galg}$. In fact, if we accept the derivation of Proposition~\ref{prop:pcw-drb}, the
  entire action on~$\pcwext{\galg}$ would be well-defined except for the~$\cum f \der$ rule
  of~\cite[Table~1]{RosenkranzRegensburger2008a}, which breaks down because the strong Rota-Baxter
  axiom does not hold in~$\pcwext{\galg}$. While even this could repaired by constructing the
  differential Rota-Baxter operator ring~\cite[\S4]{GaoGuoRosenkranz2015a} instead of the usual
  integro-differential operator ring, we do not need this here since no differential operators are
  involved in computing~$u = Gf \in \pcwext{\galg}$.

  This settles the question of existence. For proving uniqueness, it is sufficient to show that the
  homogeneous problem~\eqref{eq:bvp} with~$f=0$ has only the trivial
  solution~$u \in \distmod{\galg}$. But we know
  that~$\distmod{\galg} = \galg \oplus \distmod^*{\galg}$ as differential $K$-vector spaces (i.e.\@
  differential modules over the ground field~$K$), temporarily
  setting~$\distmod^*{\galg} := \big( \distmod{\galg} \setminus \galg \big) \cup \{ 0 \}$. Moreover,
  the derivation $\vder\colon \distmod^*{\galg} \to \distmod^*{\galg}$ respects the filtration
  outlined earlier (at the end of Section~\ref{sec:biv-distributions}). Therefore~$Tu = 0$
  implies~$u \in \galg$, and this in turn implies~$u=0$ since~$(T,\bspc)$ was assumed to be a
  regular boundary problem over~$\galg$.

  Finally, note that Theorem~\ref{thm:extr-corr} is still valid when restricted to Green's operators
  of well-posed boundary problems (hence the unnecessary---and now invalid---cases for the first and
  third row in Table~\ref{tab:ex-map} can be removed). This can be seen by a straightforward
  generalization of the computations in the proof of Theorem~\ref{thm:extr-corr} (for the second and
  fourth case).
\end{proof}

Our treatment of forcing functions in~$\pcwext{\galg}$ includes the classical case of piecewise
smooth functions by using the standard example~$\galg = C^\infty(\RR)$. With the reservations made
above (cf.\@ Footnote~\ref{fn:crazy-func}), this includes in particular the \emph{calculus for
  functions with jumps} outlined in Example~11 of~\cite[\S2.1]{StakgoldHolst2011}.

\section{Conclusion}

Our algebraic treatment of \emph{piecewise functions and distributions} is no more than a starting
point. Future work might also consider the two constructions in separate developments. Indeed, we
have pointed out in Remark~\ref{rem:no-diffring} that the multiplicative structure exported from the
piecewise extension~$\pcwext{\galg}$ is independent of the other structures on the distribution
module $\distmod{\galg}$; we might impose any product whatsoever. While this might be construed as a
weakness of the algebraic approach, it clarifies at least the complementary character of the
Diracs~$\delta_a$ and the Heavisides~$H_a$: While the multiplication of the latter reflects an order
structure in the ground field, the former encode point evaluations without any relation to the
order. The only link between the two structures is the defining relation~$H_a' = \delta_a$.

A more ambitious treatment would also allow \emph{piecewise continuous coefficients} of the
differential operator~$T$; this is what is typically encountered in interface
problems~\cite[\S1.4]{StakgoldHolst2011}. However, it would be difficult to accommodate such a case
directly into our present approach since generalizing~$\galg$ to be a differential Rota-Baxter
(rather than an integro-differential) algebra entails the loss of the strong Rota-Baxter
axiom~\eqref{eq:str-rb-axiom}. In that case, Green's operators/functions cannot be computed as usual
(at least it needs a different justification).

We have constructed bivariate distributions only in so far as needed for describing Green's
functions (cf.\@ Remark~\ref{fn:biv-dist}). It would be very interesting, and highly important for
practical applications in LPDE problems, to generalize the present algebraic approach to the (truly)
\emph{multivariate distributions}. In particular, the LPDE analog of the distributional differential
equation in~\eqref{eq:dist-bvp}, without the boundary conditions, is a crucial tool for the analytic
treatment of LPDE, known as the \emph{fundamental solution}~$\Psi$. For example in the Laplace
equation with~$T = - \Delta = - \der_x^2 -\der_y^2$ one
finds~$\Psi(x,y; \xi,\eta) = -\log \sqrt{(x-\xi^2 + (y-\eta)^2} /2\pi$.

On another note, one may also contemplate \emph{substitution} of functions in distributions from an
algebraic viewpoint (in the multivariate case this would subsume cases such as the diagonal
distribution introduced in Section~\ref{sec:biv-distributions}). Analysis tells us the key relation
$\delta(f(x)) = \delta(x-z)/|f'(z)|$ if~$f$ is suitably regular and has one simple root~$z \in \RR$
within the domain of consideration. However, it is not clear at this point in how far such a
relation can be mapped to an algebraic setting unless one has a suitable algebraic treatment of
composing functions with each other. Not much seems to be available in terms of general settings (as
far as we are aware), apart from some promising new developments
like~\cite[\S3.3]{Robertz2014}. Future work might bring up some interesting new connections.

\begin{acknowledgment}
  This work was supported by the Austrian Science Fund under the FWF Grant No.\@ P30052. We thank
  the anonymous referees for their valuable hints.
\end{acknowledgment}


\end{document}